\date{}
\documentclass [11pt]{article}
\usepackage{amsthm,amsfonts,amsmath,amssymb,graphicx,enumerate,xstring}

\title{Dynamic concentration of the triangle-free process}
\author{Tom Bohman \thanks{Department of Mathematical Sciences, Carnegie Mellon
University, Pittsburgh, PA 15213, USA. Email: {\tt tbohman@math.cmu.edu}.
Research supported in part by NSF grants DMS-1001638 and DMS-1100215.}
\and Peter Keevash \thanks{Mathematical Institute,
University of Oxford, Oxford, UK. Email: keevash@maths.ox.ac.uk.
Research supported in part by
ERC grants 647678 and 239696 
and EPSRC grant EP/G056730/1.} }

\theoremstyle{plain}
\newtheorem{theo}{Theorem}[section] 

\newtheorem{lemma}[theo]{Lemma}

\theoremstyle{definition}
\newtheorem{defn}[theo]{Definition}
\newtheorem{rem}[theo]{Remark}

\newcommand{\mc}[1]{\mathcal{#1}}
\newcommand{\mb}[1]{\mathbb{#1}}
\newcommand{\nib}[1]{\noindent {\bf #1}}

\newcommand{\brac}[1]{\left( #1 \right)}
\newcommand{\bracsq}[1]{\left[ #1 \right]}

\newcommand{\bfl}[1]{\left\lfloor #1 \right\rfloor}

\newcommand{\sub}{\subseteq}
\newcommand{\subn}{\subsetneq}

\newcommand{\wt}{\widetilde}
\newcommand{\es}{\emptyset}

\newcommand{\sm}{\setminus}
\newcommand{\eps}{\varepsilon}

\newcommand{\aA}{\alpha}
\newcommand{\bB}{\beta}
\newcommand{\gG}{\gamma}
\newcommand{\dD}{\delta}
\newcommand{\DD}{\Delta}
\newcommand{\tT}{\vartheta} 
\newcommand{\TT}{\Theta}

\newcommand{\Ss}{\Sigma}

\newcommand{\GG}{\Gamma}
\newcommand{\OO}{\Omega}

\newcommand{\Var}{\text{Var}}

\newcommand{\I}[1]{%
 \IfEqCase{#1}{%
 {0}{I_{\text{ext}}}
 {1}{I_{\text{glo}}}
 {2}{I_{\text{con}}}
 {3}{I_{\text{stk}}}
 }%
}%

\def\COMMENT#1{}

\def\qed{\hfill $\Box$}

\begin{document}
\maketitle

\begin{abstract}
The triangle-free process begins with an empty graph on $n$ vertices and iteratively adds edges chosen uniformly at random subject to the constraint that no triangle is formed.  We determine the asymptotic number of edges in the maximal triangle-free graph at which the triangle-free process terminates. We also bound the independence number of this graph, which gives an improved lower bound on the Ramsey numbers $R(3,t)$:\ we show $R(3,t) > (1/4-o(1))t^2/\log t$, which is within a $4+o(1)$ factor of the best known upper bound. Our improvement on previous analyses of this process exploits the self-correcting nature of key statistics of the process. Furthermore, we determine which bounded size subgraphs are likely to appear in the maximal triangle-free graph produced by the triangle-free process:\ they are precisely those triangle-free graphs with density at most 2.
\end{abstract}

\section{Introduction}

Constrained random graph processes provide both an interesting class of random graph models and a natural source for constructions in graph theory. Although the dependencies introduced by the constraints make such processes difficult to analyse, the evidence to date suggests that they are particularly useful for producing graphs of interest for certain extremal problems.  Here we consider the triangle-free random graph process, which is defined by sequentially adding edges, starting with the empty graph, chosen uniformly at random subject to the constraint that no triangle is formed. Formally, let $G(0)$ be the empty graph on $n$ vertices. At stage $i$ we have a graph $G(i)$; we denote its edge set by $E(i)$, and let $O(i)$ be the set of pairs $xy$ that are {\em open}, in that $G(i) \cup \{xy\}$ has no triangle. We obtain $G(i+1)$ from $G(i)$ by adding a uniformly random pair from $O(i)$.

This process was introduced by Bollob\'as and Erd\H{o}s (see \cite{BR}), and first analysed by Erd\H{o}s, Suen and Winkler \cite{ESW}, using a differential equations method introduced by Ruci\'nski and Wormald \cite{RW} for the analysis of the constrained graph process known as the `d-process'. One motivation for their work was that their analysis of the triangle-free process led to the best lower bound on the Ramsey number $R(3,t)$ known at that time. The Ramsey number $R(s,t)$ is the least number $n$ such that any graph on $n$ vertices contains a complete graph with $s$ vertices or an independent set with $t$ vertices. In general, very little is known about these numbers, even approximately. The upper bound $R(3,t) = O(t^2/\log t)$ was obtained by Ajtai, Koml\'os and Szemer\'edi \cite{AKS}, but for many years the best known lower bound, due to Erd\H{o}s \cite{E}, was $\OO(t^2/\log^2 t)$. The order of magnitude was finally determined by Kim \cite{K}, who showed that $R(3,t) = \OO(t^2/\log t)$. He employed a semi-random construction that is loosely related to the triangle-free process, thus leaving open the question of whether the triangle-free process itself achieves this bound; this was conjectured by Spencer \cite{Sp0} and proved by Bohman \cite{B}. There is now a large literature on the general $H$-free process, obtained by replacing `triangle' by any fixed graph $H$ in the definition; see \cite{BK, BR2, OT, P1, P2, P3, Wa1, Wa2, Wa3, Wz1, Wz2, Wz3}. However, the theory is still very much in its early stages:\ we conjectured that our lower bound for $H$ strictly 2-balanced, given in \cite{BK}, gives the correct order of magnitude for the length of the process, but so far this has only been proved for some special graphs
(cycles \cite{P2,P3,Wa2}, $K_4$ \cite{Wa3} and the diamond \cite{P1}).

In this paper, we specialise to the triangle-free process, where we can now give an asymptotically optimal analysis. Our improvement on previous analyses of this process exploits the self-correcting nature of key statistics of the process.  For a treatment of self-correction in a simpler context see \cite{BFLold}.  The methods that we use to establish self-correction of the triangle-free process build on the ideas used recently by Bohman, Frieze and Lubetzky \cite{BFL} for an analysis of the triangle-removal process. Furthermore, the results of this paper have also been obtained independently and simultaneously by Fiz Pontiveros, Griffiths and Morris \cite{theother}; their proof also exploits self-correction, but is different to ours in some important ways (particularly in the methodologies for establishing self-correction and the analysis of the early part of the process, and also including many subtle differences, such as the definitions of the ensemble of key statistics that can be mutually controlled throughout the process).

Let $G$ be the maximal triangle-free graph at which the triangle-free process terminates.
\begin{theo} \label{edges}
With high probability, every vertex of $G$ has degree $(1+o(1)) \sqrt{\tfrac{1}{2}n\log n}$.
Thus the number of edges in $G$ is $ \left(\tfrac{1}{2\sqrt{2}}+o(1) \right) (\log n)^{1/2} n^{3/2}$ 
with high probability.
\end{theo}
\noindent
We also obtain the following bound on the size of any independent set in $G$.
\begin{theo} \label{indep}
With high probability, $G$ has independence number at most $(1+o(1)) \sqrt{2n\log n}$.
\end{theo}
\noindent
An immediate consequence is the following new lower bound on Ramsey numbers.
The best known upper bound is $R(3,t) < (1+o(1))t^2/\log t$, due to Shearer \cite{Sh}.
\begin{theo} \label{ramsey}
$R(3,t) > \left( \tfrac{1}{4}-o(1) \right) t^2/\log t$.
\end{theo}
\noindent
These results are predicted by a simple heuristic.  The graph $ G(i)$ that we get after $i$ steps of the triangle-free process should closely resemble the Erd\H{o}s-R\'enyi random graph $ G_{n,p} $ with $ i = n^2p/2 $, with the exception that $ G_{n,p} $ should have many triangles while $ G(i)$ has none.  

In addition to Theorems~\ref{edges}~and~\ref{indep} we show that this heuristic extends to all small subgraph counts; in particular, we answer the 
folklore question (brought to our attention by Joel Spencer)
of which subgraphs appear in $G$.
The \emph{density} of a graph $H$ with $V_H \ne \es$  
is $d(H) = \frac{|E_H|}{|V_H|}$.
The \emph{maximum density} $m(H)$ of $H$ is the maximum of $d(H')$
over non-empty subgraphs $H'$ of $H$.
\begin{theo} \label{subgraphs}
Let $H$ be a non-empty triangle-free graph.
\begin{enumerate}[(i)]
\item If $m(H) \le 2$ then $\mb{P}(H \sub G)=1-o(1)$.
\item If $m(H) > 2$ then $\mb{P}(H \sub G)=o(1)$.
\end{enumerate}
\end{theo}
\noindent Thus, the small subgraphs that are likely to appear in $G$ 
are exactly the same as the triangle-free subgraphs that appear in $G_{n,p}$
when $p = \TT(n^{-1/2} \log^{1/2} n)$.

Note that the lower bound on $ R(3,t)$ given by the triangle-free process is non-constructive; for an explicit construction of a triangle-free graph on \( \Theta( t^{3/2}) \) vertices with independence number less than \(t\) see Alon \cite{a}.  Alon, Ben-Shimon and Krivelevich \cite{abk} gave a construction that can be applied to $G$ to produce a {\em regular} Ramsey \( R(3,t) \) graph, at the cost of a worse constant in the lower bound on $ R(3,t)$.

The bulk of this paper is occupied with the analysis required for the lower bound in Theorem \ref{edges}. To prove this, we in fact prove much more generally that we can `track' several ensembles of `extension variables' for most of the process; this is formalised as Theorem \ref{good}. 
The proof of Theorem \ref{good} is outlined in the next section, then implemented over the four following sections. In Section \ref{sec:couple} we present some coupling and union bound estimates that are needed throughout the paper, and also prove Theorem \ref{subgraphs}, assuming Theorem \ref{good}. In Sections \ref{sec:global}, \ref{sec:control} and \ref{sec:stack}, we prove Theorem \ref{good} via a self-correcting analysis of three ensembles of random variables.  
Section~\ref{sec:indep} is mostly occupied by the proof of Theorem \ref{indep}; it also contains the proof of the upper bound in Theorem \ref{edges}, which is similar and easier.  We conclude with some brief remarks in Section~\ref{sec:conclude}.

\section{Overview of lower bound}
\label{sec:overviewlower}

In this section we outline the proof of the lower bound in Theorem \ref{edges}.  We are guided throughout by the heuristic that $ G(i)$ should resemble $ G_{n,p} $ with $ i = n^2 p/2 $.  Before proceeding with the outline of the proof we mention a consequence of this heuristic that is central to the entire argument.  
We introduce a time parameter $t$ that is a rescaling of the number of steps $i$,
defined by \[ t=in^{-3/2}. \] 
For intuition, it is helpful to think of $t$ as a continuous parameter,
as it takes values less than $\sqrt{\log n}$, which is negligible 
compared with the polynomial scalings of the key statistics of the process.

Note that \[p = 2t n^{-1/2}.\]
We define $ Q(i) $ to be the number of open \emph{ordered} pairs in $ G(i) $.  (So $Q(i) = 2|O(i)|$.)  This variable is crucial to our understanding of the process.  We have $ Q(0) = n^2-n$, and the process ends exactly when $ Q(i)=0$.  How do we expect $ Q(i)$ to evolve?  
If $ G(i) $ resembles $ G_{n,p} $ then for any pair $ uv$ we should have
\[ \mathbb{P}( uv \in O(i)) \approx \left( 1 - p^2 \right)^{n-2}
\approx e^{ - n p^2} = e ^{-4t^2}. \]
We set $q(t) =  e ^{-4t^2} n^2$ and
expect to have
\[ Q(i) \approx q(t) \]
for most of the evolution of the process.  This is exactly what we prove.

\subsection{Strategy}
\label{sec:strategy}

We use dynamic concentration inequalities for a carefully chosen ensemble of random variables associated with the process. We aim to show $V(i) \approx v(t)$ for all variables $V$ in the ensemble, for some smooth function $v(t)$, which we refer to as the \emph{scaling} of $V$. Here $V(i)$ denotes the value of $V$ after $i$ steps of the process, and we scale time as
$ t=in^{-3/2}$. 
For each $V$ we define a \emph{tracking variable} $\mc{T}V(i)$ and aim to show that $\mc{D}V(i) = V(i) - \mc{T}V(i)$ satisfies $|\mc{D}V(i)| < \dD_V(t) v(t)$, for some error functions $\dD_V(t)$. We use $\mc{T}V(i)$ rather than $v(t)$ so that we can isolate variations in $V$ from variations in other variables that have an impact on $V$.

The improvement to earlier analysis of the process comes from `self-correction', i.e.\ the mean-reverting properties of the system of variables. We take $\dD_V(t) = f_V(t) + 2g_V(t)$, where we think of $f_V(t)$ as the `main error term' and $g_V(t)$ as the `martingale deviation term'.  We usually have $ g_V \ll f_V$, but there are some exceptions when $t$ is small and hence $f_V(t)$ is too small. We require $g_V(t)v(t)$ to be `approximately non-increasing' in $t$, in that $g_V(t')v(t') = O(g_V(t)v(t))$ for all $t' \ge t$.\footnote{There will be one exceptional type of variable, the vertex degrees, for which this does not hold.}
We define the {\em critical window} 
\[W_V(i) = [ (f_V(t) + g_V(t))v(t), (f_V(t)+ 2g_V(t))v(t)].\]
We aim to prove the {\em trend hypothesis} for $V$,
which is the following statement\footnote{
We will only give the analysis for `upper critical window', 
i.e.\ we consider $\mc{D}V(i)$ positive;
the case of $\mc{D}V(i)$ negative
can be treated in exactly the same way with reversed signs. 
We also remark that we need to `freeze' $\mc{Z}V(i)$ 
if $V$ becomes `bad' (see \eqref{eq:freeze} in Section~\ref{sec:stop}).} 
\begin{equation} \label{eq:trend}
 \mc{Z}V(i) := | \mc{D}V(i) | - \dD_V(t)v(t) 
\text{ is a supermartingale when } |\mc{D}V(i)| \in W_V(i).
\end{equation}
The trend hypothesis will follow from the \emph{variation equation} for $\dD_V(t)$, which balances the changes in $\mc{D}V(i)$ and $\dD_V(t)v(t)$. Since errors can transfer from one variable to another, each variation equation is a differential inequality that can involve many of the error functions.

We aim to track the process up to the time 
\[ t_{max} = \tfrac{1}{2} \sqrt{ (1/2 - \eps) \log n},\]
where $\eps>0$ is a constant, fixed throughout the paper,
that can be arbitrarily small.
More precisely, we will define a stopping time $I$ 
as the first step $i$ at which we have failure of various events
(defined below), which include the event that $V$ satisfies
its required bounds. It will suffice to show that 
$ I > i_{max} := t_{max}n^{3/2}$ 
with high probability.

One way that $ I \le i_{max}$ can occur is when
there exists $ i^* = I \le i_{max} $ and some variable $V$
where $\mc{D}V(i^*)$ is too large.
In this situation, $\mc{D}V$ enters $W_V$ from below at some\footnote{
We will be able to assume a certain lower bound $i'>i_V$ 
via coupling arguments given in Section \ref{sec:couple},
and also that $V$ is `good' (see Section~\ref{sec:stop}).} 
step $i'<i^*$, stays in $W_V(i)$
for $i' \le i < i^*$ then goes above $W_V(i^*)$ at step $i^*$.  
During this time $\mc{Z}V(i)$ is a supermartingale, 
with $\mc{Z}V(i'-1) \le -g_V(t')v(t')$ and $\mc{Z}V(i^*) \ge 0$,
so we have an increase of at least $g_V(t')v(t')$ against the drift of the supermartingale.
Then we use Freedman's martingale inequality \cite{F}, which is as follows. 
\begin{lemma}[Freedman] \label{Freedman}
Suppose $(X(i))_{i \ge 0}$ is a supermartingale with respect to the
filtration $\mc{F}=(\mc{F}_i)_{i \ge 0}$.  Suppose that $X(i+1)-X(i) \le B$
for all $i$ and define $V(j) = \sum_{i=1}^j \Var(X(i) \mid \mc{F}_{i-1})$.
Then for any $a, v >0$ we have
\[ \mb{P}\left( \exists i \text{ such that } X(i) \ge X(0) + a \text{ and } V(i) \le v \right)
\le \exp \brac{-\frac{a^2}{2(v+Ba)}}. \]
\end{lemma}

To apply Freedman's inequality, we let $\mc{F}=(\mc{F}_i)_{i \ge 0}$
be the natural filtration for the triangle-free process, 
in which each $\mc{F}_i$ consists of all events
determined by the choice of the first $i$ edges,
and we estimate 
\[\Var_V(i) :=  \Var(\mc{Z}V(i) \mid \mc{F}_{i-1}) 
\ \ \ \ \text{ and } \ \ \ \ N_V(i) := |\mc{Z}V(i+1)-\mc{Z}V(i)|.\]
Since $g_V(t)v(t)$ is approximately non-increasing 
(unless $V$ is a vertex degree variable), 
to obtain the required estimate $|\mc{D}V(i)| < \dD_V(t) v(t)$
with subpolynomial failure probability,
it suffices to have the following two bounds, which together 
we call the {\em boundedness hypothesis}:
\begin{equation}
\label{eq:boundVar}
g_V(t)^2v(t)^2 = \omega \left( \Var_V(i) (n\log n)^{3/2} \right),
\end{equation}  
\begin{equation}
\label{eq:boundN}
g_V(t)v(t) = \omega \left( N_V(i)\log n \right).
\end{equation}

The lower bound of Theorem~\ref{edges} will follow from Theorem \ref{good} below,
in which we show $I > i_{max}$ with high probability, so every variable in our
ensembles satisfies the required estimate for all $i<i_{max}$;
in particular $Q(i)>0$, so the process persists at least to step $i_{max}$.
The proof of Theorem \ref{good} is by a union bound over a polynomial number
of events, each of which has subpolynomial failure probability
(for brevity, we say these events hold `whp', meaning `with high probability').
We divide these events into four groups, which are treated successively over
the next four sections:\ firstly events not analysed by the critical window method
described in this section, and then critical window events for
three types of variables. The above discussion proves that for each variable $V$
the required critical window event holds whp under the trend and boundedness
hypotheses. For ease of reference we formulate this as a lemma,
in which $I_V$ denotes the first $i \ge i_V$ 
(the `activation step' for $V$, see Definition \ref{def:active})
at which the required estimate on $V$ fails
(we let $I_V=\infty$ if there is no such step).

\begin{lemma} \label{th+bh}
For any variable $V$ and step $i_V \ge 1$,
if $|\mc{D}V(i_V)| < \dD_V(t_V) v(t_V)$
and the trend and boundedness hypotheses  for $V$ hold
for all $i_V \le i <I$ then whp we do not have $I=I_V \le i_{max}$.
\end{lemma}


\subsection{Variables}

All definitions are with respect to the graph $G(i)$
of edges at step $i$ of the triangle-free process.  
Sometimes we use a variable name to also denote the set that it counts, e.g.\ $Q(i)$ is the number of ordered open pairs, and also denotes the set of ordered open pairs.  We usually omit $(i)$ and $(t)$ from our notation, e.g.\ $Q$ means $Q(i)$ and $q$ means $q(t)$.  We use capital letters for variable names and the corresponding lower case letter for the scaling. We express scalings using the (approximate) edge density and open pair density; these are respectively
\[ p = 2in^{-2} = 2tn^{-1/2} \ \ \text{ and } \ \ \hat{q} = e^{-4t^2}.\]

The next most important variable in our analysis, 
after the variable $ Q $ defined above,
is the variable $ Y_{uv}$ which, for a fixed pair of vertices $uv$, 
is the number of vertices $w$ such that $uw$ is an open pair and $vw$ is an edge.  
It is natural that $ Y_{uv} $ should play an important role in this analysis,
as it directly controls the evolution in the number $Q(i)$ of ordered open pairs:\ 
if $uv$ is the edge selected at step $i+1$ then
\[ Q(i) - Q(i+1) = 2(1+Y_{uv}+Y_{vu}).\]
Similarly, we have the following expression,
used throughout the paper, for the probability
(conditional on the history of the process up to step $i$)
that any particular open pair in $Q(i)$
is not open in $Q(i+1)$:
\begin{equation} \label{eq:close}
\mb{P}(uv \notin Q(i+1) \mid \mc{F}_i, uv \in Q(i))
= 2(1+Y_{uv}+Y_{vu})/Q.
\end{equation}
From the heuristics (which we will prove)
$Y \approx y = 2t\hat{q}n^{1/2}$ and $Q \approx q=\hat{q}n^2$
we can approximate edge-closure probabilities by
\begin{equation} \label{eq:close2}
4y(t)/q(t) = 8tn^{-3/2} = - n^{-3/2} \hat{q}'(t)/\hat{q}(t),
\end{equation}
which agrees with the intuition provided 
by the mean value approximation
\[ \hat{q}(t) - \hat{q}(t+n^{-3/2}) 
\approx - \hat{q}'(t) n^{-3/2}. \]

To control these variables we need to embed them in some larger ensembles 
of variables that mutually control each other.
The motivation for introducing each of 
the ensembles defined below is as follows:\ 
control of the global variables is needed
to get good control of $ Q$ 
(better than that implied by control of all $Y_{uv}$), 
control of the stacking variables is needed to
get good control of $ Y_{uv}$, and controllable variables play a
crucial role in our analysis of the stacking variables.  

\subsubsection{Global variables}
We begin with the variable that we are most interested in understanding:\ 
the number of open pairs.
We also include two other variables that will allow us 
to maintain precise control on the number of open pairs.
\begin{itemize}
\item
$Q=2|O(i)|$ is the number of ordered open pairs.
The scaling is $q=\hat{q}n^2$.
\item
$R$ is the number of ordered triples with $3$ open pairs.
The scaling is $r = \hat{q}^3 n^3$.
\item
$S$ is the number of ordered triples $abc$ where $ab$ is an edge and $ac$, $bc$ are open pairs.
The scaling is $s = p\hat{q}^2 n^3 = 2t\hat{q}^2n^{5/2}$.
\end{itemize}
We refer to $Q$, $R$ and $S$ as {\em global} variables.

\subsubsection{Controllable variables} \label{subsub:defcon}

Next we formulate a very general condition under 
which we can approximate a variable up to a
proportional error with polynomial decay.
Suppose $\GG$ is a graph, $J$ is a spanning subgraph of $\GG$ and $A \sub V_\GG$. 
We refer to $(A,J,\GG)$ as an \emph{extension}. 
Suppose that $\phi:A \to [n]$ is an injective mapping.
We define the {\em extension variables} $X_{\phi,J,\GG}(i)$ to be the number of injective maps $f:V_\GG \to [n]$ such that
\begin{enumerate}[(i)]
\item $f$ restricts to $\phi$ on $A$,
\item $f(e) \in E(i)$ for every $e \in E_J$ not contained in $A$, and
\item $f(e) \in O(i)$ for every $e \in E_\GG \sm E_J$ not contained in $A$.
\end{enumerate}
We call $(J,\GG)$ the \emph{underlying graph pair} of $X_{\phi,J,\GG}$.  
We introduce the abbreviations $V=X_{\phi,J,\GG}$,
\[n(V) = |V_\GG| - |A|, \quad e(V) = e_J - e_{J[A]},
\quad \text{ and } \quad o(V) = (e_\GG-e_J) - (e_{\GG[A]}-e_{J[A]}),\]
which are respectively the numbers of vertices, edges and 
open pairs\footnote{We hope that this will not be confusable
with our use of the `little-o' notation $o(1) \to 0$ as $n \to \infty$.}
not contained in the base of the extension.
The \emph{scaling} of $V$ is a deterministic function of
the time $t$ defined by
\[ v = x_{A,J,\GG} = n^{n(V)} p^{e(V)} \hat{q}^{o(V)},\]
i.e.\ it predicts the 
evolution of $V$ according to the heuristic that each of the 
$\sim n^{n(V)}$ injections $f:V_\GG \to [n]$ satisfying (i)
should independently satisfy (ii) for each $e \in E_J \sm E_{J[A]}$
with probability $p$ and (iii) for each $e \in E_\GG \sm E_{\GG[A]}$
with probability $\hat{q}$. This prediction is correct only if
there is no subextension that is `dense', 
in that it has scaling much smaller than $1$. 

When considering such subextensions $(B,J[B'],\GG[B'])$ 
with $A \sub B \sub B' \sub V_\GG$,
we denote the scaling by\footnote{
The letter `S' is used for scalings and stacking variables, 
but we hope that this will not lead to any confusion, 
as the use is determined by the form of the superscript.} 
\[S^{B'}_B = S^{B'}_B(J,\GG) = n^{|B'| - |B|} p^{e_{J[B']} - e_{J[B]}} \hat{q}^{(e_{\GG[B']} - e_{J[B']}) - (e_{\GG[B]} - e_{J[B]})}.\]
For example, $S^{V_\GG}_A = v$. Note that 
if $A \sub B \sub B' \sub B'' \sub V_\GG$ 
then $S^{B''}_B = S^{B''}_{B'} S^{B'}_B$.

Let $ t' \ge 1$. We say that $V$ is \emph{controllable at time $t'$} 
if $o(V)>0$ (i.e.\ at least one pair not contained in the base is open) 
and for $1 \le t \le t'$ and $A \subn B \sub V_\GG$ we have
\begin{equation} \label{eq:dD'}
S^B_A(J,\GG) \ge n^{\dD'},
\end{equation}
where $\dD'>0$ is a fixed global parameter
much smaller than $\eps$ (see \eqref{eq:epsilondelta} below
for the parameter hierarchy).

We say that $V$ is \emph{controllable} if it is controllable at time $1$.  
The {\em controllable ensemble} is the collection of controllable variables 
$ X_{\phi,J,\GG} $ such that $| V_\GG | \le M^3$,
where $M=3/\eps$ (see \eqref{eq:M} below).

\begin{rem}
The proof that we can track the controllable variables 
(up to the precision needed for our purposes) is relatively short.  
In a certain sense, our results on controllable variables can 
be viewed as a triangle-free process analogue of the concentration on subgraph
extensions in $G_{n,p}$ that follows from Kim-Vu polynomial concentration 
(see Lemma \ref{kim-vu} below).  
A similar analogue should hold for the triangle removal process, 
and the introduction of this idea would simplify the analysis 
of the triangle removal process 
recently given by Bohman, Frieze and Lubetzky \cite{BFL}.
\end{rem}

\subsubsection{Stacking variables}

In order to understand the evolution 
of the global variables $Q$, $R$ and $S$, 
we now introduce an ensemble of {\em stacking variables}.
The name of this ensemble indicates that its members 
are obtained by stacking basic building blocks,
each of which is a one-vertex extension.
We start with two such extensions
which are defined for every ordered pair $uv$.
We have already met the first, $Y_{uv}$, in our above discussion 
of the evolution of $Q$; the second, $X_{uv}$, is clearly required
for understanding the evolution of $Y_{uv}$, 
as if $w$ contributes to $X_{uv}$ and we select the edge $vw$
then $w$ will instead contribute to $Y_{uv}$. 
\begin{itemize}
\item
$Y_{uv}$ is the number of vertices $w$ such that $uw$ is an open pair and $vw$ is an edge.  The scaling is $y = 2t\hat{q}n^{1/2}$.
\item
$X_{uv}$ is the number of vertices $w$ such that $uw$ and $vw$ are open pairs. The scaling is $x = \hat{q}^2n$.
\end{itemize}
The other two building blocks are one-vertex
`degree' extensions defined for every vertex $u$.
\begin{itemize}
\item $X_u$ is the \emph{open degree} of $u$, defined as the number of vertices $ w $ such that $ u w $ is open.
The scaling is $x_1 = n \hat{q}$.
\item $Y_u$ is the \emph{degree} of $u$, defined as the number of vertices $ w $ such that $ u w $ is
an edge.  The scaling is $ y_1 = 2t n^{1/2} $.
\end{itemize}

We will define stacking variables by composing certain sequences
of such one-vertex extensions. We start by setting up notation for describing an arbitrary such variable, although we will only track a subset of the collection 
of the stacking variables, the $M$-bounded stacking variables, 
which will be defined later in the section. 

\begin{defn} \label{def:stack1}
We define\footnote{
Each symbol represents a certain extension (as described below).
We include condition (i) so that the definition makes sense
and (ii), (iii) so as to reduce the number of cases
in the analysis of stacking variables.}
a {\em symbol set} $\Ss = \{O,E,Y^O,X^O,Y^I,X^I\}$
and let $\mc{S}$ be the set of all 
non-empty finite sequences $\pi$ in $\Ss$
(i.e.\ $\pi \in \cup_{m \ge 1} \Ss^m$) such that
\begin{enumerate}[(i)] \label{en1}
\item if $E$ occurs then it only does so as the last symbol of $\pi$,
\item $\pi(1) \notin \{Y^I,X^I\}$,
\item there is no $i$ with $\pi(i)=O$ and $\pi(i+1) \in \{Y^I,X^I\}$,
except possibly in the last two positions.
\end{enumerate}
\end{defn}

For any $ \pi \in \mc{S}$ and pair of vertices $uv$
(we will only consider $uv \not\in E(i)$) we define $S^\pi_{uv}$ 
according to the following rules. At each step
there is an {\em active rung} (initially $uv$)
and a {\em last vertex} (initially $v$).
Suppose we have constructed $i-1$ steps of our stacking variable 
and that we have an active rung $xy$ with last vertex $y$.  
If $ \pi(i) = O $ (`open') then the next step is an $X_y$ extension, 
the single open pair in this extension is the new active rung,
and the new vertex is the new last vertex.
If $ \pi(i) = E $ (`edge') then the next step is an $Y_y$ extension 
and then there is no active rung:\ the variable terminates here.

Now suppose $ \pi(i) \notin \{O,E\}$; that is, suppose $ \pi(i) $ 
indicates an $X$ or $Y$ extension on the active rung. 
The superscript indicates the \emph{direction} of this extension. 
For $Y$ it determines whether we add $Y_{xy}$ or $Y_{yx}$, 
and the new open pair becomes the active rung. 
For $X$ it determines which of the two new open pairs 
becomes the active rung.  In both cases, 
a superscript of O (for `outer') indicates that the new active rung 
is incident with the last vertex, $y$, 
while a superscript of I (for `inner') indicates that 
the next active rung is not incident with $y$ 
(i.e.\ it is incident with $x$). 

We think of $S_{uv}^\pi$ as counting injections
$\psi$ from $V(S_{uv}^\pi) := \{ \aA_u, \aA_v, \aA_1, \dots, \aA_{ |\pi|} \}$
to $[n] $ such that $ \psi( \aA_u ) = u $, $ \psi( \aA_v) = v $ and each
$ \psi (\aA_j)$ is a vertex that plays the role in the extension 
defined by $\pi(j)$ for $j =1, \dots, | \pi | $,
i.e.\ $S_{uv}^\pi = X_{\phi,J,\GG}$ is the extension variable
with $V(\GG)=V(S_{uv}^\pi)$, $A=\{\aA_u,\aA_v\}$, $\phi(\aA_u)=u$,
$\phi(\aA_v)=v$ and $(J,\GG)$ is defined so that edges specified 
by the extension are mapped to edges of $G(i)$,
and likewise for open pairs.

The above distinction between `inner' and `outer'
is crucial for understanding what kind of proportional accuracy one should expect in controlling these variables. For an intuitive explanation of this phenomenon, and to clarify the meaning of the definition, we introduce a pictorial representation of stacking variables, in which we think of the vertices of the active rung as the locations of the feet of someone walking on the graph. An outer extension corresponds to moving the other foot to that moved in the previous step, whereas an inner extension corresponds to moving the same foot (the intuition in the latter case is that the variable then `sees less' of the graph and so suffers a less accurate approximation).

\begin{figure}
\begin{center}
\includegraphics[scale=0.9]{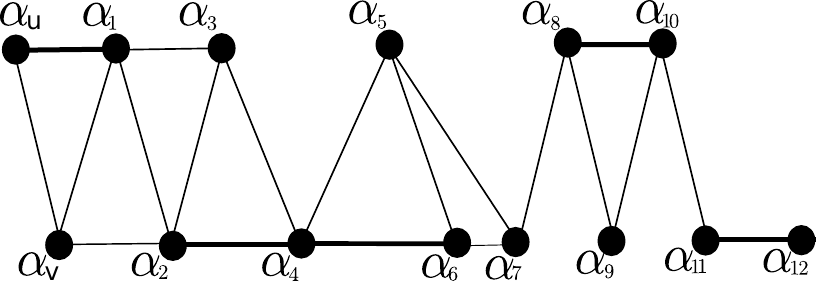}
\caption{The stacking variable $S^\pi_{uv}$ corresponding to 
$\pi = Y^O X^O X^O Y^O O Y^O X^I O O Y^O O E$. 
Thick lines represent edges and thin lines represent open pairs.  
Open pairs with one vertex in each row of the diagram are rungs. 
There are 3 triangular ladders, namely  
$ \pi[-1;4] = \pi[v;4]$, $ \pi[4;7]$ and $\pi[8;10]$,
which respectively have sets of turning points
$\{\aA_1,\aA_2,\aA_3\}$, $\{\aA_5\}$ and $\{\aA_9\}$.}
\end{center}
\end{figure}

In our pictorial representation (see Figure 1), 
we visualise $\pi$ as a horizontal strip 
of two rows (`top' and `bottom'), 
with vertex labels arranged sequentially from left to right 
according to the corresponding order in $\pi$. 
We start by assigning $\aA_u$ to the top and $\aA_v$ to the bottom.
In each step we assign the new vertex so that any pair of vertices meets 
both rows if and only if it is a rung (this uniquely defines the assignment).
The direction superscripts indicate whether the new vertex
is added to the same (I) or different (O) row to the last vertex.
Conversely, any such drawing determines a unique order $\aA_1,\dots,\aA_t$ of vertices, which we call the \emph{stacking order}, 
from which we can reconstruct $\pi$. 

We note that the vertex set of any rung is a cutset 
of the graph $\GG$ associated with $S_{uv}^\pi$.

The simplest stacking variables are those of length $1$, 
namely the building blocks $S^{X^O}_{uv} = X_{uv}$, $S^{Y^O}_{uv} = Y_{vu}$, 
$S^O_{uv} = X_v$ and $S^E_{uv} = Y_v$. The last two examples illustrate the general phenomenon that when $\pi(1) \in \{O,E\}$ we obtain an extension based at the single vertex $v$, which does not depend on $u$. While we could denote this variable more simply by $S^\pi_v$, it is convenient to have a unified notation for stacking variables that allows the effective base of the extension to have one or two vertices.

We also introduce some further terminology which is suggested by the 
faint resemblance between our drawings of stacking variables and ladders.
A {\em triangular ladder} $\pi[x;y]$ of $\pi$ is a portion of $V(S_{uv}^\pi)$
cut off by a subsequence  $x-2, \dots, y$ of consecutive positions in $\pi$ 
where $x, \dots, y$ is a maximal subsequence such that $\pi(j) \notin \{O,E\}$
for all $x \le j \le y$. (In this definition, we adopt the convention
$ \aA_u = \aA_{-1} $ and $\aA_v = \aA_0 $ so as to allow $x \in \{1,2\}$.)
If $ x <i < y$ we say that $\aA_i$ is a \emph{turning point}
if the superscript of $ \pi(i+1) $ is $O$. 
Note that if $ \aA_i$ is a turning point 
then it is in at least two rungs. 
The open pairs containing $\aA_i$ are $\aA_{i^-}\aA_i$ 
and $\aA_j\aA_i$ for all $i+1 \le j \le i^+$, 
where $i^-$ is the previous turning point
(or $x$ if there is none)
and $i^+$ is the next turning point
(or $y$ if there is none).
If $\aA_i$ is in the top row (for example) 
then $\aA_{i^-}$ and $\aA_j$ for $i+1 \le j \le i^+$ 
are consecutive along the bottom row.
We note that any stacking variable is a concatenation of some number of triangular ladders and paths of open pairs, possibly ending with a pendant edge.

We refer to an edge or open pair that is a not a rung as a {\em stringer}.

We do not track all of the stacking variables defined above;
instead, we will track a certain finite family 
(with size bounded as a function of $\eps$).
The precise definition of this family is quite subtle,
as we need to take account of both size and direction
in order to obtain an ensemble that can be controlled
mutually with the other ensembles of variables.
We will impose a bound on the length of any consecutive 
subsequence consisting only of symbols with superscript $I$
(which corresponds to the walker keeping one foot fixed).
We will also bound the {\em weight} of $\pi \in \mc{S}$,
defined by $w(\pi) = w_1(\pi) + w_2(\pi)$, where
\begin{equation} \label{eq:weight}
w_1(\pi) = |\{ i \in |\pi|: \pi(i) \in \{O,E\} \}| 
\ \text{ and } \
w_2(\pi) = |\{ i \in |\pi|: \pi(i) \in \{X^O,Y^O\} \}|.
\end{equation} 
Now we define the $M$-bounded stacking variables
that constitute our stacking ensemble.

\begin{defn} \label{def:stack2}
We say that a stacking sequence $\pi \in {\mathcal S}$
(see Definition \ref{def:stack1}) is \emph{$M$-bounded} if\footnote{
The precise form of this definition will be crucial in
Sections \ref{subsub:outer} (outer destruction)
and \ref{subsub:fan} (fan end destruction).}
\begin{enumerate}[(i)]
\item $w( \pi) \le 2M$, and $w(\pi')<2M$, where $\pi'$ 
is obtained from $\pi$ by deleting $\pi(|\pi|)$,
\item $\pi$ does not contain any consecutive subsequence
of length $M$ using only $\{X^I,Y^I\}$.
\end{enumerate}

\vspace{-0.2cm}

We let $ \mc{S}_M$ be the set of $M$-bounded stacking sequences. 
  
The {\em stacking ensemble} is the collection of all variables
of the form $S_{uv}^\pi $ where $ \pi \in \mc{S}_M$.
\end{defn}

We conclude this section with a simple observation 
on $M$-bounded stacking sequences.
\begin{lemma}
\label{note:Mlength}
If $ \pi \in \mc{S}_M$ is an $M$-bounded stacking sequence 
then the length of $\pi$ is $|\pi| < 2M^2$.
\end{lemma}
\begin{proof}
Let $ w_3(\pi)$ be the number of maximal consecutive subsequences 
of $\pi$ using only $\{ Y^I, X^I \}$. 
By Definition \ref{def:stack2}.i 
we have $ w_3(\pi) \le 2M-1$, as any two such sequences 
are separated by positions that contribute to $ w(\pi)$.  
Furthermore, by Definition \ref{def:stack2}.ii each such subsequence 
of has length at most $ M-1$. Therefore
$|\pi| \le w(\pi) + (M-1)w_3(\pi) \le 2M + (2M-1)(M-1)   < 2M^2$.  
\end{proof}

\subsection{Tracking variables} \label{deftrack}

Recall that each variable $V$ has a tracking variable $ \mc{T}V$ 
and we track the difference $ \mc{D}V = V - \mc{T}V$, so as 
to isolate variations in $V$ from other variations in $G(i)$.

The tracking variables are defined as follows. 
For the global variables we take
\[ \mc{T}Q = q,
\ \ \ \
\mc{T}R = n^3 \cdot( Q/n^2 )^3 = Q^3 n^{-3},
\ \ \ \
\mc{T}S = n^3 \cdot 2tn^{-1/2} \cdot ( Q/ n^2 )^2 = 2t n^{-3/2} Q^2. \]
Note that $\mc{T}R$ and $ \mc{T}S$ are chosen so that $ \mc{D}R$ and $ \mc{D}S$ isolate the variations in $R$ and $S$ that do
not naturally follow from the variation in $Q$.

If $V$ is a one-vertex extension with $a$ edges and $b$ open pairs 
not within its base we take
\[ \mc{T} V = n \cdot ( 2t n^{-1/2})^a \cdot( Q/n^2)^b. \]
That is, we set  $\mc{T}X_{uv} = Q^2n^{-3}$, $\mc{T}Y_{uv} = 2tn^{-3/2} Q$, $\mc{T}X_u = Qn^{-1}$, and $ \mc{T} Y_u =pn = 2tn^{1/2}$.

For the stacking variable $ S_{uv}^ \pi$ with $ |\pi| \ge 2 $ 
we have two cases,\footnote{
Section \ref{sub:stktrack} includes more discussion and motivation 
of the definition of $\mc{T} S_{uv}^\pi$.} 
depending on the form of $ \pi$. 
The first case is that
$ \pi( |\pi|-1) \neq O $ or $ \pi( |\pi|) \in \{O,E\}$. 
We write $ \pi = \pi^- \circ U $, 
where $U$ is the last element of $\pi$, 
and let \[ \mc{T} S_{uv}^\pi = S_{uv}^{\pi^-} \mc{T}U. \]
Note that this choice of $ \mc{T}S_{uv}^\pi$ isolates variations 
that are not caused by variations in $ S_{uv}^{\pi^-} $.

The second case is that $ \pi( |\pi|-1) = O $ and 
$ \pi( |\pi|) \notin \{O,E\}$ (we must have $|\pi| \ge 2$).
We write $ \pi = \pi^- O U $, where $U$ is the last element of $ \pi$, 
let $ \beta = \aA_{|\pi|-2} $ and
\[
\mc{T} S_{uv}^\pi =
\begin{cases}
\sum_{ f \in S_{uv}^{\pi^-}} X_{ f( \beta) }^2 \cdot Qn^{-2} & \text{ if } U \in \{ X^I, X^O \} \\
\sum_{ f \in S_{uv}^{\pi^-}} X_{ f( \beta) }^2 \cdot 2tn^{-1/2} & \text{ if } U = Y^I \\
\sum_{ f \in S_{uv}^{\pi^-}} X_{ f( \beta) } Y_{ f( \beta)} \cdot Qn^{-2} & \text{ if } U = Y^O,
\end{cases}
\]
recalling that
$X_a$ denotes the open degree of vertex $a$ and 
$Y_b$ denotes the degree of vertex $b$.

For a controllable variable $V$ we will only obtain fairly weak
approximations, so the precise definition of the tracking variable
is not very important; it is convenient for the calculations to isolate
the variation due to $Q$, so we let
\[ \mc{T}V = n^{n(V)} p^{e(V)} (Q n^{-2})^{ o(V) } .\]

\subsection{Error functions and activation times} \label{sec:error}

With the definitions of our variables in hand, 
we will now introduce some further notation 
and define the error functions $\dD_V$
(recall that we aim to show
$V = \mc{T}V \pm v \dD_V $ for each variable $V$
in each of the three ensembles).
Throughout the paper we fix parameters 
according to the hierarchy
\begin{equation}
\label{eq:epsilondelta}
n^{-1} \ll \dD \ll \dD' \ll \eps; 
\end{equation}
the roles of these parameters may be understood
by reference to \eqref{eq:tmax} and \eqref{eq:M}
for $\eps$, to \eqref{eq:dD'} for $\dD'$,
and to Definition \ref{def:e} for $\dD$.
Our asymptotic notation is respect to $n$,
e.g.\ $o(1)$ denotes a quantity that can be made
arbitrarily small for $n$ sufficiently large.
We track the process until the time $t_{max}$ 
at which $\hat{q}(t_{max})=n^{-1/2+\eps}$; thus
\begin{equation} \label{eq:tmax}
t_{max} = \tfrac{1}{2} \sqrt{ (1/2 - \eps) \log n}.
\end{equation} 
The constant $M$ that bounds the size of the stacking 
and controllable variables depends on $t_{max}$ through $\eps$:
we let 
\begin{equation} \label{eq:M}
M = 3/\eps.
\end{equation}
We will now define the error functions $\dD_V$.  

\begin{defn} \label{def:e}
Write\footnote{
We hope that $e$ will not be confused with the base of natural logarithms;
the exponential function is denoted by $\exp$ throughout the paper.
}
\[e(t) = \hat{q}(t)^{-1/2}n^{-1/4} \quad \text{ and } \quad  L = \sqrt{\log n}. \]
Our error functions take the form $\dD_V = f_V + 2g_V$, where\footnote{
We defer the definitions of $c_V$ and $\tT$ to Definition \ref{def:c}.}
\[f_V(t) = c_V \phi_V(t), \ \ 
g_V(t) = c_V \phi_V(t) \cdot  \tT(t) L^{-1} (1+t^{-e(V)})
\ \text{ and } \]
\[ \phi_V = \begin{cases} 
e  \ \ \text{ if } V \text{ is a stacking variable}, \\
e^2 \ \ \text{ if } V \text{ is a global variable}, \\
e^\dD \ \ \text{ if } V \text{ is a controllable variable}. 
\end{cases} \] 
\end{defn}

The behaviour of the error functions in Definition \ref{def:e}
is mainly determined by the functions $\phi_V$,
and can be understood without reference to the deferred definitions
of $c_V$ and $\tT$, as the $c_V$ are `constants'
(i.e.\ independent of time; they are polylogarithmic in $n$)
and the function $\tT(t)$ is bounded by 
constants (depending on $\eps$, but not on $n$).
We introduce $\tT$ and the $t^{-e(V)}$ term in $g_V(t)$
to handle some technicalities that arise for $t=o(1)$
(which is not the most significant regime of the process,
but nevertheless exhibits slightly different behaviour from
the later regime, so our proof must account for this difference).
When $t = \Omega(1)$ we have $ g_V = O( L^{-1} f_V) = o(f_V)$,
whereas if $t=o(1)$ with sublogarithmic decay and $e(V)>0$
then we have $ g_V \gg f_V$. The point of the $t^{-e(V)}$ term
is that the dominant term in $vg_V$ as $t \to 0$ 
does not contain a power of $t$. 

The intuition for taking $\phi_V=e$ for stacking variables
is that they include the variables $Y_{uv}$, 
which have scaling $y=2t\hat{q}n^{1/2}=2te^{-2}$,
and which one cannot expect to control to proportional 
error better than $y^{-1/2}$. 
Thus $e$ is a natural reference point
for discussing approximations.  
We note for future reference that 
\begin{equation} \label{eq:e}
e \ \text{ increases from } \
e(0)=n^{-1/4} \ \text{ to } \
e(t_{max}) = n^{-\eps/2},
\end{equation}
so $e$ always has sublogarithmic decay in $n$.
The notation $L=\sqrt{\log n}$ will be
convenient as we always have $t \le t_{max} < L$.
We also note for future reference that the density $\hat{q}$ 
of open pairs is always much large than the density $p$ of edges:
we have
\begin{equation} \label{eq:q/p}
\hat{q}/p = e^{-2}/2t > n^\eps/2L > n^{\eps/2}.
\end{equation}
We take $\phi_V=e^2$ for the global variables so that 
for these variables we can neglect `product' error terms 
arising from applications of Lemma \ref{lem:pat} below.
This is well within the theoretical limit on the accuracy
for $Q$, namely $q^{-1/2} = e^{-1} n^{-3/4} \ll e^{-2}$;
the `extra room' will be helpful 
in the coupling arguments in Section \ref{sec:couple} 
for establishing the required estimates for small $t$.
For the controllable variables we only require accuracy
that decays sublogarithmically, so we take $\phi_V=e^\dD$,
where for $\dD$ we recall 
the parameter hierarchy \eqref{eq:epsilondelta}.

The constants $ c_V$ that appear in Definition \ref{def:e}
will be chosen in Definition \ref{def:c}  to establish
the trend hypotheses 
(i.e.\ to show that each $ \mc{Z} V$ is a supermartingale).  
We will see that approximation errors migrate in a complex fashion
between the variables and so these choices are quite delicate.  
As we treat each ensemble of variables in turn during the next three 
sections we will derive inequalities that these constants must 
(and do) satisfy in order for the trend hypothesis to hold:
see the `variation equations'
\eqref{eq:cR}, (\ref{eq:varQ}),  (\ref{eq:varR}),  
(\ref{eq:varS}), \eqref{eq:varcon},
(\ref{eq:simplecond}), (\ref{eq:fanendcond1}) and (\ref{eq:fanendcond2}).

We think of the $c_V$'s as `constant' as they do not depend on time
(they are all polylogarithmic in $n$).
We specify them now in advance of the analysis, but we will 
keep the notation general so that it is clear how to choose the constants.
We also define the function $\tT(t)$ used above.
Note that the constants for the stacking variables are chosen very carefully,
so that they {\em decrease} as the length of $\pi$ increases
(corresponding to more accurate approximations for longer extensions),
which will be important in Section \ref{subsub:simple} (simple destruction),
and there is a more substantial decrease for each 
occurrence of $O$ or $E$ (counted by $w_1(\pi)$),
which will be important in Section \ref{subsub:fan} (fan end destruction).
There is also an adjustment for the case $\pi=O$,
as our argument for controlling degree extensions requires 
a slightly smaller constant for open degree extensions 
(see Section \ref{sec:vxdeg}). 

\begin{defn} \label{def:c}
For all controllable variables we take $c_V = 1$.
For the global variables we take
\[ c_R = L^{40}, \quad c_S = 2 L^{40}, \quad c_Q = 4 L^{40}. \]
For a stacking variable $V = S^\pi_{uv} $,
recalling $w_1(\pi)$ from \eqref{eq:weight}, we set
\[  c_V = c_\pi = L^{15} 9^{ 4M^2 - |\pi| - M w_1(\pi)} (2.2)^{-1_{\pi=O}}. \]
Let $K = M^6 = (3/\eps)^6$ and  $\tT: [0,\infty] \to [1,\infty]$ 
be any increasing smooth function such that
\[ \tT(t) = e^{Kt} \ \text{ for } \ 0 \le t \le 1,  
\ \ \sup_{t \ge 0} |\tT(t)| \le 2e^K 
\ \text{ and } \
\sup_{t \ge 0} (|\tT'(t)|+|\tT''(t)|) < \infty.\]
\end{defn}

Recalling from Definition \ref{def:stack2} and Lemma \ref{note:Mlength}
that $w(\pi) \le 2M$ and $|\pi| < 2M^2$, we see 
that $L^{15} \le c_V \le L^{15} 9^{4M^2}$
for any $V = S^\pi_{uv} $.

Next we define the
`activation step' $i_V$ at which we start tracking a variable $V$ 
using the martingale arguments in Section \ref{sec:strategy}
(before then we will use the coupling arguments of Section \ref{sec:couple}).
Our definition is uniform across all $V$ bar one technical exception
in which the activation step is slightly later than one might expect.

\begin{defn} \label{def:active}
For any variable $V$, the {\em activation step} $i_V$ is the 
smallest $i \ge n^{5/4}$ for which $g_V(t) \le L^{-1}$,
except that if $V$ is a stacking variable with $e(V)=1$
we let $i_V = n^{1.26}$. 

The {\em activation time} is  $t_V = i_V n^{-3/2}$. 
\end{defn}

In the following lemma we give some estimates for the activation steps
of various variables; we also show that all error functions are $o(1)$
after activation, and justify our earlier informal assertion
that the functions $vg_V$ are approximately non-increasing
(unless $V$ is a vertex degree variable).
The notation $\wt{\TT}$ denotes approximation 
up to a factor polylogarithmic in $n$.

\begin{lemma} \label{lem:tV}
Let $V$ be any variable in any ensemble with $o(V)>0$
(i.e.\ not a vertex degree).
\begin{enumerate}[(i)]
\item If $e(V)=0$ or $V=S$ then $t_V=n^{-1/4}$.
\item If $V$ is a stacking variable with $e(V)>1$
then $t_V = \wt{\TT}(n^{-1/4e(V)})$.
\item If $V$ is a controllable variable with $e(V)>0$
then $t_V = \wt{\TT}(n^{-\dD/4e(V)})$.
\item $\dD_V = o(1)$ for all $t \ge t_V$.
\item $v(t)g_V(t) = O( v(t') g_V(t'))$ whenever $t \ge t'$.
\end{enumerate}
\end{lemma}

\begin{proof}
For (i), we first note that if $e(V)=0$
then $g_V=\wt{\TT}(\phi_V)$.
We have $\phi_V \le e^\dD < n^{-\eps\dD/2} \ll L^{-1}$ 
by \eqref{eq:e}, so by definition
$i_V = n^{5/4}$, i.e.\ $t_V=n^{-1/4}$.
Also, $g_S(t) = \wt{\TT}(e^2)(1+t^{-1}) = \wt{\TT}(n^{-1/2} t^{-1})$
for $t \le 1$, so $g_S(n^{-1/4}) = \wt{\TT}(n^{-1/4}) \ll L^{-1}$,
giving $t_S = n^{-1/4}$, as required.

For (ii), we have $g_V(t) = \wt{\TT}(e)(1+t^{-e(V)})
= \wt{\TT}(n^{-1/4}t^{-e(V)})$ for $t \le 1$, which
hits $L^{-1}$ at some $t_V = \wt{\TT}(n^{-1/4e(V)})$;
we obtain (iii) similarly from $g_V(t) = \wt{\TT}(e^\dD)(1+t^{-e(V)})$.

For (iv), we note that
$f_V(t_V) = O(Lg_V(t_V))(1+t_V^{-e(V)})^{-1}$.
If $e(V)=0$ then $f_V(t_V) = O(L e^\dD ) = o(1)$.
Otherwise, as $g_V(t_V) \le L^{-1}$ by definition of $t_V$,
(i--iii) give $f_V(t_V) = O(Lg_V(t_V))(1+t_V^{-e(V)})^{-1}
= O(t_V^{e(V)}) = \wt{O}(n^{-\dD/4}) = o(1)$.
The estimate for $t \ge t_V$ follows 
as $f_V(t)$ and $g_V(t)/\tT(t)$ are decreasing in $t$,
and $\tT(t)$ is bounded by $2e^K = O(1)$ by Definition \ref{def:c}.

Finally, to see (v) we write $h(t) = v(t)g_V(t) 
= \TT(n^{n(V)} \hat{q}^{o(V)} L^{-1}c_V \phi_V (1+t^{e(V)}))$,
then note that $h(t) = \TT(h(0))$ for $t=O(1)$,
and as $o(V)>0$ there is some $t_0=O(1)$ 
such that $h'(t)<0$ for $t>t_0$.
\end{proof}

\subsection{Stopping times and the main technical result} \label{sec:stop}

In this section we formulate our main result regarding 
the stopping time $I$ (mentioned above) that
provides the lower bound in Theorem \ref{edges}.
For convenience in breaking up the proof into sections,
we define 
\[ I  = \min\{\I0,\I1,\I2,\I3\}  \] 
in terms of $4$ other stopping times defined below,
which are in turn analysed over the next $4$ sections.
Each of these stopping times is defined as the first step
at which certain good events fail 
(or $\infty$ if there is no such step).
The stopping time $\I0$ controls 
various events that we think of as `external' 
to the main martingale strategy of 
critical window events in Section \ref{sec:strategy}.
The other stopping times control critial window events
for each of the three ensembles:
$\I1$ controls global variables,
$\I2$ controls controllable variables and
$\I3$ controls stacking variables.
We start by defining these critical window stopping times
in terms of stopping times $J_V$ and $I_V$ associated
to each variable $V$ as follows.

\begin{defn} \label{def:IV}
Consider any variable $V$ in any ensembles, 
and write $V=X_{\phi,J,\GG}$ (see Section~\ref{subsub:defcon})
for some extension $(A,J,\GG)$.
 
We say that $V$ is {\em bad} (at step $i$) 
if there is an edge $e=\phi(x)\phi(y)$ 
of $ G(i) $ with $x,y$ in $A$ 
and some $w \in V_\GG \sm A$ such that
$\GG$ contains $xw$ and $yw$, 
and at least one of them is in $J$.

If $V$ is not bad we call it {\em good}.

We let $J_V$ be the smallest\footnote{
See Definition \ref{def:active} for $i_V$ (the `activation step').}
$i \ge i_V$ such that $V$ is bad
(or $\infty$ if there is no such time).

We let the stopping time $I_V$ be the smallest $i$ with $i_V < i < J_V$ 
such that $|\mc{D}V(i)| > \dD_V(t) v(t)$ 
(or $\infty$ if there is no such time).

We let $\I1$, $\I2$ and $\I3$ be the respective minima of $I_V$ 
over all variables $V$ in the global, controllable, and stacking ensembles.   
\end{defn}

We note that the global variables are always good.
We also note if some $V$ is bad then $V=0$,
as a copy of $V$ would require either a triangle in $G(i)$
(which does not exist in the triangle-free process!)
or a triangle containing two edges in $G(i)$ and one open pair
(which contradicts the definition of `open').
For example, if $uv$ is an edge then $Y_{uv}=0$.
On the other hand, if $uv$ is closed (not an edge or open)
then we do track $Y_{uv}$; this will be important
e.g.\ for \eqref{claim:SC} in the proof of Theorem~\ref{edges}.

As indicated earlier, for the actual definition of the variable
$\mc{Z}V(i)$ appearing in the trend hypothesis of Section \ref{sec:strategy}
we `freeze' it at step $J_V$, as follows:

\begin{equation} \label{eq:freeze}
\mc{Z}V(i) = \begin{cases}
| \mc{D}V(i) | - \dD_V(t)v(t) \ \text{ if } \ i<J_V \\
\mc{Z}V(J_V-1) \ \text{ if } \ i \ge J_V.
\end{cases}
\end{equation}

While the stopping times $\I1, \I2, \I3$ are the main subject of the proof, 
we will also need some additional information about the evolution of the process,
which will be captured by the `external' stopping time $\I0$. 
This includes properties of $ G(i)$ for $ i < n^{5/4} $, 
sharper estimates on $Q$ and $Y_{uv}$ for $ i < i_Y $, 
crude estimates for a broad class of extension variables,
and control of vertex degrees (which cannot be treated by
the general strategy applied to all other variables).

\begin{defn} \label{def:I0}
We let the stopping time $ \I0$ be the first step $i$ 
at which $ G(i) \notin \mc{G}_i$
(or $\infty$ if there is no such step), 
where $\mc{G}_i$ is the `good event' 
that the following estimates hold:
\begin{enumerate}[(i)]
\item $Q(i)/q(t)$, $X_u/x_1(t)$ and $X_{uv}/x(t)$
are $1 \pm O(t^2)$ for every vertex $u$ and pair $uv$,
whenever $n^{-0.49} \le t \le 0.01$,
\item $Y_{uv}(i)/y(t)$ and $Y_u(i)/y_1(t)$
are $1 \pm O(L^8 t^2) \pm O(t^{-0.4}n^{-0.2})$
for every vertex $u$ and non-edge $uv$, 
whenever $n^{-0.49} \le t \le 0.01$,
\item $Z_{uv}(i) \le L^4 $ for all pairs $uv$, 
where the {\em codegree} 
$ Z_{uv}(i)$ is the number of vertices  
adjacent to both $u$ and $v$ in $G(i)$,
\item For every extension $(A,J,\GG)$ 
on at most $M^3$ vertices
and all injections $\phi: A \to [n]$ we have 
$ X_{ \phi, J, \GG}(i) < L^{4 | V_\GG |} 
 \max_{A \subseteq B \subseteq V_\GG } S_B^{V_\GG}(i)$, 
\item For every extension $(A,J,\GG)$ on $M^3+1$ vertices 
such that $ S_B^{V_\GG} \le  y/{L^7} $ 
for all $ A \sub B \sub V_\GG$
and all injections $\phi: A \to [n]$ we have 
$ X_{ \phi, J, \GG}(i) < L^{4 | V_\GG |} 
 \max_{A \subseteq B \subseteq V_\GG } S_B^{V_\GG}(i)$,
\item no good controllable variable $V=X_{\phi,J,\GG}$
has $|\mc{D}V(i)| > n^{-\dD^2} v(t)$ 
for any $n^{-1/4} \le t \le t_V$ such that
$S^B_A > n^{\dD'}$ for all $A \subn B \sub V_\GG$,
\item $Y_u(i) = (1 \pm \dD_{Y_1}(t)) y_1(t)$
for every vertex $u$.
\end{enumerate}
\end{defn}

To aid intuition, we make some remarks on the use of the various
properties in the definition of $\mc{G}_i$. The error terms from
$Q$ and $Y_{uv}$ are ubiquitous throughout the calculations,
and the tighter control expressed by (i) and (ii) 
handles some technical difficulties that arise for small $t$;
a similar motivation applies to (vi).
We include (vii) in $\mc{G}_i$ as the vertex degrees
cannot be treated by the same method used for the other variables.
Combining (i) and (ii) with the martingale estimates
for $Q$ and $Y_{uv}$ after their activation steps,
we obtain the following bounds that hold 
for all $n^{5/4} \le i < I$.
We emphasise that we will often use without further comment 
the facts that the approximation errors $\dD_Q^*$ and $\dD_Y^*$ 
for $Q$ and $Y_{uv}$ have sublogarithmic decay
and $\dD_Y^*=O(\dD_Y)$ for all $i \ge n^{5/4}$.
\begin{align} \label{eq:e*}
& \text{For } n^{5/4} \le i < I \text{ we have }
Q(i) = (1 \pm \dD_Q^*)q(t) \text{ and }
Y_{uv}(i) = (1 \pm \dD_Y^*)y(t) \text{ if } uv \notin E(i), \\
& \text{ where } \dD_Q^* \le \dD_Q, \ \dD_Q^* = O(t^2), \
 \dD_Y^* \le 2\dD_Y \text{ for } i \ge i_Y 
 \text{ and } \dD_Y^* = O(L^8 t^2) + O(t^{-0.4}n^{-0.2}). \nonumber
\end{align}

The intuition for the codegree variable $Z_{uv}$ in (iii) is that 
it should scale in expectation like $p^2 n = 2t < \sqrt{\log n}$,
so whp will be at most polylogarithmic.
An important application is that
\begin{equation} \label{eq:dfid}
\text{For any two open pairs } e \text{ and } e'
\text{ at most } L^4 
\text{ open pairs can simultaneously close both}. 
\end{equation} 
We think of \eqref{eq:dfid} as `destruction fidelity',
as it will allow us to approximate the number of possibilities
for a set of destruction events by a sum over each event.
To see that \eqref{eq:dfid} follows from (iii),
we can assume that $e$ and $e'$ share a vertex
(otherwise at most $2$ pairs can close both),
say $e=xu$ and $e'=xv$, and then the required
bound is immediate from $Z_{uv} \le L^4$.
The bound on $Z_{uv}$ is similar to those
in (iv) and (v), but we state and prove it separately
to emphasise its importance and because its proof
is much simpler than those of the general statements.

Conditions (iv) and (v) in $\mc{G}_i$
both give the same estimate
(under different hypotheses) for general extensions.
This estimate is quite crude, in that it exceeds by a 
polylogarithmic factor $L^{4 | V_\GG |}$ the 
`worst-case expectation estimate'
$\max_{A \subseteq B \subseteq V_\GG } S_B^{V_\GG}(i')$
(our union bounds cannot rule out the event that $\phi$ 
extends to some embedding of $(B,J,\GG)$,
which we would then expect to have 
$S_B^{V_\GG}$ extensions).
This polylogarithmic loss makes it ineffective 
when verifying trend hypotheses, 
but it is easily absorbed
when verifying boundedness hypotheses.
This will be crucial for the controllable ensemble,
where we recall that we imposed the size restriction
$|V_\GG| \le M^3$, so condition (v) enables us to verify
the boundary case $|V_\GG| = M^3$
(this idea makes our treatment of extensions significantly
simpler than that in \cite{theother}).

Now we state our main result on the triangle-free process.

\begin{theo}\label{good}
With high probability $I > i_{max} := t_{max}  n^{3/2}$.
\end{theo}

The lower bound in Theorem \ref{edges} follows from Theorem \ref{good}.  
To see this, we note that if $I > i_{max}$
then $I_Q > i_{max}$, so the process persists
until time $t_{max} = \tfrac{1}{2} \sqrt{ (1/2 - \eps) \log n}$,
and $\I0 > i_{max}$, so by Definition \ref{def:I0}.vii  
all vertex degrees at time $t_{max}$ are  
$(1 \pm \dD_{Y_1}(t_{max})) y_1(t_{max}) 
= (1+o(1))2t_{max}n^{1/2}$.

We will prove Theorem \ref{good} over the next four sections,
in which we in turn bound the probabilities of the events
$\{ I=\I0 \le i_{max} \}$ (Theorem \ref{Gformal}),
$\{ I=\I1 \le i_{max} \}$ (Theorem \ref{mainensemble}),
$\{ I=\I2 \le i_{max} \}$ (Theorem \ref{controllable})
and $\{ I=\I3 \le i_{max} \}$ (Theorem \ref{stacking});
in combination these theorems imply Theorem \ref{good}.

Note that if $I \le i_{max}$ then 
either $ G(I) \not\in \mc{G}_I$ or there is some $V$ 
such that $I=I_V \le i_{max}$, i.e.\ $|\mc{D}V(I)|$ is too large 
and $V$ is good at step $I$.  We emphasize that, 
since we can restrict our attention to $ i < I$, 
we may assume $ \mc{G}_i$ and  $|\mc{D}V(i)| \le \dD_V(t) v(t)$ 
for all good variables $V$ when verifying the trend and boundedness hypotheses.

\subsection{Some calculations and further notation}

We will employ the following useful lemma extensively to estimate sums of products.  
The proof given here is due to Patrick Bennett.
\begin{lemma}[Product Lemma]
\label{lem:pat}
Suppose $x$, $y$, $ (x_i)_{i\in I} $ and $ (y_i)_{i \in I}$ 
are real numbers such that
$|x _i - x| \le \dD$ and $|y_i-y| < \eps$
for all $i \in I$.  Then we have
\[ \left| \sum_{i \in I} x_i y_i -  \frac{1}{|I|} \left( \sum_{i \in I} x_i \right)
\left( \sum_{i \in I} y_i \right) \right|  \le  2 |I| \dD \eps \]
\end{lemma}
\begin{proof}
The triangle inequality gives
$$\displaystyle \left| \sum_{i \in I} (x_i - x)(y_i - y) \right| \leq |I|\dD\eps.$$
Rearranging this inequality gives
\begin{equation*}
\begin{split}
\sum_{i \in I} x_i y_i  & =  x\sum_{i \in I}y_i + y \sum_{i \in I} x_i
- |I|xy \pm |I|\dD \eps \\
& = \frac{1}{|I|} \left( \sum_{i \in I} x_i \right)
\left( \sum_{i \in I} y_i \right)  - |I| \left(\frac{1}{ |I|} \sum_{i \in I} x_i  - x\right)
\left(\frac{1}{ |I|} \sum_{i \in I} y_i  - y\right) \pm |I|\dD \eps.
\end{split}
\end{equation*}
\end{proof}
\noindent
The following notation and conventions 
that are used throughout the paper.
\begin{itemize}
\item We use compact notation for one-step differences, 
writing $\DD_i(F) = F(i+1)-F(i)$ for any sequence $F(i)$
and $\DD_i(f) = f( t+ n^{- 3/2}) - f(t)$ for any function $f(t)$. 
\item The `O-tilde' notation $f = \wt{O}(g)$ 
means $|f| \le (\log n)^A |g|$ for some absolute constant $A$.
\item `whp' means `with high probability';
all such statements will have subpolynomial failure probability,
which will justify us taking a polynomial number of them in union bounds.
\item We reiterate that we denote the vertex set by $[n]=\{1,\dots,n\}$.
\end{itemize}

We conclude this section by estimating the one-step
differences for variable scalings $v$ and error terms $v\dD_V$
(recall Definitions \ref{def:e} and \ref{def:c}). 
To interpret the latter formula,
note that in the main term we have factored out the scaling $v\dD_V$
and the approximate probability $8tn^{-3/2}$
(see \eqref{eq:close2}) of closing 
any given open pair at step $t$; a crucial feature 
of the trend hypothesis calculations later will be the 
self-correction of open pairs in $V$
that cancels the $o(V)$ term.
We let $P_V$ denote the power of $e$ in $\dD_V$,
i.e.\ $P_V$ equals $2$, $1$ or $\dD$ according
as $V$ is global, stacking or controllable.

\begin{lemma} \label{lem:seconds}
For any variable $V$ in any ensemble and $t \ge n^{-1/4}$ we have
\[ \DD_i(v) = v' n^{-3/2} + O(v) n^{-5/2}, \text{ and } \]
\[ \DD_i(v\dD_V) 
 = \left( \tfrac{e(V)}{8t^2} - o(V) \right) \dD_V v \cdot 8tn^{-3/2} 
  + \dD'_V vn^{-3/2} + O(\dD_V v) n^{-5/2}, \text{ where } \]
\[ \dD'_V \ge 4tP_V \dD_V + (\tT'/\tT - e(V)t^{-1}) 2g_V. \]
\end{lemma}

\begin{proof}
By Taylor's Theorem, for any smooth function $h(t)$ we have
\[ \DD_i(h) = h'(t)n^{-3/2} + O(n^{-3} |h''(t')|),
\ \text{ where } \ t < t' < t + n^{-3/2}. \]
We apply this first with $h=v$, 
which has the form  $v(t) = a(t)e^{b(t)}$, 
where $a$ and $b$ are polynomials in $t$
and $b$ has degree at most 2,
so satisfies $v'/v = O(t+t^{-1}) = O(n^{1/4})$
and $v''/v = O(t^2+t^{-2}) = O(n^{1/2})$ for $t \ge n^{-1/4}$;
this gives the first estimate.
For the second, we recall that
$v = n^{n(V)} p^{e(V)} \hat{q}^{o(V)}$,
so \[ v'/v = e(V)/t - 8to(V).\]
Applying Taylor's Theorem to $h = v\dD_V$,
as $h'/h = v'/v + \dD'_V/\dD_V$
the main term in the second estimate
is equal to $h'(t)n^{-3/2}$,
so it remains to show 
$|h''(t')| = O(n^{1/2}) \dD_V v$ for $t \le t' \le t+n^{-3/2}$.
To see this, we recall that $\dD_V = f_V + 2g_V$,
where $vf_V$ and $vg_V/\tT$ 
both have the form $a(t)e^{b(t)}$ as above, 
so $(vf_V)'' = O(t^2 + t^{-2}) vf_V = O(n^{1/2}) v\dD_V$,
$(vg_V/\tT)' = O(t + t^{-1}) vg_V/\tT = O(n^{1/4}) v\dD_V$ and
$(vg_V/\tT)'' = O(t^2 + t^{-2}) vg_V/\tT = O(n^{1/2}) v\dD_V$.
Recalling that $\tT'$ and $\tT''$ are bounded
(see Definition \ref{def:c}) we deduce 
$(vg_V)'' = (vg_V/\tT)'' \tT + 2 (vg_V/\tT)' \tT'
+ \tT'' = O(n^{1/2}) v\dD_V$, as required.
The bound on $\dD'_V$ follows from $f'_V/f_V = 4tP_V$
and $g'_V/g_V = 4tP_V + \tT'/\tT 
- e(V)t^{-1}(1+t^{e(V)})^{-1}$.
\end{proof}

\section{Coupling and union bounds} \label{sec:couple}

In this section we gather two types of estimates that can be made without using dynamic concentration, namely coupling and union bounds.
The two key applications of these arguments are
(i) showing that whp every variable $V$ in each of three ensembles
obeys its required estimates at its activation step $i_V$
(see Lemma \ref{lem:beginning}), and
(ii) showing that whp the stopping time $\I0$
of Definition \ref{def:I0} controlling
the good event $\mc{G}_i$ does not occur by step $i_{max}$.
We state the latter as the main theorem of this section.
 
\begin{theo}\label{Gformal}
With high probability we do not have $I=\I0 \le i_{max}$.
\end{theo}

Theorem~\ref{Gformal} follows by combining various lemmas
proved in this section showing that each of the defining properties 
of the event $\mc{G}_i$ in Definition \ref{def:I0} hold whp; specifically, 
properties (i), (ii) and (vi) are in Lemma \ref{lem:beginning},
(iii) in Lemma \ref{zsmall}, (iv) in Lemma \ref{extensions},
(v) in Lemma \ref{pre-extensions}
and (vii) in Lemma \ref{lem:vdeg}.

\subsection{Extension variables in $G_{n,p}$}

Our coupling arguments will compare extension variables in
the triangle-free process $G(i)$ with extension variables 
in the Erd\H{o}s-R\'enyi random graph $G_{n,p}$.
In this subsection we briefly review some well-known
theory of the latter.
Suppose $J$ is a graph and $A \sub V_J$. 
We refer to $(A,J)$ as an \emph{extension}. 
Given an injective map $\phi:A \to [n]$, 
where $[n]$ is the vertex set of $ G_{n,p}$, we let
$X^{ER}_{\phi,J}$ be the number of injective maps $f:V_J \to [n]$ 
such that $f$ restricts to $\phi$ on $A$ and 
$f(e)$ is an edge of $G_{n,p}$ for every $e \in J \sm J[A]$.
Thus $X^{ER}_{\phi,J}$ is formally defined in the same way as
the extension variable $X_{\phi,J,J}$ on $G(i)$
(see Section~\ref{subsub:defcon}), but we emphasise
that $X^{ER}_{\phi,J}$ is defined on $ G_{n,p}$, not on $G(i)$.

The following definition and accompanying lemma describe how
a general extension can be naturally decomposed into a series of
extensions that are `strictly balanced', in that they do not have
any `dense subextension'.

\begin{defn}
Given $A \sub B \sub B' \sub V_J$ we define the \emph{scaling} 
$$S^{B'}_B = S^{B'}_B(J) = n^{|B'| - |B|} p^{e_{J[B']}-e_{J[B]}}.$$
We say that $(A,J)$ is \emph{strictly balanced} (in $G_{n,p}$) 
if $S^{V_J}_B < 1$ for all $A \subn B \subn V_J$.
The \emph{extension series} (in $G_{n,p}$) for $(A,J)$, 
denoted $(B_0,\dots,B_d)$, is constructed by the following rule.
We let $B_0 = A$. For $i \ge 0$, 
if $(B_i,J)$ is not strictly balanced then we choose $B_{i+1}$
to be a minimal set $C$ with $B_i \subn C \subn V_J$ 
that minimises $S^C_{B_i}$; otherwise we
choose $B_{i+1}=V_J$, set $d=i+1$ and terminate the construction.  
\end{defn}

\begin{lemma} \label{lem:extGnp}
Let $(A,J)$ be an extension and $(B_0,\dots,B_d)$
be its extension series in $G_{n,p}$. Then
\begin{enumerate}[(i)]
\item if $A \sub B \sub B' \sub B'' \sub V_J$ 
then $S^{B''}_B =  S^{B'}_B S^{B''}_{B'}$,
\item if $A \sub B \sub V_J$ and $C \sub V_J \sm B$ 
then $S^{B \cup C}_{A \cup C} \le S^B_A$,
\item each extension $(B_i,J[B_{i+1}])$ is strictly balanced,
\item $S^{B_{i+1}}_{B_i} \ge 1$ for $i>0$.
\end{enumerate}
\end{lemma}

\begin{proof}
Statements (i) and (ii) are clear.
For (iii), we cannot have $S^{B_{i+1}}_B \ge 1$ 
for some $B_i \subn B \subn B_{i+1}$, as then 
$S^B_{B_i} = S^{B_{i+1}}_{B_i}/S^{B_{i+1}}_B 
\le S^{B_{i+1}}_{B_i}$ contradicts minimality of $B_{i+1}$.
For (iv), suppose for contradiction that 
$S^{B_{i+1}}_{B_i} < 1$ for some $i>0$.
If $i+1<d$ then $S^{B_{i+1}}_{B_{i-1}} <  S^{B_i}_{B_{i-1}}$
contradicts the definition of $B_i$.
On the other hand, if $i+1=d$ we will obtain 
a contradiction by showing that $(B_{i-1},J)$
is strictly balanced (so the extension series
should have terminated with $B_i=V_J$).

To see this, consider any
$B_{i-1} \subn B \subn V_J$ and write 
$B^\cup = B \cup B_i$, $B^\cap = B \cap B_i$.
By strict balance of $(B_i,J)$
we have $S^{V_J}_{B^\cup} \le 1$,
with equality only if $B^\cup=V_J$
(as $S^{B_{i+1}}_{B_i} < 1$).
By (ii) and strict balance of $(B_{i-1},J[B_i])$ we have
$S^{B^\cup}_B \le S^{B_i}_{B^\cap} \le 1$,
with equality only if $B^\cap = B_i$.
At least one of these inequalities is strict,
so $S^{V_J}_B = S^{B^\cup}_B S^{V_J}_{B^\cup} < 1$.
This contradiction completes the proof.
\end{proof}

Next we quote the following general extension estimate
of Kim and Vu  \cite[Theorem 4.2.4]{kimvu} 
in a weakened form that suffices for our purposes.
\begin{lemma} \label{kim-vu}
For any $\aA>0$ there is $\bB>0$ so that
for any extension $(A,J)$ with $S^B_A > n^\aA$ 
for all $ A \subsetneq B \sub V_J$ in $ G_{n,p}$ 
whp $X^{ER}_{\phi,J} = (1 \pm n^{-\bB}) S^{V_J}_A$ 
for all injections $\phi:A \to [n]$.
\end{lemma}

We also require a weaker estimate that can be applied to sparse extensions,
as given by the following union bound lemma.
We include a brief proof as it illustrates a method 
we will also use for similar estimates in the triangle-free process.
We recall that $L=\sqrt{\log n}$.

\begin{lemma}\label{Gnpbound}
If $(A,J)$ is strictly balanced in $G_{n,p}$ then whp 
$X^{ER}_{\phi,J} < L^{4|V_J \sm A|}  \max \{S^{V_J}_A,1\}$ 
for all injections $\phi:A \to [n]$.
\end{lemma}

\begin{proof}
First we note that for any fixed $f:V_J \to [n]$ restricting to $\phi$ on $A$ 
we have $\mb{P}(f \in X^{ER}_{\phi,J}) = p^{e_J - e_{J[A]}}$. 
Next we estimate the probability that there are
$s$ extensions in $X^{ER}_{\phi,J}$ that are disjoint outside of $\phi(A)$. 
An upper bound is $s!^{-1} (n^{v_J-|A|})^s \cdot (p^{e_J - e_{J[A]}})^s 
< (3s^{-1}S^{V_J}_A)^s$, which is subpolynomial for $s = L^4 \max \{S^{V_J}_A,1\}$.

Now we show the statement of the lemma by induction on $|V_J \sm A|$.
The base case $|V_J \sm A|=1$ holds by the bound on disjoint extensions.  
Now suppose $|V_J \sm A| > 1$.  We consider a maximal collection $C$ 
of extensions disjoint outside of $A$. 
As shown above, whp $|C| \le s = L^4 \max \{S^{V_J}_A,1\}$.
By maximality, any extension $\phi' \in X^{ER}_{\phi,J}$
intersects some extension $\phi^* \in C$ outside of $A$.  
By strict balance and the induction hypothesis,
for any $\phi^*$ the number of choices for $\phi'$ 
is at most $ 2^{|V_J|} L^{4(|V_J \sm A|-1)} < L^{4|V_J \sm A|-1}$.
Therefore  $X^{ER}_{\phi,J} < L^{4|V_J \sm A|-1} |C| 
< L^{4|V_J \sm A|} \max \{S^{V_J}_A,1\}$.
\end{proof}

We deduce the following estimate on general extensions.

\begin{lemma} \label{gnp-general}
For any extension $(A,J)$ whp 
$X^{ER}_{\phi,J} < L^{4|V_J \sm A|} \max_{A \sub B \sub V_J} S^{V_J}_B$ 
for all $\phi$.
\end{lemma}

\begin{proof}
Let $(B_0,\dots,B_d)$ be the extension series in $G_{n,p}$ for $(A,J)$. 
By Lemma \ref{lem:extGnp}.iii we can apply Lemma \ref{Gnpbound} bound
to each step of the extension series, so whp for each $0 \le i < d$
and injection $\phi_i:B_i \to [n]$ we have
$X^{ER}_{\phi_i,J[B_{i+1}]} < L^{4|B_{i+1} \sm B_i|}  
\max \{S^{B_{i+1}}_{B_i},1\}$.
Thus for any injection $\phi:A \to [n]$ we have
$X^{ER}_{\phi,J} < \prod_{i=0}^{d-1}
L^{4|B_{i+1} \sm B_i|}  \max \{S^{B_{i+1}}_{B_i},1\}$.
By Lemma  \ref{lem:extGnp}.iv we have 
$S^{B_{i+1}}_{B_i} \ge 1$ for $i \ge 1$,
so $X^{ER}_{\phi,J} < L^{4|V_J|}
\max \{S^{B_1}_{B_0},1\} S^{V_J}_{B_1}
= L^{4|V_J|} \max \{ S^{V_J}_{B_0}, S^{V_J}_{B_1}\}$.

It remains to show that this bound 
is identical to that claimed by the lemma.
To see this, consider any $A \sub B \sub V_J$ 
and write $B^\cup = B \cup B_1$, $B^\cap = B \cap B_1$  
and $S^{V_J}_B = S^{B^\cup}_B S^{V_J}_{B^\cup}$.
Then $S^{B^\cup}_B \le S^{B_1}_{B^\cap} 
\le \max \{S^{B_1}_{B_0},1\}$ and 
$S^{V_J}_{B^\cup} \le S^{V_J}_{B_1}$
by Lemma \ref{lem:extGnp}, as required.
\end{proof}

\subsection{Coupling estimates}

In this subsection we estimate our variables for small $t$ 
by coupling the triangle-free process $G(i)$ inside
the Erd\H{o}s-R\'enyi random graph process $ER(n,j)$,
which is defined in the same way as $G(i)$
but without the condition of being triangle-free,
i.e.\ we consider a uniform random order of the set of pairs in $[n]$ 
and let the edge-set of $ER(n,j)$
consist of the first $j$ pairs in this order.
The coupling is defined by rejecting any pair in $ER(n,j)$ that is closed, 
in that it forms a triangle with previous (non-rejected) edges. 
Thus after $j$ steps the selected edges form the triangle-free process 
$G(i)$ after $i$ steps, where $j-i$ edges were rejected. 
The number of rejected edges is bounded by the number of 
triangles in $ER(n,j)$; call this $T(j)$. 

The intuition (made precise in Lemma \ref{lem:coupler} below) is that 
for small $t$ few edges are rejected, so variables in $G(i)$ 
are well-approximated by corresponding variables in $ER(n,j)$.
This allows us to side-step technical difficulties that arise for small $t$
when implementing the main martingale strategy of Section \ref{sec:strategy}
(i.e.\ that powers of $t$ in the error functions blow up for small $t$,
and in any case we have to exclude very small $t$ to obtain concentration).
We will see in the calculations below that the coupling gives us the required
bounds up to $t=n^{-1/4}$ (and beyond in some cases), which explains our
choice of activation step $i_V$ in Definition \ref{def:active} above.

A well-known paradigm of Random Graphs is that 
the random graph $ER(n,j)$ of fixed size is very similar 
to the usual binomial model $G_{n,p_j}$
where edges are chosen independently
with probability $p_j = j/\binom{n}{2}$;
the following lemma makes this statement precise.
\begin{lemma}[Lemma~1.2 in \cite{alan}]
\label{alan}
Let $ \mc{P} $ be any graph property and $ p_j = j/\binom{n}{2} $ 
where $ j = j(n) \to \infty$ and $\binom{n}{2} -j \to \infty$.  
Then for $n$ sufficiently large
\[ \mb{P}( ER(n,j) \in \mc{P}) \le 10 j^{1/2} \mb{P}(  G_{n,p_j} \in \mc{P}). \]
\end{lemma}

We will view $j=j(i)$ as a random variable on the probability space
of the coupling of $G(i)$ and $ER(n,j)$, which is equal to the number
of steps of the Erd\H{o}s-R\'enyi process $ER(n,j)$ 
that are revealed in order to obtain $i$ edges 
in the coupled triangle-free process $G(i)$.
We can approximate $j(i)$ and so estimate variables in $G(i)$
by those in $G_{n,p}$ as follows.

\begin{lemma} \label{lem:coupler}
If $i=tn^{3/2}$ with $t \in (n^{-0.49},0.01)$
then whp $i \le j(i) < (1+O(t^2)) i$.
Thus for any extension $(A,J,\GG)$ 
and injection $\phi:A \to [n]$
whp $X_{\phi,J,\GG} \le X^{ER}_{\phi,J}$ 
in $G_{n,p'}$ with $p' = (1+O(t^2))p$.
\end{lemma}

\begin{proof}
By definition of the coupling we have $0 \le j-i \le T(j)$, 
where $T(j)$ is the number of triangles in $ER(n,j)$.
As $t>n^{-0.49}$, by Lemmas~\ref{kim-vu}~and~\ref{alan} whp 
$T(j) < 2p_j^3 n^3 < 20 (j/n)^3$.
We deduce $j < 2i$, as at step $2i$ we have seen 
at least $2i - 20(2i/n)^3 = (1-80t^2) 2i > i$ 
edges of the triangle-free process
(using $t<0.01$). Thus $T(j) = O(t^2) i$, 
which gives the first statement.

To see the second, note that $X_{\phi,J,\GG}$
is bounded deterministically (via the coupling)
by $X^{ER}_{\phi,J}$ in $ER(n,j(i))$,
and by Chernoff bounds on the number of edges
in $G_{n,p'}$ we can include $G_{n,p'}$ in the coupling
(`tripling'?) so that whp $ER(n,j(i)) \sub G_{n,p'}$.
\end{proof}

Having established the coupling, we now turn to its application,
which is to show that any good variable $V$ is not in or beyond
its critical interval at its activation step $i_V$ when we begin 
its martingale analysis; 
this is the final statement of the next lemma
(we also include some stronger bounds required for
the event $\mc{G}_i$ in Definition \ref{def:I0},
and a stronger statement for stacking variables).
We require these bounds as earlier steps
are not covered by the martingale analysis:
we recall from Definition \ref{def:IV} that the stopping time $I_V$
is the smallest $i$ with $i_V < i < J_V$ 
such that $|\mc{D}V(i)| > \dD_V(t) v(t)$ 
(or $\infty$ if there is no such time).
We can assume $V$ is good
by definition of $J_V$ (also in Definition \ref{def:IV}).
For convenience, we recall the estimates 
on $t_V = i_V n^{-3/2}$ given in Lemma \ref{lem:tV}:
if $e(V)=0$ or $V=S$ then $t_V=n^{-1/4}$, otherwise
$t_V = \wt{\TT}(n^{-1/4e(V)})$ if $V$ is a stacking variable 
or $t_V = \wt{\TT}(n^{-\dD/4e(V)})$ if $V$ is a controllable variable.
 
\begin{lemma} \label{lem:beginning}
With high probability 
\begin{enumerate}[(i)]
\item $V(i) = (1 \pm O(t^2))v(t)$ for any good variable $V$
with $e(V)=0$ and $n^{-0.49} \le t \le 0.01$,
\item $Y_{uv}(i)/y(t)$ and $Y_u(i)/y_1(t)$
are $1 \pm O(L^8 t^2) \pm O(t^{-0.4}n^{-0.2})$
for every vertex $u$, non-edge $uv$ 
and $n^{-0.49} \le t \le 0.01$,
\item no good controllable variable $V=X_{\phi,J,\GG}$
has $|\mc{D}V(i)| > n^{-\dD^2} v(t)$ 
for any $n^{-1/4} \le t \le t_V$ such that
$S^B_A > n^{\dD'}$ for all $A \subn B \sub V_\GG$,
\item no good stacking variable $V$
has $|\mc{D}V(i)| > (f_V(t)+ g_V(t)) v(t)$ 
for any $n^{-1/4} \le t \le t_V$,
\item no good variable $V$ has 
$|\mc{D}V(i_V)| > (f_V(t_V)+ g_V(t_V)) v(t_V)$.
\end{enumerate}
\end{lemma}

\begin{proof}
For (i), we first estimate the maximum degree $\DD(i)$ of $G(i)$.
By Lemma \ref{lem:coupler}, we can bound $\DD(i)$ whp
by the maximum degree in $G_{n,p'}$ with  
$p' = (1+O(t^2))p = O(p)$, so whp $\DD = O(pn) = O(tn^{1/2})$.
Thus any vertex is incident to $O(pn)$ edges 
and $O(pn)^2 = O(t^2) n$ closed pairs.
Now consider any variable $V$ with $e(V)=0$,
and recall that $v(t) = n^{n(V)} \hat{q}(t)^{o(V)}$,
where $\hat{q}(t) =  e^{-4t^2} = 1 - O(t^2)$.
We have $n^{n(V)} \ge V(0) \ge V(i) 
\ge v(t) - O(t^2) n \cdot n^{n(V)-1}$,
so $V(i) = (1+O(t^2))v(t)$, as required.
This also proves (v) for such variables;
indeed, we have $t_V = n^{-1/4}$,
so $\mc{D}V(i_V) = O(n^{-1/2}) v(t)$ and
$f_V(n^{-1/4}) + g_V(n^{-1/4}) 
\ge f_R(n^{-1/4}) + g_R(n^{-1/4})
= \TT(L^{40} n^{-1/2}) 
\gg \mc{D}V(i_V)/v(t)$.

For (ii), consider any non-edge $uv$
and $n^{-0.49} < t < 0.01$.
By Lemma \ref{lem:coupler} we can bound $Y_{uv}$ whp 
above by the degree $d(v)$ of $v$ in $G_{n,p'}$ 
with $p' = (1+O(t^2))p$. By Chernoff bounds whp
$d(v) = (1+O(t^2)) pn \pm (pn)^{0.6}$,
where $d(v)/y(t) = 1 + O(t^2) + O(tn^{1/2})^{-0.4}$
as $y(t) = (1+O(t^2)) pn$ and $pn=2tn^{1/2}$.
We can bound $Y_{uv}$ below whp by
$d^{ER}(v) - T(v) - P_3(uv)$, where $d^{ER}(v)$ 
is the degree of $v$ in $ER(n,j(i))$,
and $T(v)$, $P_2(v)$ are the numbers
of triangles containing $v$
and paths of length $3$
from $u$ to $v$, both in $G_{n,p'}$
(a bound on the same quantities in $ER(n,j(i))$).
By Lemma \ref{gnp-general}, noting that $pn > n^{0.01}$,
we can bound $T(v)$ and $P_2(v)$ 
by $L^8 \max\{ p^3 n^2, 1\} = O(L^8 t^2) y$,
which gives the stated estimate for $Y_{uv}$.
The argument for $Y_u$ is the same, 
except that there is no $P_3(uv)$ term.

For (iii), we have already shown the required bounds
when $e(V)=0$, so we can assume $e(V)>0$.
By Lemma \ref{lem:coupler} (which applies as 
$n^{-1/4} \le t \le t_V 
= \wt{\TT}(n^{-\dD/4e(V)}) < 0.01$) 
we can bound $V(i)$ whp above by $X_{\phi,J}$ 
in $G_{n,p'}$ with $p' = (1+O(t^2))p $.
As $S^B_A > n^{\dD'}$ for all $A \subn B \sub V_\GG$
by Lemma \ref{kim-vu} we have 
$X_{\phi,J} = (1 \pm n^{-2\dD^2}) v(t)$, say,
as $\dD \ll \dD' \ll \eps$ and $e(V) < M^2 = 9\eps^{-2}$.
For a lower bound on $V(i)$, we consider for each pair
$xy$ in $V_J$ not contained in $A$ how it can
prevent extensions in $X_{\phi,J}$ from being counted in $V$
(we do not need to consider $xy \sub A$, as such edges 
either make $V$ bad or have no effect on $V$).
We let $J+xy$ be obtained from $J$ by adding $xy$ as an edge
and $J*xy$ be obtained from $J$ by adding a new vertex $z$
adjacent to both $x$ and $y$. Then we can bound $V(i)$ whp below 
by $X_{\phi,J} - \sum_{xy} X_{\phi,J+xy} - \sum_{xy} X_{\phi,J*xy}$.

We will bound both $X_{\phi,J+xy}$ and $X_{\phi,J*xy}$
by $n^{-2\dD^2} v$. To see this bound for $X_{\phi,J+xy}$, 
note that $S_A^{V_J}(J+xy) = pv $ and 
for any $A \subn B \sub V_J$ that
$S_B^{V_J}(J+xy) \le  S_B^{V_J}(J) = v/S^B_A < n^{-\dD'} v$,
so $X_{\phi,J+xy} < n^{-2\dD^2} v$ by Lemma \ref{gnp-general}.
A similar argument applies to $X_{\phi,J*xy}$
(also using $t = \wt{O}(n^{-\dD/4e(V)})$),
so $V(i) = (1 \pm 4n^{-2\dD^2})v(t)$.
As $\mc{T}V(i) = v(t)(Q/q)^{o(V)} 
= (1+O(t^2))v(t) = (1 \pm n^{-2\dD^2})v(t)$,
this gives (iii). As $g_V(t_V)=L^{-1}$ by definition,
this also proves (v) for controllable variables.

For (iv), we may assume $e(V)>0$.
As $\dD_V(t) = \OO(\dD_V(t_V))$
for $n^{-1/4} \le t \le t_V$ it suffices to show 
$|\mc{D}V(i)|/v(t) = o(\dD_V(t_V))$.
Applying (i) and (ii) to each step 
in the stacking order of $V$, noting that only $O(1)$ 
choices are forbidden at each step due to using a vertex 
already used by a previous step, we obtain 
$V(i)/v(t) = 1 \pm O(L^8 t^2) \pm O(t^{-0.4}n^{-0.2})$.
Similarly, the tracking variable $\mc{T}V$ 
satisfies the same estimate for $\mc{T}V(i)/v(t)$, 
so $|\mc{D}V(i)|/v(t) < O(L^8 t^2) + O(t^{-0.4}n^{-0.2})$.
This satisfies the desired bound,
as if $e(V) \ne 1$ we have $\dD_V(t_V)=\wt{\TT}(1)$
and $t_V = \wt{\TT}(n^{-1/4e(V)})$,
so $|\mc{D}V(i)|/v(t) = \wt{O}(n^{-1/2e(V)} + n^{-0.1})$
or if $e(V)=1$ (see Definition \ref{def:active})
we have $\dD_V(t_V)=\wt{\TT}(n^{-0.05})$ and $t_V = n^{-0.24}$,
so $|\mc{D}V(i)|/v(t) = O(n^{-0.1})$.
This proves (iv) and (v) for stacking variables.

For (v), the only remaining case is $V=S$,
for which we recall $t_S=n^{-1/4}$.
We have $S(n^{5/4}) \le 2n^{5/4} \cdot n = 2n^{9/4}$,
as each triple counted by $S$ determines an ordered edge and a vertex.
We do not count such triples if the other pairs are closed or edges,
so $S(n^{5/4}) \ge 2n^{9/4} - 2P_2 - 2P_3$, where $P_\ell$ 
is the number of paths of length $\ell$ in $G_{n,p'}$ with 
$p' = O(n^{-3/4})$ (using Lemma \ref{lem:coupler}).
As $G_{n,p'}$ whp has degrees $O(n^{1/4})$ we have
$S(n^{5/4}) = 2n^{9/4} \pm O(n^{7/4})$, 
which is well within the desired bound
$(f_S(n^{-1/4})+ g_S(n^{-1/4})) s(n^{-1/4}) = \wt{\TT}(n^2)$.
\end{proof}

\subsection{Union bounds}

In this subsection we adapt the argument of Lemmas~\ref{Gnpbound}~and~\ref{gnp-general} to give a crude bound on general extension variables that holds throughout the triangle-free process. Along the way, we prove Theorem~\ref{subgraphs}, assuming Theorem~\ref{good}.  We start with the simplest instance of this argument, 
which is bounding the codegree $Z_{uv}(i)$
of any two vertices $u$ and $v$ in $ G(i)$.

\begin{lemma} \label{zsmall}
Whp for every non-edge $uv$,
if $i'-1<I$ then $ Z_{uv}(i') \le L^4 $.
\end{lemma}

\begin{proof}
At any step $i \le i'$ the edge added at step $i$
completes a path of length two between $u$ and $v$ 
with probability $ ( Y_{uv} + Y_{vu}) Q^{-1}$.
We can bound this probability by $O(y/q) = O(Ln^{-3/2})$
for $t \ge 1$ or by $O(y(1)/q(0)) = O(n^{-3/2})$ for $t \le 1$.
Taking a union bound over all subsets of $L^4$ steps 
at which we might increment $ Z_{uv}$, the probability that 
$Z_{uv}$ reaches $L^4$ by step $i'$ is at most
$\binom{ i_{max} }{ L^4} O\left( L n^{-3/2}  \right)^{L^4}  
= O \left( L^{-2} \right)^{L^4}$.
\end{proof}

We need some further notation and terminology 
for general extensions in the triangle-free process,
which mirrors that used previously for extensions in the Erd\H{o}s-R\'enyi process.
We say that $(A,J,\GG)$ is \emph{strictly balanced at time $t$} 
if $S^{V_\GG}_B < 1$ for all $A \subn B \subn V_\GG$.
The \emph{extension series at time $t$} for $(A,J,\GG)$, 
denoted $(B_0,\dots,B_d)$, is constructed by the following rule.
We let $B_0 = A$. For $i \ge 0$, if $(B_i,J,\GG)$ is not strictly balanced 
then we choose $B_{i+1}$ to be a minimal set $C$ with $B_i \subn C \subn V_\GG$ 
that minimises $S^C_{B_i}$; otherwise we choose $B_{i+1}=V_\GG$, 
set $d=i+1$ and terminate the construction.

In Lemma \ref{extensions} we will give a general estimate 
for extension variables in the triangle-free process.
First we illustrate the argument in the following lemma, 
which shows that sparse graph pairs do not appear; 
this is the main tool needed for the proof of Theorem \ref{subgraphs}.
Here we take $A=\es$, write $V_{J,\GG} = X_{\phi,J,\GG}$,
where $\phi$ is the unique map from $\es$ to $[n]$,
and $v_{J,\GG} = S^{V_\GG}_\es(J,\GG)$. 

\begin{lemma} \label{no-sparse}
Suppose $v_{J,\GG}(t') < n^{-c}$ for some $c>0$ and time $t'$.
Then the probability that $\mc{G}_{i'}$ holds, $i'-1<I$ 
and $V_{J,\GG}(i')>0$ is at most $2n^{-c}$.
\end{lemma}

\begin{proof}
For $t' \le L^{-1}$ we appeal to the coupling with the Erd\H{o}s-R\'enyi random graph process. By Lemma \ref{lem:coupler} it suffices to estimate the probability that $J$ appears in $G_{n,j}$, where $j=(1+o(1))i'$. The expected number of copies of $J$ is at most $2n^{-c}$, so the required bound follows from Markov's inequality. Thus it suffices to consider $t' \ge L^{-1}$. 

To estimate $\mb{P}(V_{J,\GG}(i')>0)$, we take a union bound of events, where we specify the injection $f:V_\GG \to [n]$, and for $e \in J$ we specify the 
{\em selection step} $i_e$ at which the process selects the edge $f(e)$. 
Fix some choice and let $\mc{E}$ be the specified event. 

For each $i \le i'$ we estimate the probability that the selected edge is compatible with $\mc{E}$. At a selection step $i = i_e$ the selected edge is specified, so the probability is $2/Q(i_e) = (1+o(1))2q(t_e)^{-1}$, where $t_e = n^{-3/2} i_e$
(the approximation of $Q$ by $q$ holds on $\mc{G}_{i'}$ and $i'-1<I$).

For other $i$, the required probability is $1 - N_i/Q$, 
where $N_i$ is the number of ordered open pairs 
that cannot be selected at step $i$ on $\mc{E}$.
If $i$ is a selection step we write $N_i=0$. Therefore
\begin{equation} \label{eq:nosparse}
\mb{P}(\mc{E} \wedge \mc{G}_{i'}) 
\le \prod_{e \in J} (1+o(1))2q(t_e)^{-1} \cdot \prod_{i=1}^{i'} (1 - N_i/Q).
\end{equation}

Now we estimate $N_i$ when $i$ is not a selection step. For $i < L^{-1} n^{3/2}$ we use the trivial estimate $N_i \ge 0$, so suppose $i \ge L^{-1} n^{3/2}$. Suppose there are $k_i$ choices of $e \in J$ with $i_e>i$. Then there are $|\GG \sm J| + k_i$ open pairs that must not become closed, namely the open pairs of $f(\GG \sm J)$ and the $k_i$ pairs of $f(J)$ that have yet to be selected as edges. 
We recall from \eqref{eq:dfid} that by property (iii) of $\mc{G}_{i'}$ 
only $O(L^4) = o(y)$ choices of $e_i$ can close more than one such open pair.

As $\mc{G}_{i'}$ holds and $i'-1<I$,
by \eqref{eq:e*} all $Y$-variables are $(1+o(1))y$, 
so we obtain $N_i = (1+o(1))(|\GG \sm J| + k_i)\cdot 4y$. 
Thus for $i \ge L^{-1}n^{3/2}$ we can write 
$1 - N_i/Q \le 1-(1+o(1))(A_i+B_i)$, where 
\[ A_i = |\GG \sm J| \cdot 8tn^{-3/2} = |\GG \sm J|  \cdot 8in^{-3}
\ \text{ and } \ 
B_i = k_i \cdot 8in^{-3}.  \] 
This holds for all $i$ if we set $A_i=B_i=0$ for $i < L^{-1} n^{3/2}$.

We estimate each factor by
$1 - (1+o(1))(A_i+B_i) \le \exp \{ -(1+o(1))A_i \} 
\exp \{ -(1+o(1))B_i \}$ and bound separately
the contributions from all $A_i$ and from all $B_i$.
The contribution from all $A_i$ is
\begin{align*}
\exp \left\{ - \sum_{i=1}^{i'} (1+o(1)) A_i \right\}
& = \exp \left\{ -(1+o(1)) |\GG \sm J| 
\sum_{i=L^{-1} n^{3/2}}^{i'} 8in^{-3} \right\} \\
& = (1+o(1)) \exp \left\{ - |\GG \sm J| \cdot  4 (i')^2 n^{-3} \right\} \\
& = (1+o(1)) e^{ - 4 (t')^2 |\GG \sm J|} = (1+o(1)) \hat{q}(t')^{|\GG \sm J|},
\end{align*}
since $\sum_{i=1}^{L^{-1} n^{3/2}} in^{-3} < L^{-2} = o(1)$.
The contribution from all $B_i$ is 
\begin{align*}
\exp \left\{ -\sum_{i=1}^{i'} (1+o(1)) B_i \right\}
& = \exp \left\{ - (1+o(1)) \sum_{e \in J} \sum_{i=L^{-1} n^{3/2}}^{i_e} 8in^{-3} \right\} \\
& = \prod_{e \in J} (1+o(1)) \hat{q}(t_e).
\end{align*}
Substituting in \eqref{eq:nosparse} we obtain
\[\mb{P}(\mc{E} \wedge \mc{G}_{i'}) 
\le (1+o(1)) \hat{q}(t')^{|\GG \sm J|} 
\prod_{e \in J} 2\hat{q}(t_e)/q(t_e)
= (1+o(1)) \hat{q}(t')^{|\GG \sm J|} (2n^{-2})^{|J|} .\]
Summing over at most $n^{|V_\GG|}$ choices for $f$ and $(i')^{|J|}$ choices for the selection steps, we estimate 
$\mb{P} \left( \left\{V_{J,\GG}(i')>0 \right\} 
\wedge \mc{G}_{i'} \right) < (1+o(1)) v(t') < 2n^{-c}$.
\end{proof}

\vskip3mm

\nib{Proof of Theorem \ref{subgraphs}.}
Statement (i) is immediate from \cite[Theorem 1.6(iii)]{BK}.
For (ii), fix $H' \sub H$ with $d(H')>2$.
By choosing the global parameter $\eps>0$ sufficiently small
we can assume  $|E_{H'}|(1/2-\eps) > |V_{H'}| + \eps$.
Note that if $H \sub G$ then $V_{J,H'}(i_{max})>0$ 
for some spanning subgraph $J$ of $H'$,
i.e.\ there is some potential embedding $\phi$ of $H'$ 
that survives until step $i_{max}$, in that some subgraph
$\phi(J)$ is selected by the triangle-free process,
and the remaining subgraph $\phi(H' \sm J)$ remains open,
so that it is available for the remainder of the process
(which we do not analyse).
We have \[v_{J,H'}(t_{max}) 
= n^{|V_{H'}|} p^{|E_J|} \hat{q}(t_{max})^{|E_{H'}|-|E_J|}
= n^{|V_{H'}| - |E_J|/2 - (1/2-\eps)(|E_{H'}|-|E_J|)} < n^{-\eps}.\]
Thus the result follows from Theorem \ref{good} and Lemma \ref{no-sparse}.
\qed

\vskip3mm

We now turn to a key lemma which includes the union bound arguments
that are most significant for the whole proof:
it implies property (v) of Definition \ref{def:I0}
and will also be used in the proof of Lemma~\ref{extensions},
which implies property (iv) of Definition \ref{def:I0}.

\begin{lemma} \label{pre-extensions}
For any extension $(A,J,\GG)$ with $|V_\GG|=O(1)$,
if $S_B^{V_\GG} < y/L^7 $ for all $ A \sub B \sub V_\GG$ at step $i'$
then whp we do not have $I=\I0=i'$ due to some $\phi$ with
$X_{\phi,J,\GG}(i') \ge L^{4|V_\GG \sm A|} 
\max_{A \sub B \sub V_\GG} S^{V_\GG}_B$.
\end{lemma}

\begin{proof}
As in the proof of Lemma \ref{no-sparse},
it suffices to consider $t' \ge L^{-1/2}$, 
as for smaller $t'$ we can simply appeal to the coupling with the  
Erd\H{o}s-R\'enyi random graph process (Lemma \ref{lem:coupler}) 
and apply the bound from Lemma \ref{gnp-general}. 
Furthermore, the general case of the lemma follows 
from the case that $(A,J,\GG)$ is strictly balanced,
by applying it to each step of the extension series
(in the same way that Lemma \ref{gnp-general} 
followed from Lemma \ref{Gnpbound}).
We will therefore only consider the case that
$ (A, J, \GG)$ is strictly balanced. 

We argue by induction on $|V_\GG \sm A|$.
Similarly to the proof of Lemma~\ref{Gnpbound},
we first estimate the probability that there is 
a set of $s$ extensions $\{f_1,\dots,f_s\}$ 
in $V(i') := X_{\phi,J,\GG}(i')$ that are disjoint outside of $\phi(A)$,
where $s = \max\{ L^4, 6 \max_{A \sub B \sub V_\GG} S_B^{V_\GG} \}$.

Our method for estimating this probability 
is similar to the argument of Lemma~\ref{no-sparse},
but now we consider $s$ embeddings simultaneously. 
We take a union bound of events in which we specify $f_1,\dots,f_s$, 
and for each $1 \le j \le s$ and $e \in J \sm J[A]$ 
we specify the {\em selection step} $i_{j,e}$ 
at which the process selects the edge $f_j(e)$. 
Fix some choice and let $\mc{E}$ be the specified event. 

For each $i \le i'$ we estimate the probability 
that the selected edge $e_i$ is compatible with $\mc{E}$. 
At a selection step $i = i_{j,e}$ the selected edge is specified, 
so the probability is $2/Q(i_{j,e}) = (1+o(1))2q(t_{j,e})^{-1}$, 
where $t_{j,e} = n^{-3/2} i_{j,e}$. For other $i$, 
the required probability is $1 - N_i/Q$, 
where $N_i$ is the number of ordered open pairs 
that cannot be selected at step $i$ on $\mc{E}$.
If $i$ is a selection step we write $N_i=0$. 
Then we estimate
\[\mb{P}(\mc{E} \wedge \mc{G}_{i'}) \le \prod_{j=1}^s \prod_{e \in J \sm J[A]} (1+o(1))2q(t_{j,e})^{-1} \cdot \prod_{i=1}^{i'} (1 - N_i/Q).\]

Now we estimate $N_i$ when $i$ is not a selection step, 
assuming that we are in the event $ \mc{G}_i$ and $i<I$. 
For $i < L^{-1/2} n^{3/2}$ we use the trivial estimate $N_i \ge 0$, 
so suppose $i \ge L^{-1/2} n^{3/2}$. 
Suppose there are $k_i$ choices of $(j,e)$ with $i_{j,e}>i$. 
Then there are $o(V)s + k_i$ open pairs that must not become closed, 
namely the $o(V)$ open pairs specified by each $f_1,\dots,f_s$ 
and the $k_i$ pairs that have yet to be selected as edges
(these pairs are distinct by disjointness 
of $f_1,\dots,f_s$ outside $\phi(A)$.)
By \eqref{eq:dfid} the number of choices of the selected edge $e_i$ 
that close more than one such open pair is $O(s^2 L^4) = o(syL^{-2})$, 
as by assumption on $\max_B S_B^{V_\GG}$ and choice of $s$
we have $s < y(t') L^{-7} < y(t) L^{-6.5}$. 

As $i<I$, by \eqref{eq:e*} all $Y$-variables are $(1+o(1))y$, 
so $N_i = (1+o(1))(o(V)s + k_i)\cdot 4y$. 
Similarly to the proof of  Lemma~\ref{no-sparse}, we write
\[ 1 - N_i/Q \le 1-(1+o(1))(A_i+B_i)
\le \exp \{ -(1+o(1))A_i \} \exp \{ -(1+o(1))B_i \}, \]
where $A_i=B_i=0$ for $i < L^{-1/2} n^{3/2}$
and otherwise $A_i = o(V)s \cdot 8tn^{-3/2}$ 
and $B_i = k_i \cdot 8tn^{-3/2}$. As before,
we estimate separately all $A_i$ terms 
and all $B_i$ terms to obtain
\[  \exp \left\{ - \sum (1+o(1)) A_i \right\} 
\le \left[ (1+o(1)) \hat{q}(t')^{o(V)}\right]^s, \text{ and } \]
\[  \exp \left\{ - \sum (1+o(1)) B_i \right\} 
\le \prod_{j=1}^s \prod_{e \in J \sm J[A]} 
(1+o(1)) \hat{q}(t_{j,e}), \text{ so } \]
 \[\mb{P}( \mc{E} \wedge \mc{G}_{i'}) 
 \le \hat{q}(t')^{o(V)s} \prod_{j=1}^s 
 \brac{(1+o(1)) \prod_{e \in J \sm J[A]} 2n^{-2}}.\]
Summing over at most $s!^{-1} n^{n(V)s}$ choices for $f_1,\dots,f_s$ 
and $(i')^{e(V) s}$ choices for the selection steps,  
the probability that such $f_1,\dots,f_s$ exist 
is at most $s!^{-1} [(1+o(1))v(t')]^s < (3s^{-1}v(t'))^s$, 
which is subpolynomial.

The required bound on $X_{\phi,J,\GG}(i')$ follows from this estimate
by induction as in the proof of Lemma \ref{Gnpbound}. 
(The base case $|V_J \sm A|=1$ holds by the bound on disjoint extensions,
and for $|V_J \sm A| > 1$ the bound follows by considering a maximal 
collection $C$ of extensions disjoint outside of $A$
-- we have just shown whp $|C| \le s$ --
noting by strict balance and the induction hypothesis 
that at most $L^{4|V_J \sm A|-1}$ embeddings intersect
some embedding in $C$ outside of $\phi(A)$.)
This completes the proof when $(A,J,\GG)$ is strictly balanced,
and as noted above, the general case follows by applying this 
to each step of the extension series.
\end{proof}

\begin{lemma} \label{extensions}
For any extension $(A,J,\GG)$ with $|V_\GG| \le M^3$,
whp we do not have $I=\I0=i$ due to some $\phi$ with
$X_{\phi,J,\GG} \ge L^{4|V_\GG \sm A|} 
\max_{A \sub B \sub V_\GG} S^{V_\GG}_B$.
\end{lemma}

\begin{proof}
By bounding each step of the extension series
we can assume that $(A,J,\GG)$ is strictly balanced.
If $S_A^{V_\GG}(t) < n^{\dD'}$ then the required bound
follows from Lemma~\ref{pre-extensions}. On the other hand,
if $S_A^{V_\GG}(t) \ge n^{\dD'}$ then $X_{\phi,J,\GG}$
is controllable at time $t$, so the required bound 
follows from $i < \I2$.
\end{proof}

\subsection{Vertex degrees} \label{sec:vxdeg}

Recall that we cannot apply our general strategy to vertex degree variables,
as $g_{Y_1}(t) y_1(t)$ is not approximately non-increasing.  
We conclude this section with a separate (much simpler) 
argument for these variables, which establishes 
property (vii) of $\mc{G}_i$ in Definition \ref{def:I0}.

\begin{lemma} \label{lem:vdeg}
whp we do not have $\I0=i'$ due to some $uv$ 
with $ |Y_u(i') - y_1(t') | \ge \dD_{Y_1}(t') y_1(t')$.
\end{lemma}

\begin{proof}
For each $1 \le i \le i'$, the probability that 
we choose an edge incident to $u$ is \[ \frac{2 X_u(i)}{Q(i)} 
= \frac{(1 \pm \dD_{X_1})2 x_1}{(1 \pm \dD_Q)q} 
= \Big( 1 \pm (1+o(1))\dD_{X_1} \Big) \frac{2}{n}.\]
By coupling, we can bound $Y_u(i')$ by
sums $\Ss^\pm$ of independent Bernoulli random variables 
with probabilities $(1 \pm 2\dD_{X_1})2/n$.
Now we recall from Definition \ref{def:c} that 
$c_{Y_1} = 2.2c_{X_1}$, and note that $f_{Y_1} = 2.2 f_{X_1}$ 
and $g_{Y_1}/g_{X_1} = 2.2(1+t^{-1})/2 > 1.1$,
so $\dD_{ Y_1} > 1.1\dD_{X_1}$.
Thus on the event $ |Y_u(i') - y_1 | \ge \dD_{Y_1} y_1 $ 
one of $\Ss^\pm$ deviates from its mean $(1 + o(1)) 2tn^{1/2}$
by more than $\dD_{Y_1} y_1/100 > L^{13} n^{1/4}$. 
By Chernoff bounds, whp this does not occur for any vertex $u$.
\end{proof}

\section{Global Ensemble} \label{sec:global}

In this section we prove that the global variables 
have the desired concentration, assuming that this is the case  
for all ensembles at earlier times. 
Recall that the global variables are 
the number $Q(i)$ of ordered open pairs in $ G(i)$, 
the number of ordered triples $R(i)$
where all the pairs within the triple are open, 
and the number $ S(i) $ of ordered triples $abc$ 
such that $ab$ is an edge while $bc$ and $ac$ are open pairs. 
The global variables have scalings  $q = \hat{q} n^2$, 
$r= \hat{q}^3 n^3 $ and $ s = 2t  \hat{q}^2n^{5/2} $.
Recall that we track each variable $V$ relative to 
a {\em tracking random variable} $\mc{T} V$ to isolate variations in $V$ 
from variations in other variables that might have an impact on $V$.
We use the tracking variables
\[\mc{T}Q = q, \quad \mc{T}R = Q^3 n^{-3}, \quad \text{ and } \quad
\mc{T}S = 2t n^{-3/2} Q^2. \]
(Note that the tracking variable for $Q$ is a deterministic function.)

We show that the {\em difference random variables}
\[ \mc{D}V = V - \mc{T}V \]
for $V \in \{Q,R,S\}$ are all small throughout the process.
Recall that $\I1$ is the minimum of the stopping times $I_V$ 
over all variables $V$ in the global ensemble,
i.e.\ the first time at which some global variable $V$
(is good and) fails to satisfy $|\mc{D}V| \le \dD_V v$.
(Global variables are automatically good,
so we can ignore that part of the definition.)
The following theorem bounds the probability that 
we reach the universal stopping time $I$ before step $ i_{max} $ 
because a global variable $V$ fails to satisfy 
the required bounds $|\mc{D}V| \le \dD_V v$.  
\begin{theo}\label{mainensemble}
With high probability we do not have $I=\I1 \le i_{max}$.
\end{theo}

We prove Theorem~\ref{mainensemble} using the strategy
described in Section~\ref{sec:strategy}.  
We divide the argument into three parts,
in which we respectively bound the one-step expected changes in the difference variables,  determine variation equations that suffice to
establish the trend hypothesis, and verify the boundedness hypothesis.

\subsection{One-step changes in the difference variables} 

In this subsection, for each variable $V$ in the Global Ensemble, we give
an upper bound on the one-step expected change in the difference variable,
conditional on the history of the process, i.e.\
\[ \mb{E} [ \DD_i \mc{D}V \mid \mc{F}_i ] 
= \mb{E}[ \mc{D}V(i+1) - \mc{D}V(i) \mid \mc{F}_i], \]
under the assumption that $V$ is in its upper critical window, i.e.\
\[(f_V+g_V) v < \mc{D}V < (f_V+2g_V) v.\]
Recall that we can assume $n^{5/4} \le i < I$, 
so we can apply the estimates from $\mc{G}_i$
in Definition \ref{def:I0} and the bounds 
$V =  \mc{T}V \pm \dD_V v$ for any
variable $V$ if $i_V \le i \le J_V$ (i.e.\ if $V$ is good and activated).  
To illustrate later calculations, which are often more complicated 
than those in this section, as we proceed we will indicate 
how certain specific calculation
are instances of a more general framework.

We will consider the effect of each open pair 
and edge in the structure counted by $V$ separately;
the final expression is then obtained by linearity of expectation.
When an open pair in a copy of the structure counted by $V$ 
is chosen or closed, we say that the copy is {\em destroyed}.
We balance the change in $V$ due to destructions 
with the change in $\mc{T}V$ due to the change in $Q$.
(The case $V=Q$ is handled differently as $Q$ is tracked 
relative to the deterministic function $q$.)
Adding the edge $e_{i+1}$ can also create new copies of the structure counted 
by $V$ in which $e_{i+1}$ plays the role of one of the edges in the structure; 
then we say that a copy of $V$ is {\em created}
(for global variables this only applies to $V=S$).
The change in $V$ that comes from creations is balanced 
with the change in $t$ in $ \mc{T}V$.

We begin with destructions. The main point to note in these calculations is 
that the assumption that $V$ is in its critical window gives 
a self-correction term of $- 8tf_V vn^{-3/2}$ for each open pair, 
which will cancels with a corresponding $8tf_V vn^{-3/2}$ 
term from the change in $\dD_V v$;
this arises from the critical window excess of $f_V v$ 
in $V$ relative to $\mc{T} V$, recalling from \eqref{eq:close}
that in each such `excess copy' of $V$ the corresponding open pair
becomes closed with probability about $8tn^{-3/2}$.

\subsubsection{$Q$:\ simple destructions} 

We will show the following estimate for
the expected one-step change in $Q$.

\begin{lemma} \label{lem:Qdestruct}
If $n^{5/4} \le i < I$ and $Q \ge (1+f_Q+g_Q)q$ then
\[\mb{E} [ \DD_i(\mc{D}Q) \mid \mc{F}_i ] 
\le -(f_Q+g_Q - (1+o(1))\dD_S) 8tqn^{-3/2}.\]
\end{lemma}

For the variable $Q$ there is another variable $S$ in our ensemble
that counts situations when some open pair counted by $Q$ is closed. 
We call destructions of this form {\em simple destructions}.  
(We will see examples of this type again in Section~\ref{sec:stack} 
where we treat the stacking variables.) 

\begin{proof}
Each triple in $S$ contains $4$ ordered open pairs, 
each of which would decrease $Q$ by $2$ ordered pairs 
if selected as the edge at step $i$, and by symmetry in $S$
we count each of these possibilities twice.
The selected edge itself also removes $2$ ordered open pairs, so
\begin{equation}
\label{eq:Qstar}
 \mb{E} [\DD_i Q \mid \mc{F}_i] = - 2 - 4S/Q. 
\end{equation}
Recalling from Lemma \ref{lem:seconds} that 
$\DD_i(q) = -8t qn^{-3/2} + O(qn^{-5/2})$, we calculate
\begin{align*}
\mb{E} [ \DD_i(\mc{D}Q) \mid \mc{F}_i ] 
& = \mb{E} [ \DD_i(Q) - \DD_i(q) \mid \mc{F}_i ] \\
& = -( 2  + 4S/Q) + 8t q n^{-3/2} \pm O(qn^{-5/2}) \\
& = - 8t Q n^{-3/2} \pm (8+O(\dD_Q)) \dD_S t n^{1/2} \hat{q} 
+ 8t q n^{-3/2} \pm O(1) \\
& \le -(f_Q+g_Q - (1+o(1))\dD_S) 8tqn^{-3/2}.
\end{align*}
In the third estimate we used
$S = (1 \pm \dD_S) \mc{T}S = (1 \pm \dD_S) 2tn^{-3/2}Q^2$
and $Q = (1 \pm \dD_Q)q$, which are valid as 
$n^{5/4} = i_S = i_Q \le i < I$.
In the last line we used $\mc{D}Q = Q-q \ge (f_Q+g_Q)q$ 
when $Q$ is in its upper critical window, 
and $\dD_S \ge g_S \ge c_S L^{-1} e^2 t^{-1}$,
where $c_S = 2L^{40}$ (see Definition \ref{def:c}), so
$t\dD_S qn^{-3/2} \ge L^{-1} c_S e^2 qn^{-3/2} = L^{-1} c_S \gg 1$.
\end{proof} 

\subsubsection{$R$:\ product destructions}

We will show the following estimate for
the expected one-step change in $R$.

\begin{lemma} \label{lem:Rdestruct} 
If $n^{5/4} \le i < I$ and $R \ge (1+f_R+g_R)\mc{T}R$ then
\[\mb{E} [ \DD_i(\mc{D}R) \mid \mc{F}_i ] 
\le \left[ - (3+o(1)) ( f_R + g_R + O(\dD_Y \dD_X) 
+ O(t^{-1} e^2) + O(L^{16} t^2 n^{-1/2}) \right] 8trn^{-3/2}.\]
\end{lemma}

The destructions for $R$ are not simple destructions, 
as no variable in our ensembles counts ways 
in which triples counted by $R$ are destroyed.
Instead, we will apply the Product Lemma (Lemma \ref{lem:pat}).
For clarity we will write out the calculation separately for $R$ and $S$
(in later sections we will be more efficient by introducing extra notation
that unifies all cases).

\begin{proof}
To estimate the expected change, we first 
recall from \eqref{eq:close} that any pair $\aA\bB \in Q(i)$ 
becomes closed with probability $2(1+Y_{\aA\bB}+Y_{\bB\aA})/Q$.
Noting that closing $\aA\bB$ reduces $R$ by $3X_{\aA\bB}$, we write
\[ \mb{E}[\DD_i(R) \mid \mc{F}_i] = - \sum_{\aA \bB \in Q } 
 2Q^{-1}(1 + Y_{\aA \bB} + Y_{\bB \aA}) \cdot 3X_{\aA \bB}
 + \mb{E}[F_i(R) \mid \mc{F}_i],\]
where $F_i(R)$ is a `destruction fidelity' correction term
to remove overcounting of triples in $R$ for which the selected edge
closes two open pairs in the triple. 
Thus $\mb{E}[F_i(R) \mid \mc{F}_i] = F^*/Q$, where $F^*$ is the number of 
ordered quadruples where two adjacent pairs are edges and the other 
four pairs are open. As $i<\I0$, by property (iv) of $\mc{G}_i$ in
Definition \ref{def:I0} we have 
$F^* < L^{16} n^4 p^2 \hat{q}^4 = 4L^{16} t^2 \hat{q} r$, so
\begin{equation} \label{eq:Rfidel}
\mb{E}[F_i(R) \mid \mc{F}_i] = O(L^{16} t^2 r/n^2).
\end{equation}
Next, noting that
\[ \sum_{\aA \bB \in Q} Y_{\aA \bB} = S 
\ \ \   \text{ and } \ \ \ 
\sum_{ \aA \bB \in Q} X_{ \aA \bB} = R, \]
we estimate the main term using the Product Lemma as
\begin{equation} \label{eq:Rprod}
- \sum_{\aA \bB \in Q } 
 6Q^{-1}(1 + Y_{\aA \bB} + Y_{\bB \aA}) X_{\aA \bB}
 = - 12SRQ^{-2} \pm O(\dD_Y y \dD_X x) \pm O(x),
\end{equation}
where as $n^{5/4} \le i < I$ we have the estimates $X = (1 + O(\dD_X))x$ 
and all $Y$-variables are 
$(1 \pm \dD_Y^*)y = (1+O(\dD_Y))y$ from \eqref{eq:e*}.
The important point to observe regarding the product error term 
is that $\dD_X$ and $t \dD_Y$ are $\wt{O}(e)$,
whereas $\dD_R$ is $\wt{O}(e^2)$, so the error term
is negligible for appropriate choices
of the polylogarithmic constants $c_X$, $c_Y$ and $c_R$
(see Definition \ref{def:c}).

Next we consider the expected change in the tracking variable
$\mc{T}R = Q^3 n^{-3}$. We have 
\begin{equation} \label{eq:Rtrack0}
\DD_i(\mc{T}R) = Q(i+1)^3 n^{-3} - Q(i)^3 n^{-3}
= 3\DD_i(Q) Q^2 n^{-3} + H_i(R),
\end{equation}  
where $H_i(R)$ is a `higher order' term correcting
for the linear approximation of the difference in $Q^3$,
and as $\DD_i(Q) = O(y)$ we have 
$H_i(R) = (3\DD_i(Q)^2 Q + \DD_i(Q)^3)n^{-3} 
= O(t^2 r n^{-3})$.
By \eqref{eq:Qstar} we have
\begin{equation} \label{eq:Rtrack}
\mb{E}[\DD_i(\mc{T}R) \mid \mc{F}_i] 
= - 12SQ^{-2} \mc{T}R + O(t^2 r n^{-3}). 
\end{equation}

Combining \eqref{eq:Rfidel}, \eqref{eq:Rprod} 
and \eqref{eq:Rtrack} gives
\begin{align*}
& \mb{E} [ \DD_i(\mc{D}R) \mid \mc{F}_i ] 
  = \mb{E} \left[ \DD_i(R) - \DD_i(\mc{T}R) \mid \mc{F}_i \right] \\
&  =  - \sum_{\aA \bB \in Q} 
 6Q^{-1}(1 + Y_{\aA \bB} + Y_{\bB \aA}) X_{\aA \bB}
 - 3Q^2 n^{-3} \mb{E}[\DD_i(Q) \mid \mc{F}_i] 
 + \mb{E} [F_i(R) - H_i(R) \mid \mc{F}_i ] \\
& =   - 12SQ^{-2} R \pm O(\dD_Y y \dD_X x) + O(x)
 + 12SQ^{-2} \mc{T}R + O(L^{16} t^2 r/n^2)  \\
& = - ( 1 \pm (3+o(1))\dD_S) 8tn^{-3/2} \mc{D}R 
\pm O( \dD_Y \dD_X ) trn^{-3/2} \pm O(r/q) + O(L^{16} t^2 r/n^2) \\
& \le \left[ - (3+o(1)) (f_R+g_R) + O(\dD_Y \dD_X) 
+ O(t^{-1} e^2) + O(L^{16} t^2 n^{-1/2}) \right] 8trn^{-3/2}. 
\end{align*}
An important point to note in the above calculation is that
the same factor $12SQ^{-2}$ appears with $R$ and $\mc{T}R$,
and that we approximate $S$ by $(1 \pm \dD_S)\mc{T}S$ only after using 
the critical window bound $\mc{D}R = R-\mc{T}R \ge (f_R+g_R)r$;
thus the fact that our approximation of $S$ is weaker than that of $R$
does not cause any difficulty in this calculation for $R$.
\end{proof}

\subsubsection{$S$:\ product destructions and creations}

For $S$ we have both creations and destructions,
so we will now elaborate on how we group the calculations 
for each edge of a structure (we could gloss over this for $R$,
as it has $3$ indistinguishable edges, but it will be important
for most other variables, including $S$).
Recall that $S$ is the number of ordered triples $abc$
where $ab$ is an edge and $ac,bc$ are open pairs. We write 
\[\DD_i(S) = \DD_i(S^{12}) + \DD_i(S^{13}) + \DD_i(S^{23}),\]
where we think of $123$ as labelling each such $abc$,
and each $\DD_i(S^e)$ is the change in $S$ due to $e$,
i.e.\ $\DD_i(S^{12})$ is the number of triples $abc$ in $S$
created due to $ab$ being the edge selected at step $i$,
$-\DD_i(S^{13})$ is the number of triples $abc$ in $S$
destroyed due to $ac$ being selected or closed at step $i$,
and similarly for $-\DD_i(S^{23})$. 

Usually, we would also include
a `fidelity' term $F_i(S)$ in this decomposition of changes by edges,
reflecting the fact that the selected edge might affect more 
than one pair in a triple counted by $S$, but in fact this is not 
possible, so we can set $F_i(S)=0$. Indeed, if selecting the 
edge $ab$ creates a triple $abc$ in $S$ then by definition of $S$
it does not close $ac$ or $bc$, and a triple $abc$ cannot be destroyed
by some edge $e_i$ that simultaneously closes $ac$ and $bc$, 
as this would require $e_i=cd$ such that $ad$ and $bd$ are edges,
but then $abd$ would be a triangle, which is impossible.

We also decompose the change in the tracking variable $\mc{T}S$ 
into terms that we assign to the different parts of the calculation
corresponding to each of the edges in $S$. 
Recalling that $\mc{T}S = 2tn^{-3/2} Q^2$, 
we have $\DD_i(\mc{T}S) 
= 2(t+n^{-3/2})n^{-3/2}(Q + \DD_i(Q))^2
- 2tn^{-3/2} Q^2$, which we write as 
\[ \DD_i(\mc{T}S) = \DD_i(\mc{T}S^{12}) 
+ \DD_i(\mc{T}S^{13}) + \DD_i(\mc{T}S^{23}) + H_i(S),\]
where $\DD_i(\mc{T}S^{12}) 
= 2n^{-3} Q^2 = \mc{T}S / tn^{3/2}$ and
$\DD_i(\mc{T}S^{13}) = \DD_i(\mc{T}S^{23}) 
= 2tn^{-3/2} \DD_i(Q) Q = \tfrac{\DD_i Q}{Q} \mc{T}S$,
with the higher-order correction term
$H_i(S) = 2n^{-3}(2\DD_i(Q)Q + \DD_i(Q)^2)
+ 2tn^{-3/2} \DD_i(Q)^2 = O(yqn^{-3}) + O(tn^{-3/2}y^2)
= O(s n^{-3} ) + O(t^2 s n^{-3})$.

Now we show the calculations for the change
$\DD_i(\mc{D}S^{13}) := \DD_i(S^{13}) - \DD_i(\mc{T}S^{13})$
(the one with $23$ instead of $13$ is the same);
these are product destructions very similar to those for $R$.

\begin{lemma} \label{lem:Sdestruct} 
If $n^{5/4} \le i < I$ and $S \ge \mc{T}S + (f_S+g_S)s$ then
\[\mb{E} [ \DD_i(\mc{D}S^{13}) \mid \mc{F}_i ] 
\le \left[ - (1+o(1)) ( f_S+g_S) + O(\dD_Y^2) 
+ O(t^{-1} e^2) \right] 8tsn^{-3/2}.\]
\end{lemma}

\begin{proof}
Similarly to the proof of Lemma \ref{lem:Rdestruct}, we calculate
\begin{align*}
& \mb{E} [ \DD_i(\mc{D}S^{13}) \mid \mc{F}_i ]  = 
\mb{E} \left[ \DD_i(S^{13}) 
- \frac{\DD_i(Q)}{Q} \mc{T}S \mid \mc{F}_i \right] \\
& = - \sum_{\aA \bB \in Q } 2Q^{-1}(1 + Y_{\aA \bB} + Y_{\bB \aA}) Y_{\aA \bB} 
+ (2 + 4SQ^{-1}) Q^{-1} \mc{T}S \\
& = - 4SQ^{-2} S \pm  O(\dD_Y y)^2 + O(y) + 4SQ^{-2} \mc{T}S \pm O(s/q) \\
& = ( 1 \pm (1+o(1))\dD_S) 8tn^{-3/2} \mc{D}S 
\pm O( \dD_Y^2) tsn^{-3/2} \pm O(s/q) \\
& \le \left[ - (1+o(1)) ( f_S+g_S) + O( \dD_Y^2) + O(t^{-1} e^2) \right] 8tsn^{-3/2}.
\end{align*}
\end{proof}

Finally, we turn to creations, which among the global variables 
occur only for $S$.

\begin{lemma} \label{lem:Screate} 
If $n^{5/4} \le i < I$ then $\mb{E} [ \DD_i(\mc{D}S^{12}) \mid \mc{F}_i ] 
\le (1+o(1)) \tfrac{\dD_R}{8t^2} 8tsn^{-3/2}$.
\end{lemma}

\begin{proof}
We have $\mb{E} [ \DD_i(S^{12}) \mid \mc{F}_i ] = 2R/Q$, 
as for each triple $abc$ in $R$, with probability $2/Q$
the edge $e_{i+1}$ selected at step $i+1$ falls in position
$ab$ and turns $abc$ into a triple in $S$. Thus
\begin{align*}
& \mb{E} [ \DD_i(\mc{D}S^{12}) \mid \mc{F}_i ] 
 = \mb{E} [ \DD_i(S^{12}) - t^{-1}n^{-3/2} \mc{T}S \mid \mc{F}_i ] \\
& = 2Q^{-1} ( \mc{T}R \pm  \dD_R r ) -  t^{-1}n^{-3/2} \mc{T}S \\
& = \pm 2 \dD_R r Q^{-1} 
 = \pm (1+o(1)) t^{-1} \dD_R sn^{-3/2}.
\end{align*}
Note that there is no self-correction in creation,
but this term will be negligible as our approximation
of $R$ is better than that of $S$.
\end{proof}

\subsection{Trend Hypothesis and Variation Equations}

For each variable $V$ in the Global Ensemble 
we consider the sequence of random variables
\[ \mc{Z}V(i) = \mc{D}V - v \dD_V. \]
The following lemma establishes the trend hypothesis,
i.e.\ that this sequence is a supermartingale
when $V$ is in its upper critical window.
During the proof we will derive the Variation Equations,
which give conditions on the constants $c_V$ 
under which the trend hypothesis holds;
we will see that these conditions are satisfied
by the choices in Definition \ref{def:c}.

\begin{lemma} \label{lem:globaltrend}
For each  $V \in \{Q,R,S\}$, 
if $n^{5/4} \le i < I$ and $\mc{D}V > (f_V+g_V)v$ 
then $\mb{E} [ \DD_i \mc{Z}V \mid \mc{F}_i]  \le 0$.
\end{lemma}

\begin{proof}
We begin by gathering together the 
relevant creation and destruction calculations 
from the previous subsections; these are obtained by combining 
Lemmas \ref{lem:Qdestruct}, \ref{lem:Rdestruct}, \ref{lem:Sdestruct}
and \ref{lem:Screate}.
\begin{eqnarray*}
\mb{E} [ \DD_i \mc{D} Q \mid \mc{F}_i ] 
&\le& - \brac{ f_Q+g_Q - (1+o(1))\dD_S )} 8t qn^{-3/2}, \\
\mb{E} [ \DD_i \mc{D} R \mid \mc{F}_i ] &\le& - (1+o(1)) \bracsq{ 3(f_R+g_R) 
- O(\dD_Y \dD_X) - O(t^{-1} e^2)} 8t rn^{-3/2}, \\
\mb{E} [ \DD_i \mc{D} S \mid \mc{F}_i ] &\le& - (1+o(1)) \bracsq{ 2(f_S+g_S) - \tfrac{\dD_R}{8t^2} - O(\dD_Y^2) - O(t^{-1} e^2)} 8t sn^{-3/2}.
\end{eqnarray*}
For $R$ we have omitted the fidelity term in \eqref{eq:Rfidel};
this is valid as $F_i(R) = O(L^{16} t^2 r n^{-2}) = o(g_R) trn^{-3/2}$,
where we recall from Definition \ref{def:c} that
\begin{equation} \label{eq:cR}
c_R = L^{40} \gg L^{20} \text{ (say)}.
\end{equation}
Next we consider the change in $v\dD_V$.
From Lemma \ref{lem:seconds} we have
\[ \DD_i(v\dD_V) 
= \left( \tfrac{e(V)}{8t^2} - o(V)  \right) \dD_V v \cdot 8tn^{-3/2} 
+ \dD'_V vn^{-3/2} + O(\dD_V v) n^{-5/2}. \]
Recalling that $\dD_V = f_V + 2g_V$, 
we see that we can cancel the $8t o(V) f_V vn^{-3/2}$ term 
that occurs both in $\DD_i(\dD_V v)$ 
and in $\mb{E} [ \DD_i \mc{D} V \mid \mc{F}_i ] $; 
this is the self-correction that is fundamental to the analysis.

Thus we obtain
\begin{eqnarray*}
\mb{E} [ \DD_i \mc{Z} Q \mid \mc{F}_i ] &\le& - \brac{ \frac{ \dD'_Q}{8t} + o(f_Q) - (1+o(1))( g_Q + \dD_S )}  8t qn^{-3/2}, \\
\mb{E} [ \DD_i \mc{Z} R \mid \mc{F}_i ] &\le& - \brac{ \frac{ \dD'_R}{ 8t} + o(f_R) - (1+o(1))  3g_R + O(\dD_Y \dD_X)  - O(t^{-1} e^2) } 8t rn^{-3/2}, \\
\mb{E} [ \DD_i \mc{Z} S \mid \mc{F}_i ] &\le&
  - \brac{ \frac{\dD'_S}{8t} + \frac{\dD_S}{8t^2} + o(f_S) - (1+o(1)) ( 2g_S + \tfrac{\dD_R}{8t^2}) + O(\dD_Y^2)  - O(t^{-1} e^2) } 8t sn^{-3/2}.
\end{eqnarray*}
Recall that our error functions have the form $\dD_V = f_V + 2g_V$, where
\[  f_V = c_V e^2 \ \  \text{ and } \ \ 
g_V = c_V \tT L^{-1} (1+t^{-e(V)}) e^2
\ \ \text{ if } \ V \in \{Q,R,S\}. \]
We now show that these error functions grow quickly enough 
for each of these sequences to be supermartingales 
(i.e.\ the $ \dD_V'$ term will be dominant in each case).  
We stress that the $ t \ll 1 $ regime behaves a bit differently 
from the rest of the process in the estimates that follow.
For each global variable in turn we apply the bound on $\dD'_V$
from Lemma \ref{lem:seconds}, i.e.\
\[ \dD'_V \ge 8t \dD_V + (\tT'/\tT - e(V)t^{-1}) 2g_V. \]
For $Q$ we have
\begin{eqnarray*}
\mb{E} [ \DD_i \mc{Z} Q \mid \mc{F}_i ]
&\le&  - (1+o(1)) \bracsq{  ( f_Q + (\tfrac{2 \vartheta'}{8t\vartheta}  + 2)g_Q ) - ( g_Q + \dD_S ) } 8t qn^{-3/2} \\
&\le& - (1+o(1)) \bracsq{ (f_Q-f_S) + (\tfrac{\vartheta'}{4t\vartheta}  g_Q + g_Q - 2g_S ) } 8t qn^{-3/2}.
\end{eqnarray*}
Then the sequence $\mc{Z}Q$ forms a supermartingale provided 
\begin{equation}
\label{eq:varQ}
c_Q \ge 2c_S.
\end{equation}
Indeed, then the dominant terms are $-f_Q$ for $t \ge 1$
and/or $- \tfrac{\vartheta'}{4t\vartheta}  g_Q $ for $t \le 1$
(for $t \le 1$ we recall that $\tT'/\tT = (3/\eps)^6$ and note that 
the $t^{-1}$ in $\tfrac{\vartheta'}{4t\vartheta} g_Q$
matches the $t^{-1}$ in $g_S$).

Next consider $R$, where we have
\begin{eqnarray*}
\mb{E} [ \DD_i \mc{Z} R \mid \mc{F}_i ]
&\le& - (1+o(1)) \bracsq{ f_R + (\tfrac{ 2\vartheta'}{8t \vartheta} + 2) \cdot g_R  -  3g_R - O(\dD_Y \dD_X) - O(t^{-1} e^2) } 8t rn^{-3/2} \\
&\le& - (1+o(1))  \bracsq{ f_R + (\tfrac{ \vartheta' }{4t \vartheta} - 1)g_R - O(f_Y f_X) - O(g_Y f_X) } 8t rn^{-3/2},
\end{eqnarray*}
as $t^{-1} e^2 \ll t^{-1} g_R$.
Then $\mc{Z}R$ forms a supermartingale provided
\begin{equation}
\label{eq:varR}
c_R \ge L c_Yc_X,
\end{equation}
for this implies that the $ g_R \vartheta'/(4t \vartheta) $ term dominates 
for $ f_R < g_R/t  $ and that the $f_R$ term dominates otherwise.
As noted earlier, we chose powers of $e$ in the error functions
so that $\dD_R$ and the product error $t\dD_Y \dD_X$ 
are comparable up to log factors
(i.e.\ $e$ in $\dD_X$ and $\dD_Y$ and $e^2$ in $\dD_R$);
then the choice of polylogarithmic constants $c_V$ 
in Definition \ref{def:c} was such that \eqref{eq:varR} holds.

The final global variable is $S$, where we have
\begin{eqnarray*}
\mb{E} [ \DD_i \mc{Z} S \mid \mc{F}_i ]
&\le&  -  (1+o(1)) \bracsq{ (1 + \tfrac{1}{8t^2})f_S + (\tfrac{2 \vartheta'}{8t \vartheta} + 2) \cdot g_S}  8t sn^{-3/2} \\
& & \hskip1cm
   + (1+o(1)) \bracsq{ 2g_S + \tfrac{\dD_R}{8t^2} + O(\dD_Y \dD_Y)  + O(t^{-1} e^2) } 8t sn^{-3/2} \\
&\le& - (1+o(1))  \bracsq{ f_S + \tfrac{ f_S - f_R}{8t^2} + \tfrac{ \vartheta' t g_S/ \vartheta - g_R}{4t^2}
- O(f_Y^2) - O(g_Y^2) + o(g_S) } 8t sn^{-3/2}.
\end{eqnarray*}
Then $\mc{Z}S$ forms a supermartingale provided
\begin{equation}
\label{eq:varS}
c_S \ge 2c_R \ \ \ \ \text{ and } \ \ \ \ c_S \ge L c_Y^2,
\end{equation}
for this implies that the $ f_S/t^2 $ term dominates for $ t \ll 1 $ 
and the $ f_S$ term dominates otherwise.
\end{proof}

\subsection{Boundedness hypothesis}

For the boundedness hypothesis, for each $V$ in the Global Ensemble we estimate 
$\Var_V =  \Var(\mc{Z}V(i) \mid \mc{F}_{i-1})$ and $N_V = |\DD_i\mc{Z}V|$.  
Recall that it suffices to establish (\ref{eq:boundVar}) and (\ref{eq:boundN}); 
that is, it suffices to show the following lemma.

\begin{lemma} \label{lem:globalbound}
For each  $V \in \{Q,R,S\}$, if $n^{5/4} \le i < I$ then
$\Var_V = o\left( \frac{ (g_V v)^2}{L^3n^{3/2}} \right)$
and $N_V = o\left( \frac{g_V v}{L^2} \right)$. 
\end{lemma}

\begin{proof}
For convenience we replace $\mc{Z}V$ by $\mc{D}V$ in our calculations, 
as this does not change $\Var_V$ and only changes $N_V$ 
by an additive term which we can bound by $ O( n^{-5/4} v \dD_V ) $. 

For one-step variances we use the simple estimate $\Var_V \le N_V^2$
(so for the global variables we do not need the full power of Freedman's
inequality:\ it suffices to apply the Hoeffding-Azuma inequality). 

For $Q$ we have $g_Q q \ge c_Q L^{-1} n^{3/2}$, 
so it suffices to show $\Var_Q = o( c_Q^2 L^{-5} n^{3/2}) $ 
and $N_Q = o( c_Q L^{-3} n^{3/2})$. 
The change in $\mc{D}Q$ when the process chooses the edge $e_{i+1}=uv$ is
\begin{equation*}
 \DD_i \mc{D} Q = 2(Y_{uv} + Y_{vu} + 1) - \DD_i( q)  
 = 4( y \pm y \dD_Y) - 4y +O(1) =  O( y \dD_Y) = O(c_Y L n^{1/4}). 
\end{equation*}
Then $N_Q = \wt{O}(n^{1/4})$, $Var_Q = \wt{O}(n^{1/2})$,
and the required bounds hold easily.

For $R$ we have $g_R r \ge c_R L^{-1} \hat{q}^2 n^{5/2}$, so it suffices to show
$\Var_R = o \left(  c_R^2 L^{-5} \hat{q}^4 n^{7/2} \right) $ 
and $N_R = o \left( c_R L^{-3} \hat{q}^2 n^{5/2} \right)$.
Recall from \eqref{eq:Rtrack0} that
$\DD_i(\mc{T}R) = 3\DD_i(Q) Q^2 n^{-3} + H_i(R)$,
where $H_i(R) = O(t^2 r n^{-3})= \wt{O}(1)$.
On choosing $e_{i+1}=uv$ we have
\[ \DD_i R = F_i(R) - \sum_{ab \in Y_{uv} \cup Y_{vu} \cup \{uv\}} 6 X_{ab}, \]
where, as in the proof of Lemma \ref{lem:Rdestruct}, $F_i(R)$ is a 
`destruction fidelity' correction term to remove overcounting of triples 
in $R$ for which the selected edge closes two open pairs in the triple. 
We can bound $F_i(R)$ by the number of triples $uab$ counted by $R$
such that $va$ and $vb$ are edges (and similarly interchanging $u$ and $v$).
As $n^{5/4} \le i < I$, by property (iv) of $\mc{G}_i$ in Definition \ref{def:I0}
we have $F_i(R) = \wt{O}(1+n \hat{q}^3)$.
Combining these estimates gives 
\begin{equation*}
\begin{split}
\DD_i \mc{D}R & =   \DD_i R - 3 \frac{\DD_i(Q)}{Q} \mc{T}R + \wt{O}(1) \\
 & =   - \sum_{ab \in Y_{uv} \cup Y_{vu} \cup \{uv\}} 6 X_{ab}   - 2( Y_{uv} + Y_{vu} + 1) \cdot 3Q^2n^{-3}   + \wt{O}(1+n \hat{q}^3) \\
 & =    - 6 \left[  \sum_{ab \in Y_{uv} \cup Y_{vu} \cup \{uv\}} (X_{ab}   -  Q^2n^{-3}) \right]   + \wt{O} (y +n \hat{q}^3) \\
& =  O( y x \dD_X ) + \wt{O} (y+n \hat{q}^3)  
= \wt{O} ( \hat{q}^{5/2} n ^{5/4}). 
\end{split}
\end{equation*}
Then $N_R = \wt{O}(\hat{q}^{5/2} n^{5/4})$, 
$Var_R = \wt{O}(\hat{q}^{5} n^{5/2})$, 
and the required bounds hold easily.

For $S$ we have $g_S s \ge c_S L^{-1} \hat{q} n^2$, 
so it suffices to show
$\Var_S = o(  c_S^2 L^{-5} \hat{q}^2 n^{5/2}) $ 
and $N_S = o(  c_S L^{-3} \hat{q} n^2)$.
We bound the impact of creations and destructions separately,
recalling the decompositions of the change in $S$ as
$\DD_i(S) = \DD_i(S^{12}) + \DD_i(S^{13}) + \DD_i(S^{23})$,
where $\DD_i(S^{12})$ counts creations and
$\DD_i(S^{13})$, $\DD_i(S^{13})$ count destructions.
We also recall the corresponding decomposition 
of the change in the tracking variable as
$\DD_i(\mc{T}S) = \DD_i(\mc{T}S^{12}) 
+ \DD_i(\mc{T}S^{13}) + \DD_i(\mc{T}S^{23}) + H_i(S)$,
where $\DD_i(\mc{T}S^{12}) 
= 2n^{-3} Q^2 = \mc{T}S / tn^{3/2}$,
$\DD_i(\mc{T}S^{13}) = \DD_i(\mc{T}S^{23}) 
= 2tn^{-3/2} \DD_i(Q) Q = \tfrac{\DD_i Q}{Q} \mc{T}S$
and $H_i(S) = O((1+t^2)s n^{-3} )$.

On choosing $e_{i+1}=uv$, we estimate 
the destruction terms (e.g.\ that for $S^{13}$) by
\begin{align*}
 \DD_i \mc{D}S^{13} 
& = \DD_i S^{13} - 2 \frac{\DD_i(Q)}{Q} \mc{T}S 
+ O((1+t^2)s n^{-3} ) \\
& = - \sum_{ab \in Y_{uv} \cup Y_{vu} \cup \{uv\}} 
(Y_{ab} + Y_{ba}  -  2 \cdot 2tQn^{-3/2}) 
+ O((1+t^2)s n^{-3} ) \\
&  = O(y \cdot \dD_Y y) 
= \wt{O}( \hat{q}^{3/2} n^{3/4} ).
\end{align*}
For the creation term we have
\[ \DD_i \mc{D}S^{12} = \DD_i S^{12} - \mc{T}S/(tn^{3/2}) 
= 2X_{uv} - 2 Q^2n^{-3} = O(\dD_X x) 
= \wt{O}( \hat{q}^{3/2} n^{3/4} ).\]
The required bounds on $N_S$ and $Var_S$ hold easily.
\end{proof}

Having verified the trend and boundedness hypotheses
in Lemmas \ref{lem:globaltrend} and \ref{lem:globalbound},
Theorem \ref{mainensemble} now follows from
Lemmas \ref{th+bh} and \ref{lem:beginning}.

\section{The Controllable Ensemble}
\label{sec:control}

In this section we prove that all variables $V = X_{\phi,J,\GG}$
in the controllable ensemble have the desired concentration,
assuming that all variables in all ensembles
are well-behaved at earlier times. Recall that $\I2$ is the minimum of the 
stopping times $I_V$ over all variables $V$ in the controllable ensemble. 
The following theorem bounds the probability that we reach 
the universal stopping time $I$ before step $ i_{max}$ 
because some controllable variable $V$ is good
(see Definition \ref{def:IV}) but fails to satisfy 
the required bound $|\mc{D}V| \le \dD_V v$.

\begin{theo}\label{controllable}
With high probability we do not have $I=\I2 \le i_{max}$.
\end{theo}

\subsection{Preliminaries}

We start by recalling the definition of the ensemble.
We say $V=X_{\phi,J,\GG}$ is \emph{controllable at time $t'$} 
if $o(V)>0$ and for any $1 \le t \le t'$ we have
\begin{equation}
\label{eq:definecontrol}
S^B_A(J,\GG) \ge n^{\dD'} \ \ \text{ for all } \ \ A \subn B \sub V_\GG.
\end{equation}
The {\em controllable ensemble} consists of all such $V$
with $| V_\GG | \le M^3$ that are controllable at time $1$.

Next we record some preliminary observations.

\begin{lemma} \label{lem:obscon} 
Let $V$ be controllable at time $t'$. 
Then $v(t) \ge n^{\dD'}$ for $1 \le t \le t'$
and $S^B_A(J,\GG) \ge n^{\dD'}$ 
for all $A \subn B \sub V_\GG$
and $t_V \le t \le t'$. Furthermore, if
$V^+$ is obtained from $V$ by changing 
some edge to an open pair
then  $V^+$ is controllable at time $t'$.
\end{lemma}

\begin{proof}
The first inequality is immediate
from the definition with $B=V_\GG$.
The final statement holds as
for any $A \subn B \sub V_\GG$ 
we have $S^B_A(V^+)=S^B_A(V)$ 
or $S^B_A(V^+)=\hat{q}p^{-1}S^B_A(V) \ge S^B_A(V)$,
using \eqref{eq:q/p}.
For the remaining inequality,
consider any $A \subn B \sub V_\GG$.
By \eqref{eq:definecontrol} at $t=1$ we have
\[ n^{|B| - |A|} (2n^{-1/2} )^{ |J[B]| - |J[A]| } \ge n^{\dD'}, \]
so $|J[B]| - |J[A]| < 2 ( |B| - |A|)$. 
This gives the much stronger bound 
\[ n^{|B| - |A|} (n^{-1/2} )^{ |J[B]| - |J[A]| } \ge n^{1/2}, \]
so for $ t_V  \le t \le 1 $, recalling from Lemma \ref{lem:tV} 
that $t_V = \wt{\TT}(n^{-\dD/4e(V)})$, we have
$ S^B_A(J,\GG) 
= \OO \left(  n^{|B| - |A|} ( t_V n^{-1/2} )^{ |J[B]| - |J[A]| } \right) 
= \wt{\OO} \left( n^{1/2 - \dD/4} \right) > n^{1/4}$.
\end{proof}

It will be convenient to approximation $V$
by the following modified variable $V^*$
which has better behaviour for the martingale arguments. 

\begin{defn} \label{def:con*}
Consider a controllable variable $ V= X_{\phi,J,\GG}$. 
Given an injective map $f:V_\GG \to [n]$, 
we say that a pair $ab$ in $f(V_\GG)$ is \emph{$f$-open} 
if there is no vertex $c$ such that $ac$, $bc$ are edges 
and $c \notin f(V_\GG)$; note that it is the last condition 
that distinguishes the definition from that of `open'.
Let $V^* = X^*_{\phi,J,\GG}(i)$ be defined 
in the same way as $X_{\phi,J,\GG}(i)$,
except that pairs that are required to be open in $X_{\phi,J,\GG}(i)$ 
are only required to be $f$-open in $X^*_{\phi,J,\GG}(i)$.
\end{defn}

We will apply our usual martingale strategy to show whp
$V^* = (1 \pm \dD_{V^*}) v$ for $i_V \le i < I$, 
where $\dD_{V^*} = \dD_V-g_V/2 = f_V + 3g_V/2$; we recall
\[ e = \hat{q}^{-1/2} n^{-1/4}, \ \ \ 
f_V = e^\dD \ \ \  \text{ and } \ \ \ 
 g_V = \tT L^{-1} (1+t^{-e(V)}) e^{\dD}. \]
This will suffice in combination with the following
straightforward approximation of $V$ by $V^*$.

\begin{lemma} \label{lem:VV}
If $i_V \le i < I$ then $V = V^* \pm g_Vv/2$.
\end{lemma}

\begin{proof}
Fix $e \in \binom{V_\GG}{2} \sm \GG$ 
with $e$ not contained in the base $A$.
Let $J^e = J \cup \{e\}$ and $\GG^e = \GG \cup \{e\}$. 
We bound $|V-V^*|$ by the sum over all 
such $e$ of $X_{\phi,J^e,\GG^e}$.
As $i<I$, by property (iv) of Definition \ref{def:I0} 
we have $X_{\phi,J^e,\GG^e} \le L^{4|V_\GG|} 
S^{V_\GG}_A(J^e,\GG^e)/S^B_A(J^e,\GG^e)$,
where $B$ is chosen to minimise $S^B_A(J^e,\GG^e)$. 
For any $ A \subseteq B \subseteq V_\GG$,
if $B=A$ then $S^{B}_A(J^e,\GG^e)=1$;
otherwise, by controllability
$S^{B}_A(J^e,\GG^e) \ge p S^{B}_A(J,\GG) \ge p n^{\dD'}$.
As $ S^{V_\GG}_A(J^e,\GG^e) = p v$, it follows that
$X_{\phi,J^e,\GG^e} \le L^{ 4|V_\GG|} v n^{-\dD} \ll g_V v$,
as $ e^\dD > n^{ -\dD/4} $.
\end{proof}

\subsection{Decomposition by pairs}
\label{sec:decomp}

We decompose the one-step change in $V^*$ as 
\[ \DD_i(V^*) = \sum_{e \in \GG \sm \GG[A]} \DD_i(V^e) \pm F_i(V^*),\]
where each $\DD_i(V^e)$ accounts for the change in $V$ due to $e$,
as follows. If $e \in J$ then, letting $V^+$ be obtained from $V$
by changing $e$ from an edge to an open pair, $\DD_i(V^e)$ is the
number of embeddings $f \in (V^+)^*$ such that $f(e)$ is the edge
$e_{i+1}$ selected at step $i+1$.
If $e \in \GG \sm J$ then $-\DD_i(V^e)$ is the number 
of embeddings $f \in V^*$ which are destroyed at step $i+1$
by $f(e)$ not remaining $f$-open.
The fidelity term $F_i(V^*)$ is to correct for embeddings $f \in V^*$
where $f(e)$ is affected for more than one $e$ simultaneously.
Note that by definition of `$f$-open' this cannot occur for creation,
i.e.\ if $f(e)=e_{i+1}$ for some $e \in J \sm J[A]$;
thus $F_i(V^*)$ accounts for embeddings $f \in V^*$
where $f(e)$ becomes not $f$-open for more than one $e \in \GG \sm \GG[A]$.
This requires the selected edge $e_{i+1}$ to be $xy$ 
for some $x \in f(V_\GG)$ such that $y$ is a common neighbour
of some pair $u,v$ in $f(V_\GG)$.
As $i<I$, by property (iii) of Definition \ref{def:I0} 
all codegrees are $O(L^4)$, so
\begin{equation} \label{eq:confidel}
\mb{E} [ F_i(V^*) \mid \mc{F}_i ] = O(L^4) v/q.
\end{equation}

We also decompose the one step change in the tracking variable as
\[ \DD_i(\mc{T}V^*) = \sum_{e \in \GG \sm \GG[A]} \DD_i(\mc{T}V^e) \pm H_i(V^*),\]
where $\DD_i(\mc{T}V^e)$ is $\mc{T}V/(tn^{3/2})$ if $e$ is an edge
or $- \frac{\DD_i Q}{Q} \mc{T}V$ if $e \in \GG \sm J$ if $e$ is open,
and the higher-order correction term is 
\begin{equation} \label{eq:conhi}
H_i(V^*) = O((tn^{3/2})^{-1} + Q^{-1}\DD_i Q)^2 \mc{T}V 
= O(t^2+t^{-2}) n^{-3} v.
\end{equation}
Our calculations for the trend and boundedness hypotheses
will consider separately each
$\DD_i(\mc{D}V^e) := \DD_i(V^e) - \DD_i(\mc{T}V^e)$.

\subsection{One-step expected changes}

Here we estimate the one-step expected change in $V^*$ when
it is in its upper critical window.

\begin{lemma} \label{lem:DV*}
If $i_V \le i < I$ and $\mc{D}V^* > (f_V+g_V) v$ then 
\[  \mb{E} [ \DD_i(V^*) \mid \mc{F}_i ] \le  (1+o(1)) \left[
 \frac{e(V)}{8t^2} \dD_{(V^+)^*} - o(V) ( f_V + g_V - 2g_Y )
  \right] 8t v n^{-3/2}.\]
\end{lemma}

\begin{proof}
We estimate the one-step expected changes
$\mb{E} [ \DD_i(V^e) \mid \mc{F}_i ]$ 
for each $e \in \GG \sm \GG[A]$.

We start with creation, i.e.\ the case that $e \in J$ is an edge.
As for the global variables, we do not use the critical window assumption
or obtain any self-correction term in this calculation.
Writing $V^+ = X^*_{\phi,J\sm e,\GG}$, we have
\begin{equation*}
\begin{split}
\mb{E}[ \DD_i(\mc{D}V^e) \mid \mc{F}_i]
& = \mb{E} [\DD_i(V^e) - \DD_i(\mc{T}V^e) \mid \mc{F}_i] \\
& = 2Q^{-1}(V^+)^* - \mc{T}V/(tn^{3/2}) \\
& = 2Q^{-1}\mc{D}(V^+)^*  \\
&  \le (1+o(1)) t^{-1} \dD_{(V^+)^*} v n^{-3/2}.
\end{split}
\end{equation*}
In the third equality we used
$\mc{T}V/(2tn^{3/2}) = \mc{T}V^+/Q$
and in the last inequality we estimated $\mc{D}(V^+)^*$
using $i_V \le i < I$ and $i_{V^+} \le i_V$ (see Lemma \ref{lem:tV}).

Now we consider destruction, 
i.e.\ the case that $ab =e \in \GG \sm J$ is open. 
We have $\mb{E}[ \DD_i(V^e) \mid \mc{F}_i] 
= 2Q^{-1}\sum_{f \in V^*} (Y_{f(a)f(b)} + Y_{f(b)f(a)} \pm O(1))$,
where the $O(1)$ term corrects for the difference between `open'
and `$f$-open' and also for the possibility 
that $f(ab)$ may become selected rather than closed.
Then, recalling (\ref{eq:Qstar}), we have
\begin{align*}
\mb{E}[ \DD_i(\mc{D}V^e) \mid \mc{F}_i]
&= - 2Q^{-1}\sum_{f \in V} (1+ Y_{f(a)f(b)} + Y_{f(b)f(a)} \pm O(1)) 
 - \mb{E} \left[ \frac{ \DD_i(Q)\mc{T}V}{ Q} \mid \mc{F}_i \right]  \\
& = - 4Q^{-1} V ( \mc{T}Y \pm \dD_Y y)  \pm O(q^{-1}v) 
+ Q^{-1} \mc{T}V (2 + 4S/Q) \\
&= - (1 \pm  (1+o(1))\dD_Y) 8tn^{-3/2} V \pm O(q^{-1}v) 
 + (1 \pm O(\dD_S)) 8tn^{-3/2} \mc{T}V \\
&= - 8tn^{-3/2} \mc{D}V \pm  (1+o(1)) 8tn^{-3/2} \dD_Y V 
\pm O(\dD_S 8tn^{-3/2}v)  \pm O(q^{-1}v)  \\
&  \le - \left[ (1+o(1))(f_V + g_V) - 2g_Y  \right] 8t v n^{-3/2}.
\end{align*}
In the above calculation we note that we can afford
to approximate the multipliers of $V$ and $\mc{T}V$
independently as our approximations 
for controllable variables are weaker 
than those in the other ensembles.
The approximations of $Y$ and $S$
hold for all $n^{5/4} \le i<I$;
we also used $ f_Y + f_S = o( f_V) $ 
and $g_S = \wt{O}(1+t^{-1}) e^2 = o(g_V)$, which holds as
\begin{equation} \label{eq:ee/t}
(1+t^{-1}) e^2 = O(e) \text{ for } t \ge n^{-1/4}. 
\end{equation}
The lemma follows by summing 
the creation estimate over $e(V)$ edges and
the destruction estimate over $o(V)$ open pairs.
The $o(1)$ terms absorb the corrections of 
$O(L^4) v/q$ for fidelity
(see \eqref{eq:confidel})
and $O(t^2+t^{-2}) n^{-3} v$ for
higher-order terms (see \eqref{eq:conhi}),
\end{proof}

\subsection{Trend Hypothesis and Variation Equation}

The following lemma establishes the trend hypothesis,
i.e.\ that  $\mc{Z}V^* = \mc{D}V^* - (f_V + 3g_V/2)v$
is a supermartingale when $V^*$ is in its upper critical window;
we will see that this is valid under the choice $c_V=1$ 
made in Definition \ref{def:c}.

\begin{lemma} \label{lem:contrend}
If $i_V \le i < I$ and $\mc{D}V^* > (f_V+g_V) v$ 
then $\mb{E} [ \DD_i \mc{Z}V^* \mid \mc{F}_i]  \le 0$.
\end{lemma}

\begin{proof}
By Lemma \ref{lem:seconds} 
(replacing $2g_V$ by $\tfrac{3}{2}g_V$ to adjust for $V^*$) we have
\[ \DD_i(v\dD_{V^*}) 
= \left( \tfrac{e(V)}{8t^2} - o(V) \right) \dD_{V^*} v \cdot 8tn^{-3/2}
 + \dD'_{V^*} vn^{-3/2}  
 + O(\dD_{V^*} v) n^{-5/2}, \text{ where } \]
\[ \dD'_{V^*} \ge 4\dD t  \dD_{V^*} 
+ (\tT'/\tT - e(V)t^{-1}) \tfrac{3}{2}g_V. \]
Since 
\begin{equation} \label{eq:varcon}
c_V = 1
\end{equation}
for all $V$ in the Controllable Ensemble, 
we have $f_{V^{+*}}=f_{V^*}$, so
$\dD_{V^{+*}}-\dD_{V^*} = (3/2)(g_{V^+}-g_V)$.  

There is no $V^+$ term if $e(V)=0$, and otherwise
$\frac{g_{V^+}}{g_V} = \frac{t^{e(V)}+t}{t^{e(V)}+1} < 2t$, 
so by Lemma \ref{lem:DV*}
\begin{align*}
\frac{\mb{E}[ \DD_i(\mc{Z}V) \mid \mc{F}_i]}{8tvn^{-3/2}}
& \le (1+o(1)) \tfrac{e(V)}{8t^2} \dD_{V^{+*}} 
 - (1+o(1)) o(V)  ( f_V + g_V - 2g_Y ) \\
& \quad - \left[ (\tfrac{e(V)}{8t^2} - o(V)) \dD_{V^*} 
+ \tfrac{\dD'_{V^*}}{8t}\right]
+ \wt{O}( \dD_{V^*} t^{-1} n^{-1} )  \\
& \le \tfrac{e(V)}{8t^2} \cdot \tfrac{3}{2}(g_{V^+}-g_V) 
+ \tfrac{o(V)}{2} g_V   - \tfrac{1}{2} \dD f_V \\
& \quad - (\tfrac{\tT'}{8t \tT} - \tfrac{e(V)}{8t^2} + \tfrac{\dD}{2}) \cdot \tfrac{3g_V}{2}  + O(g_Y)  +o( \dD_{V^+}/t^2) + o(\dD_V) \\
& \le \tfrac{g_V}{2}( o(V) - \tfrac{3\dD}{2} + \tfrac{ 3 e(V)}{ 4 t}  - \tfrac{3 \tT'}{ 8t \tT } + o(1))
- \tfrac{1}{4} \dD f_V + O(g_Y)  +o( \dD_{V^+}/t^2).
\end{align*}
For the last inequality, we have cancellation
of two terms $\tfrac{e(V)g_V}{8t^2}$ with opposite signs,
and we used $g_{V^+} \le 2tg_V$.
Finally, $ \mb{E}[ \DD_i(\mc{Z}V) \mid \mc{F}_i] \le 0$, 
as the dominant terms are $-\tfrac{3\tT'}{16t \tT } g_V$
and/or $- \tfrac{1}{4} \dD f_V$.
\end{proof}

\subsection{Boundedness hypothesis}

For the boundedness hypothesis,  
we fix any $V=X_{\phi,J,\GG}$ in the Controllable Ensemble 
and estimate 
$\Var_{V^*} =  \Var(\mc{Z}V^*(i) \mid \mc{F}_{i-1})$ 
and $N_{V^*} = |\DD_i\mc{Z}V^*|$.  
Recall that it suffices to establish (\ref{eq:boundVar}) 
and (\ref{eq:boundN}), as in the following lemma.
We remark that the proof of the `boundary case' $|V_\GG|=M^3$ 
is quite delicate, and it is here that the details of
property (v) in Definition \ref{def:I0} are important.

\begin{lemma} \label{lem:conbound}
If $i_V \le i < I$ and $V$ is good then
$\Var_V = o\left( \frac{ (t^{-e(V)} e^\dD v)^2}{L^3n^{3/2}} \right)$
and $N_V = o\left( \frac{t^{-e(V)} e^\dD v}{L^2} \right)$. 
\end{lemma}

\begin{proof}
Recalling that we restrict our attention to $ t \ge t_V$, 
we can bound the one-step change in 
$\mc{T}V^* + (f_V + 3 g_V / 2 )v$ 
by $O((t + t^{-1})vn^{-3/2}) = \wt{O}( v n^{-5/4})$, 
which is negligible in comparison with the required estimates. 
It therefore suffices to consider changes in $V^*$ 
rather than $\mc{Z}V^*$. As in the trend hypothesis, we can
obtain these estimates as a sum over all $e \in \GG \sm \GG[A]$.
(Here we use $|V_\GG| \le M^3 = O(1)$ and the simple observation 
that if random variables $A$ and $B$ each have variance 
at most $ \sigma^2$ then $A+B$ has variance at most $4 \sigma^2$.)

Thus for each $e = \aA\bB \in \GG \sm \GG[A]$ we estimate 
$ N_e = |\DD_i V^e|$ 
and $\Var_e =  \Var(\DD_i V^e \mid \mc{F}_{i-1})$.

We start with the creation calculation, i.e.\ the case $e \in J$.
All scalings here will be with respect to the extension
$(\phi,J \sm e, \GG)$ obtained by changing $e$ to an open pair:\
e.g.\ $S_A^{V_\GG}  = v \hat{q} p^{-1}$.
Let $A' = A \cup \{\aA,\bB\}$,
where $A \sub V_\GG$ is the base of the extension.
We note that if  $\DD_i V^e \ne 0$ then 
for any $B$ with $A' \sub B \sub V_\GG$ 
the edge $e_{i+1}$ selected at step $i+1$ must fall in
some extension in $X_{\phi,(J \sm e)[B],\GG[B]}$.
We consider the `hardest' such extension:
let $S_{\rm m} = \min_{A' \sub B \sub V_\GG} S^B_{A}$.

Let $B_{\rm m}$ be some set $B$ 
achieving the minimum in this definition.
We note that
\begin{enumerate}
\item[(B1)] $S_{A'}^{B_{\rm m}} \le 1$,
\item[(B2)] $v/S_{\rm m} = \max_{B_{\rm m} \sub C \sub V_\GG} S^{V_\GG}_C$,
\item[(B3)] $S_{\rm m} \ge n^{\dD'} (\hat{q}/p)$,
\item[(B4)] $S_{\rm m} = \max_{A \sub C \sub B_{\rm m}} S^{B_{\rm m}}_C$,
\end{enumerate}
Indeed, (B1) and (B2) follow from the definition of $B_{\rm m}$,
and (B3) and (B4) from controllability of $V$.
By property (iv) of Definition \ref{def:I0}
applied to the extension from $A$ to $B_{\rm m}$
and (B4) we estimate 
\[ p_e := \mb{P}[\DD_i V^e \ne 0] 
< L^{4|V_\GG|} S_{\rm m}/q. \]
Also, applying property (iv) of Definition \ref{def:I0}
to the extensions from $A'$ to  $B_{\rm m}$ (using (B1))
and from $B_{\rm m}$ to $V_\GG$ (using (B2)), we estimate
\[ N_e < L^{ 4| V_\GG|} \cdot  L^{4|V_\GG|} S^{V_\GG}_{A}/S_{\rm m}
\le L^{8|V_\GG|} p\hat{q}^{-1}n^{-\dD'} v, \]
using (B3) for the second inequality. Then
\[ \Var_e < p_e N_e^2 
< L^{20|V_\GG|} (S_{\rm m}/q) (S^{V_\GG }_{A}/S_{\rm m})^2
= L^{20 |V_\GG|}( \hat{q}/p)^2 v^2/ ( q S_{\rm m})
< L^{20 |V_\GG|} ( 2tn^{3/2} )^{-1} n^{-\dD'} v^2.\]
Noting that creation only occurs when $e(V) \ge 1$,
these estimates are well within the required bounds,
as $e^\dD > n^{-\dD/4}$ and $\dD \ll \dD'$.

It remains to consider destruction, 
i.e.\ the case $e = \aA\bB \in \GG \sm J$.
Let $(A',J',\GG')$ be obtained from $(A,J,\GG)$ 
by `gluing a $Y$-variable on $\aA\bB$' as follows. 
Let $\gG$ be a new vertex, $V'=V_\GG \cup \{\gG\}$, 
$A' = A \cup \{\aA,\gG\}$, $J' = J \cup \{\bB\gG\}$ 
and $\GG' = \GG \cup \{\aA\gG, \bB\gG\}$
(so this definition depends on the order of $\aA$ and $\bB$).
To analyse destruction of extensions $f \in V^*$
due to closures of $e$ by selecting the edge
corresponding to $\aA\gG$, we consider extensions 
in $X_{\phi',J',\GG}$ where $\phi':A' \to [n]$
restricts to $\phi$ on $A$ and $\phi'(\aA\gG)$ 
is the edge $e_{i+1}$ added at step $i+1$.
In only considering the case that $\gG$ is a new vertex 
we make crucial use of the distinction between $V^*$ and $V$. 

As in the creation calculation we have
\[ p_e := \mb{P}[\DD_i V^e \ne 0] 
< L^{4|V_\GG|} S_{\rm m}/q, \]
where $S_{\rm m} = S^{B_{\rm m}}_{A}
= \min_{A' \sub B \sub V'} S^B_{A}$,
and all scalings are with respect to $(J',\GG')$.
We claim that 
\begin{equation} \label{eq:Sm}
S_{\rm m} \ge yn^{\dD'}. 
\end{equation}
To see this, note that if $B_{\rm m} = A \cup \{\gG\}$
then $  S_{\rm m} = \hat{q} n \ge y n^{\dD'}$.
Otherwise, we write 
$S^{B_{\rm m}}_A = S^{B_{\rm m}}_{B_{\rm m} \sm \gG} 
S^{B_{\rm m} \sm \gG}_A$. We have 
$S^{B_{\rm m}}_{B_{\rm m} \sm \gG} \ge y$ by construction of $(J',\GG')$
and $S^{B_{\rm m} \sm \gG}_A \ge n^{\dD'}$, since $V$ is controllable.
This proves the claim.

Now we claim that the magnitude of the change 
due to $e$ is bounded as 
\begin{equation} \label{eq:Ne}
N_e < 2L^{8|V'|+7} vy/S_{\rm m}.
\end{equation}
The lemma follows from this bound;
indeed, substituting \eqref{eq:Sm} gives
$N_e = \wt{O}( n^{-\dD'} v )$ and
\[\Var_e < p_e N_e^2 = \wt{O}(S_{\rm m}/q) (yv/S_{\rm m})^2
= \wt{O}(y^2 v^2 / q S_{\rm m}) 
= \wt{O}(n^{-\dD'} v^2 n^{-3/2})).\]
Thus it remains to prove \eqref{eq:Ne}.

First we note that the same argument as for creation
applies if we are not at the boundary of the ensemble,
i.e.\ if $|V_\GG| < M^3$, so $|V'| \le M^3$.  
Indeed, applying property (iv) of Definition \ref{def:I0}
to the extensions from $A'$ to  $B_{\rm m}$
and from $B_{\rm m}$ to $V_\GG$, we estimate
\[ N_e < L^{ 4|V'|} \cdot  L^{4|V'|} S^{V'}_{A}/S_{\rm m}
\le L^{8|V'|} n^{-\dD'} v, \]
using \eqref{eq:Sm} and $ S_A^{V'} = y S_{A}^V = yv$.

It remains to consider the boundary case $|V_\GG|=M^3$.
We start with those subcases in which 
we can still implement the preceding calculation.
We still have at most $L^{ 4|V'|}$ extensions 
from $A'$ to $B_{\rm m}$, using property (v) 
of Definition \ref{def:I0} if $B_{\rm m}=V'$.
Next we consider the extension series from $B_{\rm m}$
to $V'$ and let $C \subn V'$ be the set preceding $V'$.
We claim that if $\bB \in C$ then we can still implement 
the above bound using extensions on at most $M^3$ vertices,
so that property (iv) of Definition \ref{def:I0} still applies.
Indeed, writing $C^- = C \sm \{\gG\}$ we have 
\[ S^{V'}_{A}/S_{\rm m} = S^{V'}_{B_{\rm m}} 
= S^C_{B_{\rm m}} S^{V'}_C = S^C_{B_{\rm m}} S^{V_\GG}_{C^-},\]
so considering extensions from $C^-$ to $V_\GG$ we still have
at most $L^{4|V'|} S^{V'}_{A}/S_{\rm m}$
extensions from $B_{\rm m}$ to $V'$, as claimed.

Now we may assume $\bB \notin C$.
We can also assume $S^{V'}_C \ge y/L^7$, otherwise we can
still implement the previous calculation 
using property (v) of Definition \ref{def:I0}.
On the other hand, by definition of the extension series
we have $S^{V'}_C \le S^{C \cup \bB}_C \le y$,
as the extension from $C$ to $C \cup \bB$
contains the edge $\bB\gG$ and the open pair $\aA\bB$.
Thus we give up a factor of at most $L^7$ in bounding
extensions from $C$ to $C \cup \bB$ by a $Y$ variable,
and we can estimate extensions from $C \cup \bB$ to $V'$
using extensions from $C$ to $V_\GG$, 
since $S^{V'}_{C \cup \bB} = S^{V_\GG}_C$. This gives
\[ N_e < L^{ 4|V'|} \cdot L^{4|C|} S^C_{B_{\rm m}}
\cdot 2y \cdot L^{4|V' \sm C|} S^{V'}_{C \cup \bB}
< L^{8|V'|+7} S^{V'}_{A}/S_{\rm m}
\le 2L^{8|V'|+7} n^{-\dD'} v, \]
which completes the proof of the claim \eqref{eq:Ne}, 
and so of the lemma.
\end{proof}

Now that we have verified the trend and boundedness hypotheses
for $V^*$, Lemmas \ref{th+bh} and \ref{lem:beginning} show whp
$V^* = (1 \pm \dD_{V^*}) v$ for $i_V \le i < I$.
In combination with Lemma \ref{lem:VV}
this proves Theorem \ref{controllable}.

\section{Stacking ensemble} \label{sec:stack}

In this section we prove that all variables 
in the stacking ensemble have the desired concentration,
assuming that all variables in all ensembles
are well-behaved at earlier times. 
Recall that $\I3$ is the minimum of the 
stopping times $I_V$ over all variables $V$ in the stacking ensemble. 
The following theorem bounds the probability that we reach 
the universal stopping time $I$ before step $ i_{max}$ 
because some stacking variable $V$ is good
(see Definition \ref{def:IV}) but fails to satisfy 
the required bound $|\mc{D}V| \le \dD_V v$.

\begin{theo}\label{stacking}
With high probability we do not have $I=\I3 \le i_{max}$.
\end{theo}

As for the other ensembles, we will prove this theorem
by verifying the trend and boundedness hypotheses.
Throughout the section we consider some stacking variable 
$V = S_{uv}^\pi = X_{\phi,J,\GG}$, for some non-edge $uv$,
where we recall that $V(\GG)=V(S_{uv}^\pi) 
= \{ \aA_u, \aA_v, \aA_1, \dots, \aA_{ |\pi|} \}$,
$A=\{\aA_u,\aA_v\}$, $\phi(\aA_u)=u$, $\phi(\aA_v)=v$ 
and $(J,\GG)$ is defined so that edges specified 
by the extension are mapped to edges of $G(i)$,
and likewise for open pairs.
Recalling that we gave a separate argument for vertex degree 
variables in Lemma \ref{lem:vdeg}, 
we can assume $V$ is not such a variable.
Similarly to the analysis of controllable variables
(except that here we do not approximate $V$ by $V^*$),
we decompose the one-step change in $V$ as 
\[ \DD_i(V) = \sum_{e \in \GG \sm \GG[A]} \DD_i(V^e) \pm F_i(V),\]
where each $\DD_i(V^e)$ accounts for the change in $V$ due to $e$,
as follows. If $e \in J$ then, letting $V^+$ be obtained from $V$
by changing $e$ from an edge to an open pair, $\DD_i(V^e)$ is the
number of embeddings $f \in (V^+)^*$ such that $f(e)$ is the edge
$e_{i+1}$ selected at step $i+1$.
If $e \in \GG \sm J$ then $-\DD_i(V^e)$ is the number 
of embeddings $f \in V$ which are destroyed at step $i+1$
by $f(e)$ being selected or closed.
The fidelity term $F_i(V)$ corrects for embeddings $f \in V^*$
where $f(e)$ is affected for more than one $e$ simultaneously
(see Section \ref{sub:correct}).

\subsection{Subextensions of stacking variables} \label{sec:subext}

This subsection concerns certain subextensions of stacking variables 
that will be particularly important throughout this section.
For the following two special structures
we will appeal to the Controllable Ensemble for our estimates, 
and so we need to show that these extensions are indeed controllable. 
\begin{itemize}
\item Let $ (uv, J, \GG) $ be the extension corresponding 
to some stacking sequence $ \pi \in \mc{S}_M$ at the boundary 
of the ensemble, i.e.\ with $ w( \pi) = 2M$.  
The {\em backward extension} $ B_{ \pi} $ 
is the extension $ (A', J', \GG') $ with
$ A' = \{\aA_u,\aA_v,\aA_x,\aA_y\}$, 
$J' = J$ and $\GG' = \GG \sm \aA_x \aA_y $.
\item An \emph{$h$-fan} at the triple $ A=abc$
is any extension of the form $(A,J,\GG)$, where the base is $A=abc$, 
there are $h$ additional vertices $v_1,\dots,v_h$ in $V_\GG$, 
the sequence $bv_1\dots v_h c$ is a path of length $h+1$ in $\GG$, 
and $av_i \in \GG \sm J$ is open for $i \in [h]$.  
We emphasize that the pairs in the path $ bv_1\dots v_h c$ 
can be either edges or open pairs.
\end{itemize}

Both of these extensions arise from the boundary conditions
in our choice to restrict the stacking ensemble 
to $M$-bounded variables. Recalling Definition \ref{def:stack2},
we need to consider backward extensions
due to condition (i) that $w( \pi) \le 2M$ 
and fans due to condition (ii) forbidding a subsequence
of length $M$ using only $\{X^I,Y^I\}$:
in both cases there is at least one direction in which
we cannot stack $Y$ on the last rung.

Now we show that these two extensions are controllable.
We recall that $ M = 3/\eps$ 
and $ \hat{q}(t_{max}) = n^{-1/2 + \eps}$.

\begin{lemma}
\label{lem:backsies}
All $M$-fans and backward extension variables
are controllable at time $t_{max}$.
\end{lemma}

\begin{proof}
We start by considering an $M$-fan $(A,J,\GG)$. 
Among all such extensions, the minimum scaling is 
$(\hat{q}n)^M p^{M+1} > n^{ \eps M - 1/2} = n^{5/2}$,
which is achieved when the path
$bv_1\dots v_M c$ belongs entirely to $J$.
Fix $B$ with $A \subn B \sub V$ 
that minimises $S^B_A = S^B_A(J,\GG)$. 
We need to show that $S^B_A \ge n^{\dD'}$. 
As $S_A^{V_\GG} > n^{5/2}$
we can assume that $B \ne V_\GG$, 
so we can find $v_i$ in $B$ such that 
not both $v_{i-1}$ and $v_{i+1}$ are in $B$.  
(Here $v_0 = c$ and $v_{M+1} = b$.) 
Now removing $v_i$ from $B$ reduces the scaling 
by at least $y > \hat{q}n^{1/2} = n^{\eps}$, 
so by minimality we have $|B|=|A|+1$,
so $S^B_A \ge y > n^{\eps} > n^{\dD'}$ 
(recalling (\ref{eq:epsilondelta})).

Now consider (with notation as above) a backward extension 
$B_{\pi} = (A',J',\GG')$ with $w(\pi)=2M$.
We fix $B$ with $A' \subn B \sub V$ and estimate $S_{A'}^B$ 
as a sequence of single-vertex extensions.
First we consider the case that there is some $T \sub V$
disjoint from $B$ such that some component $C$ of $\GG' \sm T$
contains $\{\aA_x,\aA_y\}$, but not $\aA_u$ or $\aA_v$.
Then we consider vertices of $B \sm C$ in stacking order
and vertices of $B \cap C$ in reverse stacking order.
Each step contributes a factor of at least $y > n^{\eps}$
to the scaling, so $ S_{A'}^B > n^\eps > n^{\dD'}$.

Now we can assume there is no such $T$, which implies that
$B$ intersects every rung and contains all $ \aA_i $ 
such that $ \pi(i+1) = O $. We claim that $ |B| \ge M + 2 $.  
We note that this will imply the lemma, as estimating 
$S_{uv}^B$ by a sequence  of single-vertex extensions gives
\[ S_{A'}^B = S_{uv}^B/ (n^2 \hat{q})  
> (n^\eps)^{|B| - 2}/ n^2 
\ge (n^\eps)^{|M|}/ n^2  = n > n^{\dD'}.\]
It remains to show the claim. We bound the intersection of $B$
with the set of $2M$ vertices that contribute to $ w(\pi)$. 
Suppose $\pi$ has $i$ occurrences of the symbol $O$ 
in the sequence $ \pi(2), \dots, \pi(|\pi|-1)$ 
and $j$ occurrences of $O$ or $E$ in $ \{ \pi(1), \pi(|\pi|) \}$.  
Then there are at most $ i+1$ triangular ladders 
and $ \pi $ has $ 2M - i- j$ turning points 
(recall that the positions with the symbols 
$ X^O$ or $Y^O$ give turning points),
of which at most $2-j$ are in $A'$
(namely $ \alpha_u$ and $\alpha_{ |\pi|-1} $).
Let $T$ be the set of turning points not in $A'$,
so that $|T| \ge 2M-i-2$.
For each triangular ladder there is 
a path of rungs spanned by $T \cap L$, so we must have 
$|B \cap T \cap L| \ge \bfl{|T \cap L|/2}$. We deduce
$ |B \sm A'| \ge i + \frac{ 2M-2-i}{2} - \frac{ i+1}{2} \ge M-2$,
which proves the claim, and so the lemma.
\end{proof}

\begin{rem} \label{rem:fan}
The proof of Lemma \ref{lem:backsies} shows moreover 
that a fan of any size is controllable at any time 
at which it has scaling at least $n^{\dD'}$. 
\end{rem}

\subsection{Boundedness hypothesis}

Here we verify the boundedness hypothesis,
for which the arguments are somewhat similar 
to those given above for the controllable ensemble,
and are relatively short (the bulk of the section 
will then be occupied with verifying the trend hypothesis).
Recalling (\ref{eq:boundVar}) and (\ref{eq:boundN}),
and that $c_V \ge L^{15}$ for all $V$ in the stacking ensemble
(see Definition \ref{def:c}), 
it suffices to prove the following lemma.

\begin{lemma} \label{lem:stkbound}
If $i_V \le i<\min\{I,J_V\}$ then 
\[ N_V < (1+t^{-e(V)}) e v \ \text{ and } \
\Var_V <  n^{-3/2} ((1+t^{-e(V)}) e v)^2.\]
\end{lemma}

\begin{proof}
As in the proof of Lemma \ref{lem:conbound},
it suffices to establish the stated bounds
for each $e \in \GG \sm \GG[A]$ on $ N_e = |\DD_i V^e|$ 
and $\Var_e =  \Var(\DD_i V^e \mid \mc{F}_{i-1})$
(we do not need to take advantage of better bounds 
available on the change in the difference between 
these variables and their tracking variables).
There are two cases, according to whether $e$ 
is an open pair or an edge.

We start by considering the case that $e \in J$ is an edge. 
Let $ e = \aA_x\aA_y $ where $x < y$. 
Let $A' = A \cup \{\aA_x,\aA_y\}$ and 
$S_{\rm m} = \min_{A' \sub B \sub V} S^B_A = S^{B_{\rm m}}_A$, 
where all scalings are with respect to $(J \sm e,\GG)$. 
Noting that $S^{V_\GG}_A = \hat{q}p^{-1}v$,
as in the proof of Lemma \ref{lem:conbound},
as $i<I$ we have
\[ p_e := \mb{P}[\DD_i V^e \ne 0] < L^{4|V_\GG|} S_{\rm m}/q
\ \ \text{ and  } \ \ 
N_e < L^{8|V_\GG|} \hat{q}p^{-1}v/S_{\rm m},  \]
\[ \text{ so  } \ \  \Var_e < p_e N_e^2 
< L^{20|V_\GG|} (\hat{q}p^{-1}v)^2/(qS_{\rm m}).\]
We calculate the scaling $S_{\rm m}$ one vertex at a time.  
Each vertex contributes a factor of at least $p\hat{q}n = y$,
and $\aA_y$ contributes at least $\hat{q}^2 n  = x$, 
since the edge $\aA_x \aA_y$ was switched to an open pair 
in $ ( J \sm e, \GG) $. If $|B_{\rm m} \sm A| \ge 2$ 
we have $ S_{\rm m} \ge xy$, so  
\begin{align*} 
N_e & < L^{ 8|V'|} v p^{-1} \hat{q} / (xy) 
=  t^{-1} ev \cdot L^{ 8|V'|} (4t)^{-1} e^3
\ll t^{-1} ev \ \text{ and }   \\ 
\Var_e  & < L^{20 |V'|} ( \hat{q} p^{-1} v)^2 / ( q x y) 
= n^{-3/2}((2t)^{-1}ev)^2 \cdot y^{-1}L^{20 |V'|} 
\ll n^{-3/2}(t^{-1}ev)^2, 
\end{align*}
which are sufficient, as $e \in J$ implies $e(V) \ge 1$. 
On the other hand, if $|B_{\rm m} \sm A| = 1$, 
then $B_{\rm m} = A'$, so this corresponds to the edge
$e_{i+1}=u'v'$ added at step $i+1$ playing the role of an edge 
$e$ that creates the first $Y$-extension of $\pi$ 
(as $V$ is not a vertex degree variable).
Writing $\pi'$ for the stacking sequence obtained from $\pi$
by removing $\pi(1)$, and $V' = S^{\pi'}_{u'v'}$ for the
corresponding stacking variable based at $u'v'$ 
(which is open before we add $e_{i+1}$),
we can improve the above bounds to
$p_e \le 2x/q$ and $N_e \le V' \le 2v/y$, 
so $\Var_e \le 8 q^{-1}(t^{-1}v)^2$, which again suffices.

It remains to consider the changes due to closing some
open pair $\aA_x\aA_y = e \in \GG \sm J$
(which may be a rung or a stringer).
This is described by a structure 
where for some vertex $\gG$
we already have the edge $\aA_y \gG$ 
and then we add the edge $\aA_x \gG$.
There are two subcases according to whether $\gG$ 
belongs to $V_\GG$ or is a new vertex.
In both subcases, we consider $J' = J \cup \{\aA_y \gG\}$ 
and $\GG' = \GG \cup \{\aA_x \gG, \aA_y \gG\}$
on the vertex set $V'=V_\GG \cup \{\gG\}$
(which is $V_\GG$ if $\gG \in V_\GG$),
we let $A=\{\aA_u, \aA_v\}$, $A' = A \cup \{\aA_x,\gG\}$ 
and $S_{\rm m} = \min_{A' \sub B \sub V'} S^B_A = S^{B_{\rm m}}_A$, 
where all scalings are with respect to $(J',\GG')$.
As above, we estimate 
\[ p_e := \mb{P}[\DD_i V^e \ne 0] < L^{4|V'|} S_{\rm m}/q
\ \ \text{ and  } \ \ 
N_e < L^{8|V'|} S^{V'}_A/S_{\rm m},  \]
\[ \text{ so  } \ \  \Var_e < p_e N_e^2 
< L^{20|V'|} (S^{V'}_A)^2/(qS_{\rm m}).\]

Now consider the subcase $\gG \in V_\GG$.
We note that $S^{V'}_A \le v$.
As $S_{\rm m} \ge y$ we deduce
\[ N_e < L^{8|V_\GG|} v/y 
= ev \cdot L^{8|V_\GG|} (2t)^{-1} e \ \text{ and } \]
\[ \Var_e  < L^{20|V_\GG|} v^2/qy 
= n^{-3/2} (ev)^2 \cdot L^{20|V_\GG|} (2t)^{-1} e^2.\]
These bounds suffice unless $e(V)=0$,
in which case we obtain the required bounds
using the better bound $S^{V'}_A \le pv$,
where the factor of $p$ is due to the edge
$\aA_y \gG \in J' \sm J$.

It remains to consider the subcase $\gG \notin V_\GG$.
Then $\GG'$ is obtained from $\GG$ by adding a $Y$ extension 
on $\aA_x \aA_y$, so $S^{V'}_A = vy$. 
If $ S_{\rm m} > L^{40|V_\GG|} y^2$ (say)
then the above bounds are easily sufficient.
Estimating $S_{\rm m}$ vertex by vertex in the stacking order
we see that this holds if $|B_{\rm m } \sm A| \ge 3$ 
(when $ S_{\rm m} \ge y^3 \gg y^2$)
or if $|B_{\rm m } \sm A| = 2$ and not both steps 
from $A$ to $B_{\rm m}$ are $Y$ extensions
(this gives $ S_{\rm m} \ge xy \gg y^2$).

The remaining cases need more precise estimates
on $N_e$ and $\Var_e$ that avoid the 
polylogarithmic loss in the crude estimates above.
Consider the case that $|B_{\rm m } \sm A| = 2$
and $B_{\rm m}$ is obtained by two $Y$ extensions,
so $S_{\rm m} = y^2$.
Here we can use stacking variables to estimate $p_e$ and $N_e$, 
as $(A,B_{\rm m})$ induces the extension $S^{\pi(1) Y^I}_{uv}$, 
and $N_e \le S^{\pi'}_{\aA_y\aA_x}$, where $\pi = \pi(1) \pi'$. 
We have the better bounds $p_e < 2 S_{\rm m}/q = 2 y^2/q$ 
and $N_e < 2 S^{V'}_{A}/S_{\rm m} = 2 v/y$,
so $\Var_e < 8 v^2/ q = 8n^{-3/2} (ev)^2$, which suffices.

Now consider $ |B_{\rm m } \sm A| =1$,
so $ \alpha_x \in \{ \aA_u , \aA_v\}$
and $B_{\rm m} = \{ \aA_u, \aA_v, \gG\} $.
The extension from $A$ to $B_{\rm m}$ 
is an open degree, with scaling 
$S_{\rm m} = x_1 = \hat{q} n$,
so we estimate $p_e \le 2x_1/q = 2/n$.
To estimate $N_e$ we consider the extension 
$(A',J',\GG')$ in two steps, where in the first step
we add all vertices in the stacking order up to $\aA_y$,
and in the second step we add the remaining vertices.
Thus we bound $N_e \le \sum_{f \in V^1} V^2_f$,
where $V^1$ is a fan extension with base $A'$,
and $V^2_f$ is a stacking variable with base $f(\aA_x \aA_y)$.
The scalings $v_1$ and $v_2$ satisfy 
$v_1 v_2 = S^{V'}_{A'} = vy/x_1$.

If $V^1$ is controllable at time $t$
we obtain the required bounds 
from $N_e < 2v_1 \cdot 2v_2 = 4vy/x_1 = 8tn^{-1/2} v$
and $\Var_e < 2n^{-1}(4vy/x_1)^2 = 32t^2 n^{-2} v^2$.
Now suppose $V^1$ is not controllable at time $t$,
so $v_1 < n^{\dD'}$ by Remark \ref{rem:fan}.
If the fan has any non-base vertex besides $\aA_y$
then $v_1 \ge (\hat{q}n)^2 p^3 = (2t)^3 \hat{q}^2 n^{1/2}$,
giving $\hat{q} < t^{-3/2} n^{\dD'/2 - 1/4}$,
so $S_{\rm m} = \hat{q} n > L^{40|V_\GG|} y^2$,
and we have already completed the proof when this holds.
It remains to consider the case that the fan
is a single vertex extension from $A'$ to $\aA_y$.
Note that $v_1 \ge 1$ by definition of $B_{\rm m}$,
so $V^1 \le L^4 v_1$ (as $i<I$), giving
$N_e < L^4 v_1 \cdot 2v_2 = 2L^4 vy/x_1$;
this suffices by the same calculation
as when $V^1$ is controllable.
\end{proof}

\subsection{Tracking variables} \label{sub:stktrack}

Here we will recall and explain in more detail
the definition of the tracking variables
$\mc{T}V$ in Section \ref{deftrack}.
We also describe the pair decomposition 
of their one step changes. 
There will be two cases for $V = S^\pi_{uv}$
depending on the form of $\pi$.

\subsubsection{Standard tracking variables}

The first case, which we call \emph{standard}, is that
$ \pi( |\pi|-1) \neq O $ or $ \pi( |\pi|) \in \{O,E\}$. 
We write $ \pi = \pi^- \circ U $, 
where $U$ is the last element of $\pi$, 
and let \[ \mc{T} V = V^- \mc{T}U,
\text{ where } V^- = S_{uv}^{\pi^-}. \]
Note that this choice of $\mc{T}V$ isolates variations 
that are not caused by variations in $V^-$.

We say that a pair $e$ is \emph{terminal} if it belongs to $U$, 
i.e.\ it contains the final vertex of $V$; 
otherwise we say that $e$ is \emph{internal}.  We write 
\begin{equation}
\label{eq:breakdown}
\DD_i(\mc{T}V) = \DD_i(V^-) \mc{T}U + V^- \DD_i(\mc{T}U) 
= \sum_{e \in \GG \sm \GG[A]} \DD_i(\mc{T}V^e) + H_i(V),
\end{equation}
where similarly to \eqref{eq:conhi}
the higher-order correction term is
\begin{equation} \label{eq:stackhi}
H_i(V) = O(t^2+t^{-2}) n^{-3} v,
\end{equation}
and $\DD_i(\mc{T}V^e)$ is defined as follows.
\begin{enumerate}[(i)]
\item If $e$ is a terminal edge then 
$\DD_i(\mc{T}V^e) = \frac{\mc{T}V}{tn^{3/2}}$,
\item If $e$ is a terminal open pair then 
$\DD_i(\mc{T}V^e) = \frac{\DD_i(Q)}{Q}\mc{T}V$,
\item If $e$ is internal then 
$\DD_i(\mc{T}V^e) = \DD_i((V^-)^e) \mc{T}U$.
\end{enumerate}
Note that (iii) uses the definition of $\DD_i(V^e)$
above with $V^-$ in place of $V$.

\subsubsection{Partner tracking variables} 
\label{subsub:ptrack}

The other case, which we call \emph{partner}, 
is that $ \pi( |\pi|-1) = O $ and 
$ \pi( |\pi|) \notin \{O,E\}$.
We must have $|\pi| \ge 2$, and the vertices 
$\{\aA_{ |\pi|-2}, \aA_{ |\pi| -1}, \aA_{|\pi|} \}$
form a triangle in $V = S_{uv}^\pi$, in which at most
one pair is an edge and the other pairs are open.
We say that the open pair $ \aA_{ |\pi|-2} \aA_{ |\pi| -1} $ 
and the pair $ \aA_{ |\pi|-2} \aA_{ |\pi|} $ 
(which can be an edge or an open pair) are {\em partner pairs};
it is natural to treat them together because of the `symmetry' 
interchanging $\aA_{ |\pi| -1}$ and  $\aA_{|\pi|}$ 
(although it can be that one is an edge and the other is open).
The pair $ \aA_{ |\pi|-1} \aA_{ |\pi|} $ is still called terminal; 
its treatment is exactly as in (i)~and~(ii) above.

We emphasise that we do not consider partner pairs to be terminal, 
even though one of them uses the last vertex of $V$.
We also do not consider partner pairs to be internal.

We write $\pi = \pi^- O U$, $V^- = S_{uv}^{\pi^-}$, 
$ \beta = \aA_{|\pi|-2} $ and let
$ \mc{T}V = \sum_{f \in V^-} X_{f(\bB)} \hat{U}_f$, where
\[
\hat{U}_f =
\begin{cases}
X_{ f( \bB) } \cdot Qn^{-2} & \text{ if } U \in \{ X^I, X^O \} \\
X_{ f( \bB) } \cdot 2tn^{-1/2} & \text{ if } U = Y^I \\
Y_{ f( \bB)} \cdot Qn^{-2} & \text{ if } U = Y^O.
\end{cases}
\]
To interpret this formula, note that for each $f \in V^-$
we are approximating the number of choices for the three remaining 
edges as if they were independent events:\ for the partner pairs
we include a degree or open degree factor 
$Y_{f(\bB)}$ for an edge or $X_{f(\bB)}$ for an open pair,
and for the terminal pair we include a probability factor
of $Qn^{-2}$ for an open pair or $2tn^{-1/2}$ for an edge.

We unify the two definitions of $\mc{T}V$ by writing
\begin{equation} \label{eq:gentrack}
\mc{T}V = \sum_{f \in V^-} \mc{T}_f V, \text{ where }
\mc{T}_f V = \mc{T}U \text{ if } \pi = \pi^- U  \text{ or }
\mc{T}_f V = X_{f(\bB)} \hat{U}_f \text{ if } \pi = \pi^- O U.  
\end{equation}
We keep the same definition 
as in points (i) and (ii) above 
of $\DD_i(\mc{T}V^e)$ for terminal pairs,
and extend it to internal pairs 
(consistently with (iii) above)
and partner pairs as follows.
\begin{enumerate}[(i)]
\setcounter{enumi}{2}
\item 
If $e$ is an internal edge then 
$\DD_i(\mc{T}V^e) =  \sum_{f \in V^{-+}} I^e_f \mc{T}_f V$,
where $V^{-+}$ is obtained from $V^-$ 
by changing $e$ to an open pair and $I^e_f$ is 
the indicator of the event that $e_{i+1}=f(e)$. \\
If $e$ is an internal open pair then 
$\DD_i(\mc{T}V^e) =  \sum_{f \in V^-} I^e_f \mc{T}_f V$,
where $I^e_f$ is the indicator of the event 
that $e_{i+1}$ closes $f(e)$.
\item If $e$ is a partner edge then $\DD_i(\mc{T}V^e) = \sum_{ f \in V^-} \DD_i(Y_{ f( \beta)}) \cdot X_{f( \beta)} \cdot Qn^{-2}$. \\
If $e$ is a partner open pair then
$\DD_i(\mc{T}V^e) = \sum_{ f \in V^-} \DD_i(X_{ f( \beta)}) \hat{U}_f$.
\end{enumerate}

\subsubsection{Classification of pairs}

As in the controllable ensemble, we will verify 
the trend and boundedness hypotheses by considering separately 
$\DD_i(\mc{D}V^e) := \DD_i(V^e) - \DD_i(\mc{T}V^e)$
for each $e \in \GG \sm \GG[A]$. We will organise the trend hypothesis
by grouping together terms that use the same method of calculation,
so here we introduce some terminology to classify these terms.
We have met special cases of some of these terms earlier 
when we considered the global variables:\ 
again `simple' terms are those described 
by another variable in our ensemble,
and the `product' terms in the global variables
are analogous to the `internal' terms here.
We use the following notation:
\begin{itemize}
\item For any $y \le |\pi|$ we let $ \pi|_y$ denote the prefix of $\pi$ of length $y$. 
\item If the final symbol $\pi(|\pi|) \in \{X^I,X^O,Y^I,Y^O\}$ 
we let $\pi^o$ (the `opposite' variable) be obtained from $\pi$
by interchanging superscripts $I$ and $O$ in $\pi(|\pi|)$.
\end{itemize}
For our classification we use the same terms
internal, terminal and partner as above,
but we must pay special attention to the terminal open pairs,
which we divide into the following three subtypes
(recall that if a pair is not a rung we call it a stringer):
\begin{enumerate}[(a)]
\item
If $e$ is a rung and $\pi Y^I$ and $\pi Y^O$ are both $M$-bounded
we say that $e$ is \emph{simple}. \\
If $e$ is a stringer and $\pi^o Y^I$ and $\pi^o Y^O$ are both $M$-bounded
we say that $e$ is \emph{simple}.
\item
If $w( \pi) = 2M$ and $e$ is the terminal rung 
we say that $e$ is \emph{outer}. \\
If $ w( \pi ) = 2M-1 $, $\pi(|\pi|) = X^I$ 
and $e$ is the terminal stringer then we say that $e$ is \emph{outer}.
\item
If $e$ is not simple or outer we say that $e$ is a \emph{fan end} pair.
\end{enumerate}
To explain this classification, we note the following:
\begin{itemize}
\item Outer pairs are not simple, as adding $Y^O$ to any $\pi'$ 
with $w(\pi') = 2M$ gives a variable not in $\mc{S}_M$
(consider $\pi'=\pi$ if $e$ is the terminal rung 
or $\pi'=\pi^o$ if $e$ is the terminal stringer).
\item Fan end pairs are aptly named, as if there is a fan end pair
it follows from the definition of the $M$-bounded
stacking ensemble $\mc{S}_M$ (see Definition \ref{def:stack2})
that $\pi$ must end with an $(M-1)$-fan.
\end{itemize}

\subsection{Correction terms} \label{sub:correct}

Before starting on the main calculations for the trend hypothesis, 
here we will summarise various correction terms which are negligible
by comparison with the terms appearing in the variation equations.
Besides the higher-order corrections \eqref{eq:stackhi}
to changes in the tracking variable mentioned above,
we also have the following `injectivity' and `fidelity' corrections.

\begin{lemma}[Injectivity] \label{lem:inj}
Suppose $i<I$ and $V=X_{\phi,J,\GG}$ 
is a stacking variable 
or fan extension with $v \ge y$.
Then for any vertex $x \notin A$ (the base) 
there are $\wt{O}(t^{-1} e^2) v$ choices 
of $f \in V$ with $x \in \text{Im}(f)$.
\end{lemma}

\begin{proof}
Fix $a \in V_\GG \sm A$, let $A' = A \cup \{a\}$
and extend $\phi$ to $\phi'$ on $A$ by $\phi'(a)=x$.
It suffices to show that the stated bound
holds for $X_{\phi',J,\GG}$.
Fix $A' \sub B \sub V_\GG$ minimising $S^B_A$.
If $V$ is a stacking variable, then considering 
vertices one by one in the stacking order
we have $S^B_A \ge y$.
If $V$ is a fan then either $B=V_\GG$,
when $S^B_A = v \ge y$, or $B=A'$
(as in the proof of Lemma \ref{lem:backsies}),
so again $S^B_A \ge y$.
As $i<I$, by property (iv) of Definition \ref{def:I0}
the number of choices for $f$ is at most
$L^{4|V_\GG|} v/S^B_A$. 
The lemma follows as $y = 2te^{-2}$.
\end{proof}

\begin{lemma}[Fidelity] \label{lem:fid}
Suppose $i<I$ and $V=S^\pi_{uv}$ is good.
\begin{enumerate}[(i)]
\item There are $O(L^4 v)$ pairs $(f,xy)$ where $f \in V$
such that if $xy$ were the edge $e_{i+1}$ selected 
at step $i+1$ then at least two open pairs 
in $f$ would become closed,
\item Let $V^+$ be a stacking variable obtained 
from $V$ by changing some edge $e$ to an open pair. 
There are $\wt{O}(e^2 v^+)$ choices of $f \in V^+$ such that 
if $f(e)$ were the edge $e_{i+1}$ selected at step $i+1$ 
then some open pair in $f$ would become closed.
\end{enumerate}
\end{lemma}

\begin{proof}
Let $(uv,J,\GG)$ be the extension corresponding to $V$.

For (i), we first note that for each $f \in V$
there are only $O(1)$ choices of $xy \sub \text{Im}(f)$.
Any other $xy$ with the stated property must have 
one of its vertices in $\text{Im}(f)$, say $y$, 
and the open pairs in $f$ closed by $xy$
are of the form $ya$, $yb$ with $a,b$ in $\text{Im}(f)$
where $xa$, $xb$ are edges. As $i<I$, by property (iii) 
of Definition \ref{def:I0} the number of choices for $x$
given $f$ is at most $Z_{ab} < L^4$. This proves (i).

For (ii), note first that for such a configuration to exist
we must have $|\pi| \ge 2$, so $V$ has scaling $v \ge y^2$.
We consider the extension $(uv,J \sm e, \GG)$
corresponding to $V^+$ and any variable $V^*$ corresponding
to an extension $(uv,J^*,\GG)$ with $J^* = (J \sm e) \cup e'$ 
for some $e' \in \binom{V_\GG}{2} \sm \GG$.
It suffices to show $V^* = \wt{O}(e^2 v^+)$.

Note that $v^+ = \hat{q}p^{-1} v = (2te^2)^{-1} v$,
so $e^2 v^+ = (2t)^{-1} v > 1$, 
and $v^* = pv^+ =  \hat{q} v$.
Fix $uv \sub B \sub V_\GG$ minimising $S^B_{uv}$,
taking scalings with respect to $(uv,J^*,\GG)$.
If $e' \not\sub B$ or $B=uv$ then $S^B_{uv} \ge 1$,
as the scaling is the same as in $V^+$,
so by property (iv) of Definition \ref{def:I0}
we have $V^* = \wt{O}(v^*) = \wt{O}(e^2 v^+)$.
If $|B| \ge 4$ we have $S^B_{uv} \ge y^2$,
so $V^* = \wt{O}(v^+/y^2) = \wt{O}(e^2 v^+)$. 

The remaining case is that $|B|=3$ and $B = uv \cup e'$.
Write $B = \{u,v,\aA_j\}$. We cannot have $j=1$,
as $e' \notin \GG$ would then imply $\pi(1) \in \{O,E\}$,
so the assumption of the lemma could not hold:
selecting $f(e')$ as an edge for such $e'$ 
cannot close any other pair in $f$.
Thus $\aA_j$ is adjacent in $\GG$ to at most one of $u,v$,
so $S^B_{uv} \ge pn$, giving 
$V^* = \wt{O}(v^*/pn) = \wt{O}(e^2 v^+)$.
\end{proof}

\subsection{Creation}

Now we will estimate the one-step expected changes
$\mb{E} [ \DD_i(V^e) \mid \mc{F}_i ]$ 
for each $e \in J \sm J[A]$, 
according to the classification of pairs described above.
As for the other ensembles,
the error terms for creation are not as significant
as those for destruction, 
and the calculations do not require self-correction 
or use the fact that $V$ is in its critical window.
We do use $i_V \le i < I$.
Note that we do not include in these calculations
the fidelity corrections (see Lemma \ref{lem:fid}.ii).

\subsubsection{Terminal creation}

Suppose that $e$ is the terminal edge of $\pi$. 
Then $\pi( |\pi|)$ is $E$, $Y^I$ or $Y^O$, 
and if $ \pi( |\pi|) = Y^O$ then $ \pi( |\pi|-1) \neq O$ 
(otherwise $e$ would be partner).
Let $V^+$ be the variable obtained by changing $e$ to an open pair, 
i.e.\ replacing $Y$ by $X$ in $U = \pi( |\pi|)$. 
Then $\mb{E}[ \DD_i( V^e) \mid \mc{F}_i] = 2Q^{-1}V^+$.
For the tracking variable, we note that 
$\DD_i(\mc{T}V^e) = \frac{\mc{T}V}{tn^{3/2}} 
= 2Q^{-1} \mc{T}V^+$ (whether $V$ is standard or partner). 
As $v^+ = v \cdot \frac{ \hat{q} n^{1/2}}{ 2t}$
and $Q=(1+o(e))q$ for $i_V \le i < I$ we have
\begin{align*}
& \mb{E} [\DD_i(\mc{D}V^e) \mid \mc{F}_i]
 = \mb{E} [\DD_i(V^e) - \DD_i(\mc{T}V^e) \mid \mc{F}_i] 
 = 2Q^{-1} \mc{D}V^+ 
 = \pm  (1+o(e)) t^{-1} \dD_{V^+}  v n^{-3/2}. 
\end{align*}

\subsubsection{Partner creation}

Suppose that $e=\aA_x \aA_y$ with $x<y = |\pi|$ 
is the partner edge of $\pi$. We must have
$x=|\pi|-2$, $y=|\pi|-1$ and $\pi = \pi^- O Y^O$. 
In this case, we recall that the tracking variable is
$\mc{T}V = \sum_{f \in V^-} X_{f(\aA_x)} Y_{ f( \aA_x)} Qn^{-2}$,
where $V^- = S_{uv}^{\pi^-}$.
We let $V^+$ be obtained from $V$ by changing $e$ to an open pair.
Then $\mb{E} [\DD_i(V^e) \mid \mc{F}_i] = 2Q^{-1} V^+$.
We also recall that
$\mc{T}V^+ = \sum_{f \in V^-} X_{f(\aA_x)}^2 Qn^{-2}$
and $\DD_i(\mc{T}V^e) = \sum_{ f \in V^-} 
\DD_i(Y_{ f( \aA_x)}) \cdot X_{f( \aA_x)} \cdot Qn^{-2}$,
so $\mb{E} [\DD_i(\mc{T}V^e) \mid \mc{F}_i] = 2Q^{-1} \mc{T}V^+$.
Thus we obtain the same estimate as in terminal creation for
$\mb{E} [\DD_i(\mc{D}V^e) \mid \mc{F}_i]$.

Note that the definition of the tracking variables
isolates variations in $V$ from those in $V^-$,
which is crucial in this calculation:
we cannot afford the larger error term $\dD_{V^-}$.

\subsubsection{Internal creation}

Suppose that $e=\aA_x \aA_y$ with $ x<y < |\pi|$ 
is an internal edge of $\pi$ (which must be a stringer).
Let $V^+$ be obtained from $V$ by changing $e$ to an open pair.
Then $\mb{E}[ \DD_i( V^e) \mid \mc{F}_i] = 2Q^{-1}V^+$.
For the tracking variable,
we recall from \eqref{eq:gentrack} that
$\mc{T}V = \sum_{f \in V^-} \mc{T}_f V$ and
$\DD_i(\mc{T}V^e) =  \sum_{f \in V^{-+}} I^e_f \mc{T}_f V$,
where $V^{-+}$ is obtained from $V^-$ 
by changing $e$ to an open pair and $I^e_f$ is 
the indicator of the event that $e_{i+1}=f(e)$. 
Thus $\mb{E}[ \DD_i(\mc{T}V^e) \mid \mc{F}_i] = 
 \sum_{f \in V^{-+}} \mc{T}_f V =  2Q^{-1} \mc{T}V^+ $, 
so we obtain the same estimate for
$\mb{E} [\DD_i(\mc{D}V^e) \mid \mc{F}_i]$
as in terminal and partner creation.

As for partner creation, it is crucial that $\mc{T}V$
isolates variations in $V^-$ from this calculation.

\subsection{Destruction} \label{sec:destroy}

Now we will estimate the one-step expected changes
$\mb{E} [ \DD_i(V^e) \mid \mc{F}_i ]$ 
for each $e \in \GG \sm \GG[A]$,
according to the classification of pairs described above,
assuming that $V = S^\pi_{uv}$ is in its upper critical window,
so that $\mc{D}V > (f_V+g_V) v$.
As usual, the key point is that every open pair yields 
a self-correcting term of the form $ ( f_V + g_V) 8tvn^{-3/2} $.
We remark that the calculations for terminal open pairs
will be the source of the
most significant error terms in the variation equations.

\subsubsection{Simple destruction} \label{subsub:simple}

Let $ e=\aA_x \aA_y$ be a simple rung, 
i.e.\ the last rung of $\pi$
such that $\pi Y^I$ and $\pi Y^O$ 
both belong to $\mc{S}_M$.
Write $V^I = S_{uv}^{\pi Y^I}$ and $V^O = S_{uv}^{\pi Y^O}$. 
We have 
$$\mb{E}[ \DD_i(V^e) \mid \mc{F}_i] = 2Q^{-1}\sum_{f \in V} 
(Y_{ f(\alpha_x \alpha_y) } + Y_{ f(\alpha_y \alpha_x)} \pm O(1))
= 2Q^{-1}(V^I+V^O \pm O(v)).$$ 
Note that $\mc{T}V^I = \mc{T}V^O = 2 t Q n^{-3/2} V$
and $v^I=v^O=2t \hat{q} n^{1/2}v$.
Since $\DD_i(\mc{T}V^e) = \frac{\DD_i(Q)}{Q}\mc{T}V$,
recalling (\ref{eq:Qstar}) we have
\begin{align*}
& \mb{E} [\DD_i(\mc{D}V^e) \mid \mc{F}_i] 
= \mb{E} \left[ \DD_i( V^e) - \DD_i(\mc{T}V^e) \mid \mc{F}_i \right] \\
& = - 2Q^{-1}(V^I + V^O \pm O(V)) + (2+4SQ^{-1}) Q^{-1} \mc{T}V  \\
& = - 2Q^{-1}(\mc{T}V^I + \mc{T}V^O  \pm v^I\dD_{V^I} \pm v^O\dD_{V^O}) 
+ (8t n^{-3/2} \pm 4\dD_S s q^{-2}) \mc{T}V \pm O(v/q) \\
& = - 8tn^{-3/2} \mc{D}V  
\pm (1+o(1)) 8t (\dD_{V^I}/2 + \dD_{V^O}/2 + \dD_S) vn^{-3/2} \pm O(v/q) \\
& \le - (1+o(1))( f_V + g_V - \dD_{V^I}/2 
 - \dD_{V^O}/2 - \dD_S - O(t^{-1} e^2) ) 8t vn^{-3/2}.
\end{align*}
The same calculation applies if $e$ is a simple stringer 
(using $ \pi^o$ in place of $ \pi$).  
Note that the estimates for $V^I$ and $V^O$ are valid 
even before their activation steps by Lemma \ref{lem:beginning}.iv.
The appearance of their approximation errors
$ \dD_{V^I}$ and $ \dD_{V^O}$ in this calculation indicates
why we need these errors to decrease as we increase the
length of the stacking extensions (see Definition \ref{def:c}).

\subsubsection{Internal destruction}

Suppose that $e= \aA_x \aA_y$ with $ x <y < |\pi| $ is an 
internal open pair (note that we do not include partners here).
We let $W = S^{\pi'}_{uv}$ , where
$ \pi' = \pi |_{y} $ if $e$ is a rung
or $ \pi' = \pi |_{y}^o $ if $e$ is a stringer. 

For each $ f \in W $ let $F_{f, \pi}$ count 
{\em forward extensions} from $ f $ to copies of $V$, 
i.e.\ $F_{f, \pi} = X_{ f,J, \GG}$ with $f:A \to [n]$,
where $A = \{\aA_u,\aA_v, \dots, \aA_y \}$. 

We note that $F_{f, \pi}$ is closely approximated,
up to the injectivity correction from Lemma \ref{lem:inj},
by another variable $V^f_1=S^{\pi_1}_{f(e')}$ 
in the stacking variable, where $e'$ is the active rung 
at step $y$ and $ \pi |_y \circ \pi_1 = \pi $:
we have $F_{f, \pi} = V^f_1 + \wt{O}(t^{-1} e^2)v_1$, so
\[ V = \sum_{f \in W} F_{f,\pi} 
= \sum_{f \in W} \brac{ V^f_1 + \wt{O}(t^{-1} e^2) v_1 }. \]

For the tracking variable,
we recall from \eqref{eq:gentrack} that
$\mc{T}V = \sum_{f' \in V^-} \mc{T}_{f'} V$. Similarly to above, 
we define the forward extension $F_{f,\pi^-}$
from $f \in W$ to copies of $V^-$ and approximate it by
$F_{f,\pi^-} = V^f_2 + \wt{O}(t^{-1} e^2)v_2$,
where $V^f_2=S^{\pi_2}_{f(e')}$
and $ \pi |_y \circ \pi_2 = \pi^-$. Then
\begin{align} \label{eq:id1}
\mc{T}V & = \sum_{f \in W} \sum_{f' \in F_{f,\pi^-}} \mc{T}_{f'} V
= \sum_{f \in W} \brac{ \mc{T}V^f_1 + \wt{O}(t^{-1} e^2) v_1 },
\ \text{ so } \nonumber \\
\mc{D}V & = V - \mc{T}V 
= \sum_{f \in W} \brac{ \mc{D}V^f_1 + \wt{O}(t^{-1} e^2) v_1 }.
\end{align}
Similarly, writing $I^e_f$ for the indicator of the event 
that $e_{i+1}$ closes $f(e)$, noting that
$\DD_i(V^e) =  \sum_{f' \in V} I^e_{f'}
= \sum_{f \in W} I^e_f F_{f,\pi}$ and
$\DD_i(\mc{T}V^e) =  \sum_{f' \in V'} \mc{T}_{f'} V I^e_{f'}
 = \sum_{f \in W} I^e_f
  \sum_{f' \in F_{f,\pi^-}} \mc{T}_{f'} V$, we have
\begin{equation} \label{eq:id2}
\DD_i(\mc{D}V^e) = \DD_i(V^e) - \DD_i(\mc{T}V^e) 
= - \sum_{f \in W} \brac{ \mc{D}V^f_1 + \wt{O}(t^{-1} e^2) v_1 } I^e_f. 
\end{equation}
We also note from $i_V \le i < I$ and \eqref{eq:e*} that
\begin{equation} \label{eq:id3}
 W^* :=  \sum_{f \in W} (Y_{f(xy)} + Y_{f(yx)} ) 
= (1\pm \dD_Y)2Wy. 
\end{equation}
Taking expectations of \eqref{eq:id2} and applying 
Lemma~\ref{lem:pat} (the Product Lemma) we have
\begin{align*}
& \mb{E}  [\DD_i(\mc{D}V^e) \mid \mc{F}_i] 
 = - 2Q^{-1} \sum_{f \in W} (Y_{f(xy)} + Y_{f(yx)} \pm O(1)) 
 (\mc{D}V^f_1 + \wt{O}(t^{-1} e^2) v_1 )\\
& =  - \frac{2 W^* \mc{D}V}{ QW}  
 \pm O( Q^{-1}W \cdot y \dD_Y \cdot v_1 \dD_{V_1} )  
  \pm \wt{O}(t^{-1} e^2) vtn^{-3/2} \\
& \le - (1+o(1))( f_V + g_V -  O(\dD_{V_1} \dD_Y ) 
- \wt{O}(t^{-1} e^2) ) 8tvn^{-3/2} \\
& \le - (1+o(1))( f_V + g_V - \wt{O}(t^{-1}  e^2) ) 8tvn^{-3/2}.
\end{align*}
We used the scaling identities $v = wv_1$ and $v^- = wv_2$.
In the application of the Product Lemma 
on the third line we used \eqref{eq:id1} and \eqref{eq:id3}.

The last line exhibits the same crucial feature
that we saw earlier in product destruction for global variables:
the $O(\dD_{V_1} \dD_Y )$ term is negligible,
as for small $t$ the $t^{-e(V)}$ factor in $g_V$
dominates  the $t^{-e(V_1)}$ factor in $\dD_{V_1}$,
and the $\dD_Y$ factor compensates for the larger
polylogarithmic factor in $\dD_{V_1}$.

\subsubsection{Partner destruction}

Here we consider a partner open pair $e = \aA_x \aA_y$ with $x<y$.
Recall that this means $\pi(|\pi|-1) = O$, 
$ \pi( |\pi|) \notin \{O,E\}$, 
$x=|\pi|-2$ and $y \in \{ |\pi|-1, |\pi| \}$.
Let $\pi = \pi^- O U$ and $V^- = S_{uv}^{\pi^-}$.
Recall from Section \ref{subsub:ptrack} that 
$\mc{T}V = \sum_{f \in V^-} X_{f(\aA_x)} \hat{U}_f$, where
\[
\hat{U}_f =
\begin{cases}
X_{ f( \bB) } \cdot Qn^{-2} & \text{ if } U \in \{ X^I, X^O \} \\
X_{ f( \bB) } \cdot 2tn^{-1/2} & \text{ if } U = Y^I \\
Y_{ f( \bB)} \cdot Qn^{-2} & \text{ if } U = Y^O.
\end{cases}
\]
Note that if both partner pairs are open
then the definitions of $V$ and $\mc{T}V$
are symmetric under swapping the labels 
of $\aA_{|\pi|-1}$ and $\aA_{|\pi|}$,
so we can assume $y=|\pi|-1$.
This would not have been true 
with our usual practice of using the
tracking variable $\mc{T}U$ instead of $\hat{U}_f$; 
the point is that we want the self-correction 
in this section to apply to both partner pairs.
(This property of $\mc{T}V$ for partners is also
essential for our treatment of fan extensions
in Section \ref{subsub:fan}.)
On the other hand, we can think of $\hat{U}_f$
as a proxy for $\mc{T}U$ as it is a reasonable
approximation to $U$:\ as $i<I$ we have
\[ \hat{\mc{D}}U_{f(\aA_x) z} 
:= U_{f(\aA_x) z} - \hat{U}_f 
= O((\dD_U + \dD_{\hat{U}})u), \]
where $\dD_{\hat{U}}=\dD_{Y_1}$ if $U=Y^O$,
otherwise $\dD_{\hat{U}}=\dD_{X_1}$.
Writing $u$ for the scaling of $U$, we have
\begin{align} \label{eq:pd1}
 V & = \sum_{f \in V^-} 
\sum_{z \in X_{f(\aA_x)} \sm \text{Im}(f)} 
(U_{f(\aA_x) z} + O(1)), \ \ \text{ so } \nonumber \\
\mc{D}V & = V - \mc{T}V = \sum_{f \in V^-}
\Big( O(u+x_1) + \sum_{z \in X_{f(\aA_x)}} 
\hat{\mc{D}}U_{f(\aA_x) z} \Big). 
\end{align}
Recalling $\DD_i(\mc{T}V^e) = 
 \sum_{ f \in V^-} \DD_i(X_{ f( \aA_x)}) \hat{U}_f$
and writing $I_{fz}$ for the indicator of the event
that $e_{i+1}$ closes $f(\aA_x) z$, we have
\begin{equation} \label{eq:pd2}
\DD_i(\mc{D}V^e) 
= - \sum_{f \in V^-} \Big( 
\sum_{z \in X_{f(\aA_x)}} 
(\hat{\mc{D}}U_{f(\aA_x) z} \pm O(1)) I_{fz} 
- \sum_{z \in \text{Im(f)}} O(u) I_{fz} \Big) . 
\end{equation}
Also, writing $W = \sum_{f \in V^-} X_{f(\aA_x)}$,
from $i<I$ and \eqref{eq:e*} we have
\begin{equation} \label{eq:pd3}
W^* :=  \sum_{f \in W} (Y_{f(xy)} + Y_{f(yx)} ) 
= (1\pm \dD_Y)2Wy. 
\end{equation}
Taking expectations of \eqref{eq:pd2} and applying 
Lemma~\ref{lem:pat} (the Product Lemma) we have
\begin{align*}
& \mb{E}  [\DD_i(\mc{D}V^e) \mid \mc{F}_i] 
 = \mb{E} [\DD_i(V^e) - \DD_i(\mc{T}V^e) \mid \mc{F}_i] \\
& = - 2Q^{-1} \Big[
\sum_{f \in V^-}  \sum_{z \in X_{ f( \aA_x)}} 
(Y_{f(\aA_x)z}+Y_{z f(\aA_x)} \pm O(1)) \hat{\mc{D}}U_{f(\aA_x) z} 
\Big] \pm O(x_1+u) v^- y/q  \\
& = - \frac{2 W^* \mc{D}V }{QW}
\pm O(   w q^{-1} \cdot y \dD_Y \cdot u ( \dD_U + \dD_{\hat{U}} ))  
\pm \wt{O}(e^2) vn^{-3/2} \\
& =  - (1+o(1))\mc{D}V \cdot 8t n^{-3/2} 
\pm ( O( \dD_Y  \dD_U)  + O(\dD_Y \dD_{\hat{U}}) + \wt{O}(t^{-1} e^2) ) 
tv n^{-3/2}   \\
& \le - (1+o(1))( f_V + g_V - \wt{O}(t^{-1} e^2)  ) 8tvn^{-3/2}.
\end{align*}
In the application of the Product Lemma on the third line 
we used \eqref{eq:pd1} and \eqref{eq:pd3}.
The last line is valid because the product errors
$\dD_Y  \dD_U$ and  $\dD_Y \dD_{\hat{U}}$ are $o(\dD_V)$;
this holds as $\dD_Y$ has sublogarithmic decay
and the power of $t^{-1}$ in $g_V$ is at least 
those in each of $g_U$ and $g_{\hat{U}}$.

\subsubsection{Outer destruction} \label{subsub:outer}


Let $ e=\aA_x \aA_y$ be an outer rung, 
i.e.\ $e$ is terminal and $ w( \pi) = 2M$.  
We cannot apply the same analysis as for simple destructions,
as $\pi Y^O \notin \mc{S}_M$, so instead we use backward extensions,
which are controllable by Lemma~\ref{lem:backsies}.

We let $Q'$ be the set of $ab \in Q$ 
such that $ \{a,b\} \cap \{u,v\} = \es$,
and for each $ab \in Q'$ let $B_{uvab} $ 
count backward extensions that map 
the last rung of $S^\pi_{uv}$ to the open pair $ab$;
thus $V = \sum_{ab \in Q'} B_{uvab}$.

Let $b$ and $\dD_B$ be the scaling and error function
for the backward extension. Then $b = v/q$ 
and $\dD_B = O(1+t^{-e(V)}) e^\dD$.
Note also that $Q-Q'=O(x_1)$ and 
$S = \sum_{ab \in Q} Y_{ab} 
= O(x_1 y) + \sum_{ab \in Q'} Y_{ab}$.
Recalling \eqref{eq:Qstar} and
$\DD_i(\mc{T}V^e) = \frac{\DD_i Q}{Q} \mc{T}V$,
by the Product Lemma (Lemma \ref{lem:pat}) we have
\begin{align*}
& \mb{E} [\DD_i(\mc{D}V^e) \mid \mc{F}_i] 
= \mb{E} [\DD_i(V^e) - \DD_i(\mc{T}V^e) \mid \mc{F}_i] \\
& = - 2Q^{-1} \sum_{ab \in Q' } X_{uvab} (Y_{ab}+Y_{ba} \pm O(1)) 
+ \frac{4S +2Q}{Q^2} \mc{T}V \\
& = - \frac{ 2 V}{ Q Q' } \Big( 2 S - O(yx_1) \Big) 
\pm \frac{4 Q'}{Q} \cdot y \dD_Y \cdot b \dD_B  
+ \frac{4S +2Q}{Q^2} \mc{T}V   \pm O(v/q) \\
& = - \frac{ 4S}{ Q^2} \mc{D}V  
\pm O(  \dD_Y \dD_B  + t^{-1}e^2)  tv n^{-3/2} \\
& \le - (1+o(1))( f_V + g_V 
- O({\bf1}_{ e(V)=0 } \dD_Y e^\dD) - O(t^{-1} e^2) ) 8tvn^{-3/2}.
\end{align*}
The last line used $\dD_Y \dD_B = o(g_V)$ when $e(V) \ge 1$,
which holds as $\dD_B$ has sublogarithmic decay (using $i<I$)
and the power of $t^{-1}$ in $V$ is at least that in $Y$.
Thus this term is negligible unless $e(V)=0$,
in which case we can substitute $\dD_B = O(e^\dD)$.

Note that the same estimate applies 
if $e$ is an outer stringer 
(using $ \pi^o$ in place of $ \pi$). 

\subsubsection{Fan end destruction} \label{subsub:fan}

For destruction, it remains to consider the case 
when $e=\aA_x \aA_y$ is a fan end,
i.e.\ $\pi$ ends with an $(M-1)$-fan and $e$ is the terminal rung.
We cannot apply the analysis from simple destructions,
as $\pi Y^I \notin \mc{S}_M$, so instead we use controllability
of fan extensions (see Lemma \ref{lem:backsies}).

Let $V^* = S^{\pi^*}_{uv}$, where $\pi^* = \pi|_{x} O$,
i.e.\ $V^*$ is obtained from $V$ by deleting all of the fan
except its first pair $\aA_{x-1}\aA_x$ 
and last pair $e=\aA_x \aA_y$.
Then $V = \sum_{f \in V^*} F_{f, \pi}$,
where $F_{f, \pi}$ denotes the forward extension,
which is closely approximated by the $(M-1)$-fan 
extension $V_1$ from $f(\aA_{x-1}\aA_x\aA_y)$;
by Lemma \ref{lem:inj} we have
$F_{f, \pi} = V_1 + O(t^{-1} e^2) v_1$.
We recall that $V_1$ is controllable
by Lemma~\ref{lem:backsies}.

In the calculation below for 
$\mb{E} [\DD_i(\mc{D}V^e) \mid \mc{F}_i]$
we require the following estimate for the expected closures
of the terminal open pair $\aA_x \aA_y$ in copies of $V^*$,
which are described by 
\[ V^*_{\text{close}} := 2Q^{-1} \sum_{f \in V^*} 
(Y_{f(\aA_x \aA_y)} + Y_{f(\aA_x \aA_y)}).\] 

\begin{lemma} \label{lem:Vclose}
Let $\pi^\bullet = \pi|_{x} E$,
$V^x=S^{\pi|_x}_{uv}$,
$V^\bullet=S^{\pi^\bullet}_{uv}$, 
$V^{*I} = S^{\pi^* Y^I}_{uv}$ and 
$V^{*O} = S^{\pi^* Y^O}_{uv}$. Then 
\[ V^*_{\text{close}} = 8tn^{-3/2} \Big[
V^* \pm (1+o(1)) \big( \dD_{ V^*} +  \dD_{V^\bullet} 
+ \dD_{V^{*I}} +  \dD_{V^{*O}} 
+ O(\dD_{X_1}+\dD_{Y_1})\dD_{X_1} \big) v^*/2 \Big].\]
\end{lemma}

\begin{proof}
First we emphasize that all variables defined in the statement
of the lemma are in the stacking ensemble,
and this fact makes crucial use
of Definitions \ref{def:stack1} and \ref{def:stack2}.
The point is that as non-terminal $OX^I$ and $OY^I$ 
are forbidden, the fan must start with
$ \pi(x+1) \in \{ X^O, Y^O \}$,
and also $w(\pi) \le 2M-1$ as we do not
allow a strict subsequence of weight $2M$,
so $ w(\pi|_x) \le w(\pi) - 1 \le 2M-2$.   
Now
\begin{align*}
& \sum_{f \in V^*} 
(Y_{f(\aA_x \aA_y)} + Y_{f(\aA_x \aA_y)} \pm O(1))
= V^{*I} + V^{*O} \pm O(v^*) \\
& = \mc{T}V^{*I} + \mc{T}V^{*O}
\pm (\dD_{V^{I*}}y  +  \dD_{V^{*I}}y + O(1))v^*, 
\end{align*} 
where, as $V^{*I}$ and $V^{*O}$ are both partner variables, 
by Lemma \ref{lem:pat} we have
\begin{align*}
\mc{T}V^{*I} & = \sum_{f \in V^x} X_{ f(\aA_x) }^2 \cdot 2tn^{-1/2} 
=  2tn^{-1/2} \cdot V^* V^* / V^x 
\pm O \big( tn^{-1/2} v^x ( x_1 \dD_{X_1})^2 \big)
\text{ and }\\
\mc{T}V^{*O} & = \sum_{f \in V^x} X_{ f( \aA_x) } Y_{ f( \aA_x)} \cdot Qn^{-2}
=  Q n^{-2} \cdot V^* V^\bullet / V^x 
\pm O \big( \hat{q} v^x ( x_1 \dD_{X_1}) ( y_1 \dD_{Y_1}) \big).
\end{align*}
The lemma now follows from
$V^* = \mc{T}V^* \pm \dD_{V^*} v^*$ and
$V^\bullet = \mc{T}V^\bullet \pm \dD_{V^\bullet} v^\bullet$,
where $\mc{T}V^* = Qn^{-1} V^x$
and $\mc{T}V^\bullet = 2tn^{1/2} V^x$, 
so $V^x$ cancels (this is crucial to avoid
a larger $\dD_{V^x}$ error term).
\end{proof}

Now recalling $F_{f, \pi} = V_1 + O(t^{-1} e^2) v_1$,
using $\DD_i(\mc{T}V^e) = \frac{\DD_i Q}{Q} \mc{T}V$
and  \eqref{eq:Qstar}, by Lemma \ref{lem:pat} 
\begin{align*}
& \mb{E} [\DD_i(\mc{D}V^e) \mid \mc{F}_i] 
 = - 2Q^{-1} \sum_{f \in V^*} 
 (Y_{f(\aA_x \aA_y)} + Y_{f(\aA_x \aA_y)} \pm O(1)) F_{f,\pi} 
 + \frac{4S + 2Q}{Q^2} \mc{T}V  \\
& = - V^*_{\text{close}} V/V^*   \pm O(t^{-1}e^2)v^* v_1 y/q
\pm O(  v^* q^{-1} \cdot y  \dD_Y \cdot v_1 \dD_{V_1} )  
+ \frac{4S}{Q^2} \mc{T}V \pm O(v/q)\\
& = -8tn^{-3/2} V +  (1 + (1+o(1))\dD_S) 8tn^{-3/2} \mc{T}V \\
&  \qquad \pm (1+o(1)) \tfrac{1}{2}
 \big( \dD_{ V^*} +  \dD_{V^\bullet} 
+ \dD_{V^{*I}} +  \dD_{V^{*O}} 
+ O(\dD_Y \dD_{V_1} + t^{-1}e^2 + (\dD_{X_1}+\dD_{Y_1})\dD_{X_1} ) 
\big)  8tvn^{-3/2} \\
& \le - (1+o(1)) \left( f_V + g_V 
 - \tfrac{1}{2}(\dD_{V^*}+\dD_{V^\bullet}
    +\dD_{V^{*I}}+\dD_{V^{*O}}) - \dD_S 
 - O(t^{-1} e^2) - O(  \dD_Y e^\dD )  \right) 8tv n^{-3/2}.
\end{align*}
In the third line we applied Lemma \ref{lem:Vclose}.
In the last line, similarly to the case of outer destruction, 
we note that  $(\dD_{X_1}+\dD_{Y_1})\dD_{X_1} = o(\dD_V)$,
as $\dD_{X_1} = \wt{O}(\dD_V)$ and 
$\dD_{X_1}+\dD_{Y_1}$ has sublogarithmic decay.
Similarly, if $e(V) \ge 1$ then $\dD_Y = \wt{O}(\dD_V)$
and $\dD_{V_1}$ has sublogarithmic decay, 
so $\dD_Y \dD_{V_1} =  o(\dD_V)$.
Thus the only product error is
$O(\dD_Y \dD_{V_1}) = O(  \dD_Y e^\dD )$ when $e(V)=0$.

\subsection{Trend hypothesis and variation equations}

Now we combine all the estimates in this section
to verify the trend hypothesis, i.e.\ that 
if $V$  is in its upper critical window then
$\mc{Z}V = \mc{D}V - \dD_V v$ forms a supermartingale,
given the choice of constants $c_V$
made in Definition \ref{def:c}.

\begin{lemma} \label{lem:stktrend}
If $i_V \le i < I$ and $\mc{D}V > (f_V+g_V) v$ 
then $\mb{E} [ \DD_i \mc{Z}V \mid \mc{F}_i]  \le 0$.
\end{lemma}

\begin{proof}
Throughout the proof we will measure expected changes
using the `yard stick' $8tvn^{-3/2}$,
which is an approximation for the expected change in $V$
due to destruction by some fixed open pair.
Recall that we decompose the one-step change 
in $V = X_{\phi,J,\GG}$ by its pairs $e$ as
\[ \DD_i(V) = \sum_{e \in \GG \sm \GG[A]} \DD_i(V^e) \pm F_i(V),\]
where $F_i(V)$ is a fidelity correction, 
which by Lemma \ref{lem:fid} satisfies
\[ \mb{E}[F_i(V) \mid \mc{F}_i ] 
= O(L^4 v/q) + \wt{O}(e^2 v^+/q)
= (t^{-1} + t^{-2} 1_{e(V)>0}) \wt{O}(e^2) \cdot tvn^{-3/2}. \]
Recall also that we decompose the one-step change
in the tracking variable as
\[ \DD_i(\mc{T}V)
= \sum_{e \in \GG \sm \GG[A]} \DD_i(\mc{T}V^e) + H_i(V),\]
where the higher-order correction term is
\[ H_i(V) = O(t^2+t^{-2}) n^{-3} v
 = O(n^{-5/4}) \cdot tvn^{-3/2} 
 \text{ for } n^{-1/4} \le t = O(L).\]
Besides the fidelity and higher-order terms,
the remaining contributions to
$ \mb{E} [\DD_i(\mc{D}V) \mid \mc{F}_i] 
= \mb{E} [\DD_i(V) - \DD_i(\mc{T}V) \mid \mc{F}_i] $
are obtained by summing
$ \mb{E} [\DD_i(\mc{D}V^e) \mid \mc{F}_i] 
= \mb{E} [\DD_i(V^e) - \DD_i(\mc{T}V^e) \mid \mc{F}_i]$
over all $e \in \GG \sm \GG[A]$.

There are $e(V)$ edges
each giving a creation term of
\[ \pm  (1+o(e)) t^{-1} \dD_{V^+}  v n^{-3/2}
= (1+o(e)) \tfrac{\dD_{V^+}}{8t^2} \cdot 8t v n^{-3/2} .\]
 
There are $o(V)$ open pairs
each giving a destruction term in which
the main term is a self-correction term of
\[ - (1+o(1)(f_V + g_V) 8tvn^{-3/2}.\]
For open pairs that are partner or internal
the only other error term is
$\wt{O}(t^{-1} e^2) \cdot tvn^{-3/2}$,
which we can absorb into the fidelity term.
The terminal open pairs (of which there are one or two)
contribute an additional error term, 
depending on the form of $\pi$,
which we denote by $\dD_{{\rm add}} \cdot 8tvn^{-3/2}$.

We claim the following bound:
\[ |\dD_{\rm add}| \le 0.49\dD_V   + O(  \dD_Y e^\dD ).  \]
To see this, we first suppose $\pi \ne O$
and consider each of the 
three types of terminal open pair.
 \begin{itemize}
\item
The only contribution to  $ \dD_{\rm add} $ 
from an outer open pair is $O(  \dD_Y e^\dD ) $.
\item
The contribution to  $ \dD_{\rm add} $ from a simple open pair 
is $  (1 + o(1))(\dD_{V^I}/2 + \dD_{V^O}/2 + \dD_S) $.
We can absorb $ \dD_S$
into the $O(\dD_Y e^\dD) $ term.
From Definition \ref{def:c} we have
\begin{equation}
\label{eq:simplecond}
c_{V^O}  = c_{V^I} = c_V/9, 
\end{equation} 
so $ \dD_{V^I}/2 + \dD_{V^O}/2  \le \dD_V/9 $, 
and we can bound this contribution to  $ \dD_{\rm add} $  
by $\dD_V/8 + O(  \dD_Y e^\dD ) $.
\item
The contribution to  $ \dD_{\rm add} $ from a fan end open pair is
\[ (1+o(1))( \tfrac{1}{2}(\dD_{V^*}+\dD_{V^\bullet}+\dD_{V^{*I}}+\dD_{V^{*O}}) + \dD_S ) + O(  \dD_Y e^\dD ) .\]
Again $ \dD_S = O(  \dD_Y e^\dD )$.
The sequences defining $V^*$ and $V^\bullet$ 
each have $M-1$ fewer symbols than $\pi$,
but this is compensated for by an additional `O' or `E'.
Thus Definition \ref{def:c} gives
\begin{align}
c_{V^*} & = c_{V^\bullet} = c_V/9 , \text{ and }
\label{eq:fanendcond1} \\
c_{V^{*I}} & = c_{V^{*O}} = c_V/81. 
\label{eq:fanendcond2}
\end{align}
Thus $ \tfrac{1}{2}(\dD_{V^*}+\dD_{V^\bullet}+\dD_{V^{*I}}+\dD_{V^{*O}})  
\le \dD_V/9 + \dD_V/ 81 $,
so we can bound this contribution to  $ \dD_{\rm add} $  
by $\dD_V/8 + O(  \dD_Y e^\dD ) $.
\end{itemize}
As $ V $ can have at most two terminal open pairs,
this proves the claim when $\pi \ne O$. 
If $\pi=O$ then the only contribution 
is from the simple open pair;
recalling the adjustment in Definition \ref{def:c}
we have $c_{V^O}  = c_{V^I} = 2.2c_V/9$,
so the claim also holds in this case. 

Combining all the estimates so far gives
\begin{align*}
\frac{\mb{E}[ \DD_i(\mc{D}V) \mid \mc{F}_i]}{8tvn^{3/2}} 
& \le - (1+o(1)) o(V) (f_V + g_V)
 + (1+o(e)) e(V) \frac{\dD_{V^+}}{8t^2} \\
& \quad \quad  + 0.49\dD_V   + O(  \dD_Y e^\dD )
 + (t^{-1} + t^{-2} 1_{e(V)>0}) \wt{O}(e^2).  
\end{align*}
By Lemma \ref{lem:seconds} we have
\[ \frac{\DD_i(v\dD_V)}{8tvn^{3/2}} 
\ge \brac{ \tfrac{e(V)}{8t^2} - o(V) + O(t^{-1}n^{-1}) } \dD_V 
+  \brac{ 4t\dD_V   + (\tT'/\tT - e(V)t^{-1}) 2g_V }/8t. \]
By Definition \ref{def:c},
as $V$ is not a vertex degree we have $c_V = c_{V^+}$,
so as in the proof of Lemma \ref{lem:contrend}
we have $\dD_{V^+}-\dD_V = 2(g_{V^+}-g_V)$  
and $g_{V^+} \le 2t g_V$ 
(with no $ V^+$ term if $e(V)=0$).
Thus
\begin{align*}
& \frac{\mb{E}[ \DD_i(\mc{Z}V) \mid \mc{F}_i ]}{8tvn^{-3/2}}
= \frac{\mb{E}[ \DD_i(\mc{D}V) \mid \mc{F}_i]}{8tvn^{3/2}} 
- \frac{\DD_i(v\dD_V)}{8tvn^{3/2}} \\
& \le \tfrac{e(V)}{8t^2} \cdot 2(g_{V^+}-g_V) 
 + o(V) g_V + 0.49(f_V + 2g_V)   - f_V/2  
- (\tfrac{ \tT'/\tT }{8t} - \tfrac{e(V)}{8t^2} + \tfrac{1}{2})\cdot 2g_V   \\
& \quad + O(   \dD_Y e^{\dD}) 
 + (t^{-1} + t^{-2} 1_{e(V)>0}) \wt{O}(e^2)  
 + o( \dD_{V^+} e t^{-2}) + o(\dD_V) \\
& \le g_V \left[ o(V)  +  \tfrac{e(V)}{2t}  - \tfrac{ \tT'/\tT }{ 4t} 
- \tfrac{1}{50} \right] -  \tfrac{ f_V}{100}
+ O(   \dD_Y e^{\dD}) 
 + (t^{-1} + t^{-2} 1_{e(V)>0}) \wt{O}(e^2)  
 + o( \dD_{V^+} e t^{-2}) + o(\dD_V).
\end{align*}
To conclude the proof, it remains to check that 
this final expression is negative. This holds as
$ - g_V \tT'/( 4t\tT ) $ dominates when  $  g_V/t > f_V$ 
and $ - f_V/100 $ dominates otherwise.
Here we recall that  $ \tT'/\tT = K > M^ 6$ for $ t <1$,
and also use the later activation step
(see Definition \ref{def:active}) for the case $e(V)=1$
to see that the $t^{-2} 1_{e(V)>0} \wt{O}(e^2)$
term is negligible.
\end{proof}

Having verified the trend and boundedness hypotheses, 
Theorem \ref{stacking} now follows from
Lemmas \ref{th+bh} and \ref{lem:beginning}.

\section{Independence number and upper bound} \label{sec:indep}

In this section we prove Theorem~\ref{indep} 
on the independence number
and establish the upper bound that completes 
the proof of Theorem~\ref{edges}
on the size of the final graph in the process.  
We will use union bound arguments that take advantage
of our tight control of the evolution of key parameters 
until the process is very near its end.  

We start by giving an intuitive overview
of these arguments as applied to the independence number.  
Suppose we wish to estimate the probability that some set $K$ 
of $\TT(\sqrt{ n \log n})$ vertices is independent.
At any step $i$, with corresponding time $t=in^{-3/2}$,
we would expect that $K$ contains $\approx \hat{q}(t) |K|^2 $ 
open ordered pairs. The total number of open pairs at step $i$
is $Q(i) \approx q(t) = \hat{q}(t) n^2$,
so the probability that $K$ remains
independent throughout the period
in which we track the process should be 
roughly $( 1 - |K|^2/n^2)^{i_{max}}$.
If this were true, we could estimate $Pr( \alpha(G) > k )$ by
$$\tbinom{n}{k} \left( 1 - k^2/n^2 \right)^{i_{max}}
< \exp( k\log \tfrac{en}{k} - i_{max} k^2/n^2 ),$$
which is $o(1)$ for $k > (1+o(1))\sqrt{2n\log n}$,
as required to prove Theorem~\ref{indep}.  

However, it is {\em not} true that every such $K$ 
has $\approx \hat{q}|K|^2$ open ordered pairs;
indeed, if $K$ has a large intersection with the neighbourhood 
of some vertex then $K$ contains significantly fewer open pairs.
Thus we require a much more delicate union bound calculation 
that takes into account the way in
which vertex neighbourhoods intersect $K$.

We stress that throughout this section we assume $ I > i_{max}$.  
Under this assumption, if $i \le i_{max}$ 
the good event $ \mc{G}_i$ holds and every good $V$ in the three ensembles satifies $ |V - \mc{T}V| \le \delta_V v$.  This
assumption is valid as the events in the union 
we define are all intersected with the event $ I > i_{max} $.  
Formally speaking, in Section~\ref{sec:pfindep} 
we bound the probability of the event that $ I > i_{max}$ and the 
independence number of $ G( i_{max})$  is large, 
and in Section~\ref{sec:upper} we bound the probability that 
$I > i_{max}$ and the maximum degree has the potential 
to become large in the steps that follow $ i_{max}$.

We also stress that throughout the section
`neighbour' means `neighbour in $G(i_{max})$'
and `$N(x)$' means `$N_{G(i_{max})}(x)$'.

To lighten notation in our calculations, we introduce
the following notation for the number of steps
in which we track the process and the deterministic prediction 
for the vertex degrees:
\begin{equation} \label{eq:m+d}
m = i_{max} = \tfrac{1}{2} \sqrt{1/2 - \eps}\ n^{3/2} (\log n)^{1/2}
\ \ \text{ and } \ \ 
d = 2t_{max}\sqrt{n} = 2m/n 
= \sqrt{(1/2 - \eps)n\log n}.
\end{equation}
In the course of the proof, we will control various
polylogarithmic factors using absolute constants 
\begin{equation*} \label{eq:agb}
0< \alpha < \gamma <  \beta.  
\end{equation*}
To clarify the role of these constants 
we will not substitute actual values,
but for concreteness we note that we could let
$ \alpha = 25, \gamma = 50, \beta = 600 $. 
When these polylog factors are unimportant we 
will use `tilde' notation as before:\ recall that 
$f(n) = \wt{O}(g(n))$ and $g(n) = \wt{\OO}(f(n))$ 
mean that $f(n) \le (\log n)^A g(n)$ for some absolute constant $A$.

Our proofs require some preliminary facts 
established in Section~\ref{sec:pre} 
(these are mostly density estimates for edges and open pairs).
We prove Theorem~\ref{indep} in Section~\ref{sec:pfindep},
and then apply a similar (and easier) argument 
in Section~\ref{sec:upper} to prove Theorem~\ref{edges}.

\subsection{Preliminaries}
\label{sec:pre}

This subsection contains some density estimates
for edges and open pairs, and also some more intricate
configurations that will play a crucial role 
in the argument in Section~\ref{sec:pfindep}.
These estimates will be obtained
from the critical interval method 
as described in Section~\ref{sec:strategy}.  
We start with an observation that will be used many times 
in this section to estimate the one-step variances in some
extension variable $V=X_{\phi,J,\GG}$ due to destruction. 
This will be applied as in Section~\ref{sec:decomp} 
to bound the one-step conditional variance 
$ \Var_V(i) = \Var( \mc{Z}V(i)  \mid \mc{F}_{i-1} )$
via a sum over pairs $e$
in the configuration of the change in $ \mc{Z}V$
due to the change of status of $ f(e)$.
Thus if $e=uv$ is an open pair in this configuration 
we want to estimate the one-step variance $\Var_e$ 
due to closing $f(e)$.

\begin{lemma} \label{lem:vard}
Consider any extension variable $V=X_{\phi,J,\GG}$ 
and open pair $e \in \GG \sm J$ of $\GG$.
Suppose at step $i$ that the number $N^V_e(i)$
of injections $f$ counted by $V$ destroyed 
by closing $f(e)$ is bounded as $N^V_e(i) \le N$,
for some constant $N$. Then 
\[ \Var_e := \Var( N^V_e(i)  \mid \mc{F}_{i-1} )
\le (1+o(1)) 8tn^{-3/2} NV.\]
\end{lemma}

\begin{proof}
Consider the bipartite graph $H$ 
with parts $(A,B)$, where $A$ is the set of injections 
counted by $V$, $B = Q$ is the set of ordered open pairs, 
and $f \in A$ is adjacent to $b \in B$ 
if selecting $b$ as an edge closes $f(e)$. 
By assumption $d_H(b) \le N$ for all $b \in B$. 
We also have $e(H) = 2\sum_{f \in V} (Y_{f(uv)} + Y_{f(vu)}) 
= (1+o(1)) 4yV$. Then
$\Var_e \le Q^{-1} \sum_{b \in B} d_H(b)^2 
\le (1+o(1)) q^{-1} e(H) N = (1+o(1)) 8tn^{-3/2} NV$.
\end{proof}

With this observation in hand, we turn next to 
some lemmas on counting open pairs.

\begin{defn}
For any set $S$ let $Q_S(i)$ be the number of 
ordered open pairs in $S$ at step $i$.
For any sets $A,B$ let $Q_{AB}(i)$ 
be the number of open pairs $ab$ with $a \in A$, $b \in B$ at step $i$.
\end{defn}

\begin{lemma} \label{opens}
Whp for any set $S$ of size $s$, 
step $ i \le i_{max}$ and $\psi \ge n^{-\eps/5}$,
\begin{enumerate}[(i)]
\item if $s \ge n^{1/4}$ and any vertex $x$
has $|N(x) \cap S| \le L^{-10} \psi^2 \hat{q}s$  
then $Q_S = (1 \pm \psi) \hat{q}s^2$,
\item if $s \ge L^{11} \psi^{-2} \sqrt{n}$ 
then $Q_S = (1 \pm \psi) \hat{q}s^2$,
\item if $s < 2L^{12} \sqrt{n}$ then $Q_S < L^{13} s \hat{q}\sqrt{n}$.
\end{enumerate}
\end{lemma}

\begin{proof}
First consider statements (i) and (ii). 
We use critical window analysis for $t \ge n^{-0.4}$ to prove the bound 
$Q_S = (1 \pm \delta_O) \hat{q}s^2$, where $\delta_O = (1+t/L) \psi/2$. 
This suffices as $\delta_O \le \psi$. We use the window $[(1+\delta_O-g_O)\hat{q}s^2, (1+\delta_O)\hat{q}s^2]$, where $g_O = \psi / (40L^2)$. 

First we use coupling to the Erd\H{o}s-R\'enyi process to show that 
whp $Q_S$ does not enter the critical window at $t=n^{-0.4}$.
This follows from 
the trivial upper bound $Q_S \le s^2$, and the lower bound $Q_S \ge s^2 - 5n^{0.2}s$, obtained by subtracting the number of paths of length 
$2$ starting in $S$ in the random graph.

Next we establish the trend hypothesis that 
$\mc{Z}Q_S = Q_S -  \hat{q}s^2 - \delta_O \hat{q}s^2$
 is a supermartingale while $Q_S$ is in its critical window. 
(Note that our tracking variable in this case 
is the deterministic function $ \hat{q}s^2 $.)
The expected change in $Q_S$ is $$\mb{E}[ \DD_i Q_S \mid \mc{F}_i] = - 2Q^{-1} \sum_{ab \in Q_S} (Y_{ab}+Y_{ba}+1 ) = -8tn^{-3/2} (1 \pm O(\dD_Y)) Q_S.$$ 
We also note that $\DD_i(\hat{q}s^2) = (- 8tn^{-3/2} + O(L^2n^{-3}) ) \hat{q}s^2$ and 
$\DD_i(\delta_O \hat{q}s^2) = (1+o(1)) ((L+t)^{-1} - 8t) n^{-3/2} \delta_O \hat{q}s^2$.
When $ Q_S$ is in the critical interval we have 
\begin{align*}
\mb{E}[ \DD_i \mc{Z}Q_S \mid \mc{F}_i] & \le    -8tn^{-3/2} \hat{q}s^2  (1 + \delta_O - g_O - O( \dD_Y))  \\ 
& \hskip2cm+  \hat{q} s^2 \left(  8tn^{-3/2} - O(L^2n^{-3})  -  (1+o(1)) ((L+t)^{-1} - 8t) n^{-3/2} \delta_O \right) . \\
& \le  
- 8tn^{-3/2} \hat{q}s^2 \cdot ( \delta_O - g_O  - O(\dD_Y) +  (\tfrac{1}{8t(L+t)} - 1) \delta_O ) 
\end{align*}
Using $\tfrac{\delta_O}{8t(L+t)} \ge 2g_O$ 
and $\dD_Y \le n^{-\eps/4} = o(g_O)$ by \eqref{eq:e*}, 
when $ Q_S $ is in the critical interval
we have $\mb{E}[ \DD_i \mc{Z}Q_S \mid \mc{F}_i] \le 0$,
so the trend hypothesis holds.

To complete the proof of statements (i) and (ii) 
we will apply Freedman's inequality
and take a union bound over $S$.
To account for the number $\tbinom{n}{s}$
of events in the union, it suffices to establish 
the following strengthened form 
of the bounded hypothesis  (\ref{eq:boundVar}) and (\ref{eq:boundN}),
where we write  $N_O$ and $\Var_O$ for the maximum 
one-step change and conditional variance of $Q_S$. 
\begin{equation}
\label{eq:boundVarS}
g_O(t)^2 (\hat{q}(t)s^2)^2 
= \omega \left( \Var_O(i) (n\log n)^{3/2} s \right),
\end{equation}  
\begin{equation}
\label{eq:boundNS}
g_O(t) \hat{q}(t)s^2 
= \omega \left( N_O(i) (\log n) s \right).
\end{equation}
Since $g_O = \psi / (40L^2)$, it suffices to show
$N_O \le 2L^{-10} \psi^2 \hat{q}s$, 
as by Lemma \ref{lem:vard} this also implies 
$\Var_O \le L^{-4} n^{-3/2} s^{-1} (2L^{-2}\psi \hat{q}s^2)^2$. 
To see this bound on $N_O$ we use $N_O = O(y)$ for statement (ii),
or $N_O \le |N(x) \cap S| + |N(y) \cap S|$ and our
assumption on neighbourhoods in $S$ for statement (i).

It remains to prove (iii), which is a one-sided bound 
rather than a dynamic concentration statement, but we
can still apply a modified form of the critical interval method.
Writing $F_O = (1+t/L) L^{13} s \hat{q}\sqrt{n}/2$, 
it suffices to show $ Q_S \le F_O$ for all $S$ with 
high probability. Note that the bound is trivial for $t \le 1$, 
as $s < 2L^{12} \sqrt{n}$ implies $Q_S \le s^2 < F_O$. 
For $t \ge 1$ we use critical window analysis 
with the window $[F_O-G_O, F_O]$, where $G_O = F_O / (40L^2)$.
(Here we use capital letters $F$, $G$ 
to distinguish our notation for absolute errors 
from our usual notation $f$, $g$ for relative errors.)

When $Q_S$ is in the critical window we estimate $\mb{E}[ \DD_i Q_S \mid \mc{F}_i] \le - (1 + o(1))8tn^{-3/2} (F_O - G_O)$. 
We write $\mc{Z}Q_S = Q_S - F_O$ and note that 
$F'_O =((L+t)^{-1} - 8t) F_O$. 
Again using $\tfrac{F_O}{8t(L+t)} \ge 2G_O$,
we obtain the trend hypothesis
\[\mb{E}[ \DD_i \mc{Z}Q_S \mid \mc{F}_i] \le
- (1 + o(1))8tn^{-3/2} \cdot ( F_O - G_O +  (\tfrac{1}{8t(L+t)} - 1) F_O ) 
\le 0.\]
For the boundedness hypothesis, 
accounting for the union bound as in (i) and (ii),
and noting that $\mc{Z} Q_S(i) < - G_O(t)$ at the step before
this variable enters the critical interval, it suffices to show
\begin{equation}
\label{eq:boundonesided}
G_O(t)^2= \omega \left( \Var_O(i) (n\log n)^{3/2} s \right) 
 \ \ \ \text{ and } \ \ \ G_O(t) = \omega \left( N_O(i) (\log n) s \right).
\end{equation}
We use the bound $N_O \le 2y \le L^{-9} s^{-1} G_O$. 
By Lemma \ref{lem:vard} this implies 
$$\Var_O = O\left( t n^{-3/2} \cdot  F_O \cdot L^{-9} s^{-1} G_O  \right) = O \left( L^{-6} s^{-1} G_O^2 n^{-3/2}\right),$$ 
and the desired inequalities follow.
\end{proof}


\begin{lemma} \label{opens-bip}
Suppose $r,s \ge n^{1/4}$, $\psi \ge n^{-\eps/5}$ 
and $h \le L^{-10} \psi^2 \hat{q} \min\{r,s\}$.
Then whp we have $Q_{RS} = (1 \pm \psi) \hat{q}rs$ 
for any sets $R$, $S$ of respective sizes $r$, $s$ 
such that any vertex 
that has a neighbour in one of these sets 
has at most $h$ neighbours in the other.
\end{lemma}
\noindent
Note that Lemma~\ref{opens-bip} is simply a bipartite version of Lemma~\ref{opens}(i). The proof is essentially the same, so we omit it, 
noting that the condition on  $h$ is
needed for the boundedness hypothesis.

Next we establish some density estimates. 
\begin{defn}
For a set $S$, let $\eta_S$ denote 
the number of edges of $G(i_{max})$ in $S$.
\end{defn}
\begin{lemma} \label{density}
Whp for any set $S$ of size $s$
\begin{enumerate}[(i)]
\item if $s \ge L^{12} \sqrt{n}$ then $\eta_S < L^2n^{-1/2} s^2$,
\item if $s < 2L^{12} \sqrt{n}$ then $\eta_S < L^{15} s$.
\end{enumerate}
\end{lemma}

\begin{proof}
For (i), we estimate the probability that some such $S$ 
spans $M := L^2n^{-1/2} s^2$ edges, taking a union bound over $S$ 
and the steps at which the edges are chosen, for which 
there are $\binom{n}{s} \binom{m}{M}$ choices. 
For a specified step at time $t$, 
the probability of choosing an edge in $S$
is $Q_S(t)/Q(t) = (1+o(1))s^2/n^2$, 
using Lemma~\ref{opens}(ii) with $\psi=L^{-1/2}$. 
Thus the failure probability $p_0$ satisfies
\[ p_0 \le \binom{n}{s} \binom{m}{M} ((1+o(1))s^2/n^2)^{M}.\]
Noting that $M \ge L^{14} s$, the required
estimate $p_0=o(1)$ follows from
\begin{multline*}
\log p_0 \le O(s \log n) 
+ M \log \tfrac{ em}{M} 
+ M \log ( (1+o(1)) s^2/n^2)
= O( s \log n) + M ( O(1) - \log L ) \le - sL^{14}.
\end{multline*}

For (ii), we estimate the probability of choosing an edge in $S$ 
as $Q_S(t)/Q(t) < 2L^{13} sn^{-3/2}$ by Lemma \ref{opens}(iii). 
Then the failure probability $p_0$ satisfies
\[ p_0 \le \binom{n}{s} \binom{m}{L^{15} s} (2L^{13} sn^{-3/2})^{L^{15} s},\]
so  $s^{-1}\log p_0 \le O(\log n) + L^{15} \log \frac{2em}{L^2 n^{3/2}} \le - L^{15}$, giving $p_0=o(1)$.
\end{proof}

\medskip

Next we deduce a bound on the number of vertices 
of large degree in a given set.
For the following definition we emphasize that 
vertices in $S$ can belong to $D_d(S)$.

\begin{defn}
Let $D_d(S)$ be the set of vertices that have degree at least $d$ in $S$.  
\end{defn}
\begin{lemma} \label{large-degs}
Whp for any set $S$ of size $s$
\begin{enumerate}[(i)]
\item if $s \ge L^{12} \sqrt{n}$ and $d > 8L^2n^{-1/2}s $ 
then $|D_d(S)| < 8L^2n^{-1/2}s^2/d$,
\item if $s < L^{12} \sqrt{n}$ and $d > 4L^{15}$ 
then $|D_d(S)| < 4L^{15} s/d$.
\end{enumerate}
\end{lemma}

\begin{proof}
For (i), suppose on the contrary that there is $T \sub D_d(S)$ of size $8L^2n^{-1/2}s^2/d$. Then $S \cup T$ is a set of size at most $2s$ that spans at least $d|T|/2 > L^2n^{-1/2}(2s)^2$ edges, 
which contradicts Lemma \ref{density}(i). 
Similarly, for (ii), if there is $T \sub D_d(S)$ of size $4L^{15} s/d$ 
then $|S \cup T| \le 2s \le 2 L^{12} \sqrt{n}$ and 
$\eta_{S \cup T} \ge d|T|/2 > 2 L^{15} s \ge L^{15} |S \cup T|$, 
which contradicts Lemma \ref{density}(ii).
\end{proof}




We conclude this preliminary subsection with an estimate for a more
involved configuration required for the proof of Lemma \ref{lem:SB},
using the constants $ 0 < \alpha < \gamma < \beta $
declared in \eqref{eq:agb}. To motivate the following definition,
we remark that it will be applied with $H \sub N(x)$,
i.e.\ the neighbourhood of $x$ in $G(i_{max})$,
which will justify the assumed bounds on degrees
and open degrees into $H$ for $a \ne x$,
and also that $H$ contains no edges.
Furthermore, it will be applied at a step $i<i_{max}$
at which $H$ only contains vertices $y$ such that 
$xy$ is open and yet to be chosen as an edge,
so there will be no edges between $x$ and $H$.

\begin{defn} 
Let $H \sub V$ and $x \in V \sm H$. 
We say $(x,H)$ is {\em neighbourly} at step $i<i_{max}$
if $G(i)$ has no edges within $H \cup \{x\}$
and for any vertex $a \ne x$
at most $L^4$ edges $ab$ with $b \in H$ and 
at most $2x = 2 \hat{q}^2n$ open pairs $ab$ with $b \in H$.
We let $W_{xH}$ denote 
the number of ordered triples $(a,b,c)$ of vertices 
such that $ax$ is open, $\{b,c\} \sub H$ and $ab$, $ac$ are edges. 
\end{defn}

\begin{lemma} \label{WxH}
Whp for every neighbourly $(x,H)$ with $|H|=h$
where $L^{ \alpha } < h < L^{- \beta} \sqrt{n}$
we have $W_{xH} < 4 L^{- \alpha} h \hat{q}\sqrt{n}$. 
\end{lemma}

\begin{proof}
We will apply the critical interval method,
although we cannot do so directly for $W_{xH}$
as the boundedness hypothesis may fail due to
vertices with large open degree into $H$;
thus we will make some subtle alterations
to the structures that we count. 

We start with some definitions. We say that a vertex $a$ is 
\emph{obese} with respect to $H$ at time $t$ 
if at least $\hat{q}\sqrt{n} L^\gamma$ pairs $ab$ with 
$b \in H$ are open. (Our extravagant nomenclature here
is explained by reference to the definition of `heavy' below.) 
For any obese vertex $a$ 
we declare some subset of the open pairs $ ab$ with 
$b \in H$ \emph{inactive} so that the active open degree 
into $H$ is $ \lfloor \hat{q}\sqrt{n} L^\gamma \rfloor$. 

 We stress that the status of an open pair 
 as active or inactive can change back and forth in the course of the 
 process, but once a  pair is chosen as an edge its status as active or inactive remains the same for the rest of 
 the process.  
 
 For $j \in \{0,1,2\}$ let $W^j_{xH}$ denote the number of ordered triples 
 $(a,b,c)$ of vertices such that $ax$ is open, $\{b,c\} \sub H$, 
 the pairs $ab$ and $ac$ are both active, and their status depends on $j$:\ 
 if $j=0$ then both are open, if $j=2$ then both are edges, 
 and if $j=1$ then $ab$ is open and $ac$ is an edge. 
 Thus $W^2_{xH}$ has the same definition as $W_{xH}$, with the additional 
 condition that $ab$ and $ac$ are active at the steps 
 they are chosen as edges.

 First we show that there is a negligible difference between $W_{xH}$ 
 and $W^2_{xH}$, and so it will suffice to bound the latter. 
 Let $O$ be the set of vertices that are obese 
 with respect to $H$ at time $t$. We claim 
 that whp for any $H$ we have
\begin{equation}
\label{eq:O}
|O| < 2hL^{13-\gamma} = o(h).
\end{equation}
 To see this, suppose on the contrary
 there is $O' \sub O$ of size $2hL^{13-\gamma}$. 
 Then $|H \cup O'| < 2h$ and $Q_{H \cup O'} \ge L^{13} h \hat{q}\sqrt{n}$.   
 However, this contradicts Lemma \ref{opens}(iii), so (\ref{eq:O}) holds. 
 
 Applying Lemma \ref{opens}(iii) again, 
 we bound the number of open pairs in $H \cup O$ 
 by $Q_{H \cup O} < L^{13} \hat{q}\sqrt{n} \cdot 3h/2$. 
 Thus the probability at any given step that we choose an edge 
 between an obese vertex and $H$ is at most $2hL^{13}n^{-3/2}$.  
 For each set $H$ let $ \mc{O}_H $ be the event that the process chooses
 at least $ h L^{15} $ edges between $H$ and obese vertices (recalling that the set of obese vertices may change as the process evolves).  
 By the union bound, the probability that any $\mc{O}_H$ holds is at most
 \[  \sum_{ h = L^ \alpha}^{ n^{1/2} L^{-\beta} } \binom{n}{h} \binom{m}{h L^{15}} (2h L^{13} n^{-3/2} )^{hL^{15}} 
 \le \sum_{ h = L^ \alpha}^{ n^{1/2} L^{-\beta} } \binom{n}{h}  \left( \frac{O(1)}{L}  \right)^{hL^{15}} = o(1) . \]
 Thus we can assume that no event $ \mc{O}_H$ holds.
 Then the degree bound for neighbourly $(x,H)$ 
 implies $W_{xH} - W^2_{xH} < hL^{19}$, which is negligible 
 by comparison with the desired bound on $W_{xH}$.

For the remainder of the proof, 
we will show $W_j \le F_j := (1+t/L)w_j/2$ for $j=0,1,2$, where
\[  w_0  := L^{-\alpha-4} hx \hat{q}\sqrt{n},
\ w_1 := L^{-\alpha-2} hy \hat{q}\sqrt{n}
\ \text{ and } \   w_2 := 4 L^{-\alpha} h \hat{q}\sqrt{n}.\]
This will suffice to prove the lemma, as we will have 
$W_{xH} < W^2_{xH} + hL^{19} < F_2 + hL^{19} < w_2$.
Similarly to the proof of Lemma~\ref{opens}(iii), 
these are one-sided bounds rather than dynamic concentration statements,
but we can still use a modified form of the critical interval method.
For $W^j_{xH}$ we use the critical windows $[F_j-G_j,F_j]$, 
where $G_j = w_j/(40L^2)$. 
 
First we claim that our variables do not enter their critical windows 
for $n^{-1/4} \le t \le 1$ (assuming $ I > i_{max}$). 
For $j=0$ this follows from the trivial bound 
$W^0_{xH} \le nh^2 \ll w_0(1)$,
recalling that $\beta$ is large compared with $\alpha$.
For $j=1$ we can bound $W^1_{xH}$ by picking $\{b,c\} \sub H$ 
then a vertex counted by $ Y_{bc}(i)$, so by \eqref{eq:e*} we obtain
$W^1_{xH} \le O(y) h^2 \ll w_1$.
For $j=2$ we bound $W^2_{xH}$ by picking $\{b,c\} \sub H$ 
then a common neighbour, for which there are at most $O(L^4)$ choices 
by Definition~\ref{def:I0}(iii), so $W^2_{xH} = O( L^4 h^2 ) \ll w_2$.
Thus the claim holds.

Next we will prove the trend hypothesis, 
i.e.\ that $\mc{Z}W^j_{xH} = W^j_{xH} - F_j$ is a supermartingale 
while $W^j_{xH}$ is in its critical window. 
Below we will analyse the contributions to 
$\mb{E}[ \DD_i \mc{Z}W^j_{xH} \mid \mc{F}_i]$ 
separately according to each of the pairs $ax$, $ab$, $ac$.
When we calculate the expected change due to closing 
of $ab$ or $ac$ we will ignore correction terms 
due to changes that do not actually occur 
when $a$ is obese and these closures simply change the status 
of some other open pair from inactive to active. 
To justify this, we first give upper bounds
on these correction terms, which we will later
see are negligible compared with the main terms.


For $ a \in O$ let $A_a$ denote the set of $b \in H$ 
such that $ab$ is open and active. By \eqref{eq:O},
the contribution to $\mb{E}[ \DD_i W_{xH}^0 ] $ 
due to closing a pair $ ab $ or $ac$ where $a$ is obese
is at most
\begin{equation}
\label{eq:obese1}
2Q^{-1} \sum_{a \in O} \sum_{b \in A_a} (Y_{ab} + Y_{ba}) |A_a|  
 \le 5y q^{-1} \cdot 2 h L^{13-\gamma} \cdot ( \hat{q} n^{1/2} L^\gamma) ^2
\ll 8t n^{-3/2} F_0 L^{-2}. 
\end{equation}
Similarly, the contributions to
$\mb{E}[ \DD_i W_{xH}^1 ] $ due to closing a pair $ac$ 
where $a$ is obese is at most
\begin{equation}
\label{eq:obese2} 
2Q^{-1} \sum_{ a \in O } \sum_{b \in A_a} (Y_{ab} + Y_{ba}) L^4
\le 5 y q^{-1} L^4 \cdot 2 h L^{13-\gamma} \cdot \hat{q} n^{1/2} L^\gamma
\ll  8t n^{-3/2} F_1 L^{-2}.
\end{equation}

In the calculation of the expected change in 
 $\mc{Z}W^j_{xH} = W^j_{xH} - F_j$ we
write $$ \DD_i( F_j) = (1+o(1))F'_jn^{-3/2} \ \ \ \text{ and } \ \ \
F'_j \ge ((L+t)^{-1} - (3-j)8 t) F_j.$$
For each open pair $\aA\bB$ we have a destruction term of 
$$ 2Q^{-1} \sum_{f \in W^j_{xH}} (Y_{f(\aA\bB)} + Y_{f(\bB\aA)} + 1) 
\ge (1+o(1))8tn^{-3/2}(F_j-G_j),$$ 
when $ W^j_{xH}$ is in the critical interval.
This gives self-correction against 
a corresponding $8tn^{-3/2}F_j$ term in $\DD_i(F_j)$. 
For each edge we have a creation term of 
$$2Q^{-1} W^{j-1}_{xH} \le (1+o(1))2q^{-1}F_{j-1},$$ 
where $2q^{-1} F_0 =L^{-2} t^{-1} n^{-3/2} F_1$ 
and $2q^{-1} F_1 = t L^{-2}n^{-3/2} F_2$.

Next we account for fidelity corrections.
As there are no edges within $H \cup \{x\}$
there is no creation fidelity term
(it is not possible to add an edge 
and simultaneously close an open pair 
in the configuration).
For destruction fidelity,
we first consider configurations for $j=0,1$
in which selecting an edge $az$ simultaneously
closes the open pairs $ab$ and $ax$.
There are at most $h$ choices for $c$,
then $2v$ choices for $a$
where $v=x$ for $j=0$ or $v=y$ for $j=1$,
then $L^4$ choices for $z$
in the common neighbourhood of $b$ and $x$,
then $2y$ choices for $b \in Y_{az}$.
This gives a correction term
$O(q^{-1} h v L^4 y) \ll 8tn^{-3/2} F_j L^{-2}$.
For $j=0$ we also need to consider
configurations in which selecting $az$
simultaneously closes $ab$ and $ac$.
There are at most $h$ choices of $b$,
then $2y$ choices of $z$ in $Y_{xb}$,
then $2x$ choices of $a$ in $X_{ab}$,
then $L^4$ choices of a neighbour $c$
of $z$ in $H$ (as $(x,H)$ is neighbourly).
This gives a correction term
$O(hyx L^4) \ll 8tn^{-3/2} F_0 L^{-2}$.
Using $\tfrac{F_j}{8t(L+t)} \ge 4G_j$, 
we obtain
\begin{align*}
\mb{E}[ \DD_i \mc{Z}W^0_{xH} \mid \mc{F}_i]
& \le - (1+o(1))8tn^{-3/2} \cdot ( 3(F_0-G_0) + (\tfrac{1}{8t(L+t)} - 3) F_0 ) \le 0. \\
\mb{E}[ \DD_i \mc{Z}W^1_{xH} \mid \mc{F}_i]
& \le - (1+o(1))8tn^{-3/2} \cdot ( 2(F_1-G_1) - \tfrac{1}{8 L^2 t^2} F_1 + (\tfrac{1}{8t(L+ t)} - 2) F_1 ) \le 0. \\
\mb{E}[ \DD_i \mc{Z}W^2_{xH} \mid \mc{F}_i]
& \le - (1+o(1))8tn^{-3/2} \cdot ( F_2-G_2 - \tfrac{1}{8L^2} F_2 + (\tfrac{1}{8t(L+t)} - 1) F_2 ) \le 0.
\end{align*} 
Note that the correction terms 
(\ref{eq:obese1}) and (\ref{eq:obese2})
for inactive edges and the fidelity terms
are indeed negligible in this calculation, 
so the trend hypothesis holds.

It remains to establish the boundedness hypothesis. 
Note that since we can restrict our attention to $t \ge 1$,
the functions $G_j$ are approximately non-increasing. 
As we are proving one-sided bounds 
with a union bound over the choice of $x$ and $H$, 
it suffices to establish the boundedness hypothesis as set forth in (\ref{eq:boundonesided}) with $h$ playing the role of $s$.  
We add an additional wrinkle here.  
Recall that Freedman's inequality (Lemma~\ref{Freedman}) 
only requires a bound on the 
{\em positive} change in the random variable in question.  
For each pair $e$ in the collection $ ax, ab, ac$ 
let $N_e^+$ bound the {\em positive} one-step change in 
$\mc{Z} W^j_{xH}$ due to the change in the status of $e$ 
and let $\Var_e$ denote the one-step variance of $ \mc{Z} W^j_{xH}$ 
that can be attributed to the change in status of $e$. 
To apply Freedman's inequality, since $G_j = w_j/(40L^2)$, 
it suffices to show 
\begin{equation}
\label{eq:suff1}
N_e^+ \le w_j/(hL^5) \ \ \ \text{ and } \ \ \ \Var_e \le w_j^2/(hL^8n^{3/2}).
\end{equation}
In some cases we will show the stronger statement 
\begin{equation}
\label{eq:suff2}
N_e < w_j/ (hL^{10}),
\end{equation}
where $ N_e$ is the absolute value 
of the one-step change in $ \mc{Z} W^j_{xH}$.
Note that (\ref{eq:suff2}) clearly implies (\ref{eq:suff1}):\ 
the bound on $N_e^+$ is immediate and the bound for $\Var_e$ 
follows by Lemma \ref{lem:vard}.

First we note that the required bounds for creation
are straightforward. Indeed, for $ W_{xH}^1$
the bound on active open degrees gives
$N_e \le \hat{q} \sqrt{n} L^\gamma \ll w_1/(hL^{10})$,
and for $ W_{xH}^2$ the assumption that $(x,H)$ is neighbourly
gives $ N_e \le L^4 \ll w_2/( hL^{10}) $.

For destruction we obtain negative changes in $ \mc{Z} W_{xH}^j$, so we only need to bound $\Var_e$. First we introduce some additional definitions. 
We say that a vertex $a$ is \emph{heavy} 
with respect to $H$ at time $t$ 
if at least $\hat{q}\sqrt{n} L^{-\gamma}$ pairs 
$ab$ with $ b \in H$ are open.  
Let $T = T_{xH}$ be the set of heavy vertices $a$ 
such that $xa$ is open.
As $|T| \hat{q}\sqrt{n} L^{-\gamma}
\le \sum_{u \in H} X_{ux} < 2xh$, we have
$$ |T| < 2hx/( \hat{q}\sqrt{n} L^{-\gamma}) =  2hL^\gamma\hat{q}\sqrt{n}. $$
Let $U$ be the set of vertices $z$ such that $ zx$ is open and $z$ has at least $\hat{q}\sqrt{n} L^{-3 \gamma}$ neighbours in $T$.
By Lemma~\ref{large-degs} we have
$$ |U| < \begin{cases} 8hL^{4\gamma+15} 
& \text{ if } |T| < L^{12} \sqrt{n}, \\
32h^2 L^{5\gamma+2} \hat{q} & \text{ otherwise}. \end{cases}$$
Here we used $\hat{q}\sqrt{n} L^{-3\gamma} > 4L^{15}$ and $\hat{q}\sqrt{n} L^{-3\gamma} > 8L^2n^{-1/2} \cdot 2hL^\gamma\hat{q}\sqrt{n}$, 
which follows from our choice 
of $ \beta $ to be large relative to $ \gamma $,
 to get the lower bounds on $d$ required for Lemma~\ref{large-degs}.

Now consider destruction for the variables $ W_{xH}^j$ for $ j=0,1$. 
We write $\DD_i W_{xH}^j = \DD_i V_1 + \DD_i V_2$, 
where $\DD_i V_1$ accounts for the change in 
$V = W_{xH}^j$ that comes from the choice of an edge $xz$ where $z \in U$, 
and $\DD_i V_2$ accounts for the rest. 
For $\DD_i V_2$ we will obtain the required bound on $\Var_e$ 
by establishing the bound (\ref{eq:suff2}) on $N_e$. 
The contribution to $N_e$ from closing $ab$ or $ ac$ is bounded by 
$ 2y \hat{q} n^{1/2} L^\gamma < w_0/(hL^{10})$ for $ W^0_{xH}$ 
(using the bound on active open degrees)
and by $ 2y L^4 < w_1/(hL^{10})$ for $ W^1_{xH}$
(as $(x,H)$ is neighbourly).
Next we consider the contribution 
from closing $ xa$ where $ a$ is not heavy. 
For $j=0$ this is at most 
$(2y) (\hat{q}\sqrt{n}L^{-\gamma})^2 
< 2 \hat{q} \sqrt{n} L^{1-2\gamma} x
< L^{-\alpha-12} x \hat{q}\sqrt{n}
= w_0/(hL^{10})$,
as $ \gamma $ is large relative to $ \alpha$.
For $j=1$, as $(x,H)$ is neighbourly, 
the contribution is at most 
$(2y)(\hat{q}\sqrt{n} L^{- \gamma}) L^4
< L^{-\alpha-12} y \hat{q}\sqrt{n} = 
w_1/(hL^{10})$, 
again as $ \gamma $ is large relative to $ \alpha$.
Now we consider the contribution from closing of pairs $ xa$ 
where $a$ is heavy. Note that we do not select $xz$ with $z \in U$,
as this case will be analysed in $\DD_i V_1$,
so this contribution is at most 
$\hat{q}\sqrt{n} L^{-3\gamma} (\hat{q}\sqrt{n} L^\gamma)^2 
= L^{-\gamma} x \hat{q}\sqrt{n} \ll w_0/(hL^{10})$ for $j=0$
(by the bound on active open degrees), 
or $\hat{q}\sqrt{n} L^{-3\gamma} (\hat{q}\sqrt{n} L^\gamma) L^4 
\ll w_1/(hL^{10})$ for $j=1$
(as $(x,H)$ is neighbourly and $t \ge 1$). 
Thus we have the required bound on $N_e$ for $\DD_i V_2$.

For $j=0,1$ it remains to bound $\Var_e$ for $\DD_i V_1$. 
The probability that an edge $xz$ with $ z \in U$ is chosen 
is at most $2|U|/q$, and the resulting change in 
$W_{xH}^j$ is at most $(2y) (\hat{q}\sqrt{n} L^\gamma)^2 $ for $ j=0$, 
or $(2y) (\hat{q}\sqrt{n} L^\gamma) L^4 $ for $ j=1 $. 
Suppose first that $ |T| < L^{12} \sqrt{n} $,
so that $|U| <  8hL^{4\gamma+15}$.
Then for $j=0$ we have $\Var_e \le 16hL^{4\gamma+15} q^{-1} (2y)^2 (\hat{q}\sqrt{n} L^\gamma)^4 = \wt{O}( h\hat{q}^5 n)$, 
which suffices to establish (\ref{eq:suff1})
as $w_0^2/(hL^8n^{3/2}) = \wt{ \Omega}(h \hat{q}^6 n^{3/2})$. 
Also, for $j=1$ we have $\Var_e \le 16hL^{4\gamma+15} q^{-1} (2y)^2 (\hat{q}\sqrt{n} L^\gamma)^2 L^8 = \wt{O}( h \hat{q}^3)$, 
which suffices as $w_1^2/(hL^8n^{3/2})
= \wt{\Omega}(h \hat{q}^4 n^{1/2})$, recalling that $t \ge 1$. 
Now suppose $ |T| \ge L^{12} \sqrt{n}$,
so that $ |U| < 32h^2 L^{5\gamma+2} \hat{q}$.
Then for $j=0$ we have 
$\Var_e \le 64h^2 L^{5\gamma+2} n^{-2} (2y)^2 (\hat{q}\sqrt{n} L^\gamma)^4 
< 256h^2 L^{9\gamma+4} \hat{q}^6 n$, 
and for $j=1$ we have $\Var_e \le 64h^2 L^{5\gamma+2} n^{-2} 
 (2y)^2 (\hat{q}\sqrt{n} L^\gamma)^2 L^8 < 256h^2 L^{7\gamma+12} \hat{q}^4$.
  As $h < L^{-\beta} \sqrt{n}$ and
  $ \beta $ is large relative to $ \alpha, \gamma$
these bounds suffice to establish (\ref{eq:suff1}).

It remains to bound $\Var_e$ for destruction of $W_{xH}^2$.
Let $W$ be the set of vertices that are open to $x$ 
and have at least two neighbours in $H$. 
Then $|W| \le \sum_{a \in H} Y_{xa} < 2yh$.
Let $U'$ be the set of vertices that are open to $x$ 
and have at least $yL^{-\gamma}$ neighbours in $W$. 
By Lemma~\ref{large-degs} we have
$$ |U'| < \begin{cases} 8h L^{\gamma+15}
 & \text{ if } |W| < L^{12} \sqrt{n} \\
32h^2 y n^{-1/2} L^{\gamma+2} & \text{ otherwise}. \end{cases}$$
Here we used $yL^{-\gamma} > 4L^{15}$ 
and $yL^{-\gamma} > 8L^2n^{-1/2} \cdot 2yh$
(as $ \beta$ is large relative to $ \gamma$) 
to get the lower bound on $d$ required for Lemma~\ref{large-degs}.
We write the destruction of $W_{xH}^2$ at step $i$ as
$ \DD_i V_1 + \DD_i V_2$, 
where $\DD_i V_1$ accounts for the change in $W_{xH}^2$ 
that comes from the choice of an edge $xz$ where $z \in U'$,
and $\DD_i V_2$ accounts for the rest.

For $\DD_i V_2$ we can obtain the required bound on $\Var_e$  
from the bound (\ref{eq:suff2}) on $N_e$; 
indeed, by definition of $U'$
and as $(x,H)$ is neighbourly,
$N_e < yL^{-\gamma} \cdot L^4 < w_2/( hL^{10}) $.
For $\DD_i V_1$, suppose first that $|W| < L^{12} \sqrt{n}$,
so that $|U'| < 8h L^{\gamma+15}$. 
We choose an edge $xz$ with $ z \in U'$ with probability 
at most $2|U|/q$, and as $(x,H)$ is neighbourly
the resulting change in 
$W_{xH}^2$ is at most $2y \cdot L^4$,
so $\Var_e < 8h L^{\gamma+15} q^{-1} (2y)^2 L^{16} 
= \wt{O} ( h y n^{-3/2})$, which suffices as 
$w_2^2/(hL^8n^{3/2}) = \wt{\Omega} (hy^2n^{-3/2})$. 
On the other hand, if $ |W| \ge L^{12} \sqrt{n}$
then $\Var_e < 32h^2 y n^{-1/2} L^{\gamma+2} q^{-1} (2y)^2 L^{16} 
< 128 h^2n^{-1/2} L^{\gamma+19}y^2 n^{-3/2}$, 
which also suffices  to establish (\ref{eq:suff1}) 
as $ \beta $ is large relative to $ \alpha, \gamma$.
\end{proof}

\subsection{Proof of Theorem~\ref{indep}.}
\label{sec:pfindep}

We will show whp 
$$\aA(G) < k := (1+3\eps) \sqrt{2n\log n}.$$
As $\aA(G) \le \aA(G(i_{max}))$, 
it suffices to bound $\aA(G(i_{max}))$.
We need to estimate the probability that 
there is an independent set $K$ of size $k$.
As discussed above, we will take a union bound 
over all such sets $K$ together with certain information 
about how neighbourhoods in $ G(i_{max}) $ intersect $ K$.

Let $K$ be a potential independent set of size $k$.  
We define a sequence of vertices $x_1,\dots,x_z$, 
where each $x_\ell$ is chosen 
to maximise the number of neighbours in $K$ 
that are not also neighbours of some $ x_j $ for $ j < \ell$. 
More precisely, the $\ell$th {\em hole} is 
$H_\ell = (N(x_\ell) \sm \cup_{\ell'<\ell} N(x_{\ell'})) \cap K$, 
where $x_\ell$ is chosen to maximise $h_\ell = |H_\ell|$,
and we recall our convention that all neighbourhoods 
are defined with respect to $ G(i_{max}) $.
We stop the sequence if there are no vertices that give more than 
$L^{2\alpha}$ new neighbours in $K$. Note that $x_\ell \notin K$ for 
$\ell \in [z]$, as $K$ is independent. We say that a hole is {\em large} 
if it has size more than $L^{-\beta} \sqrt{n}$. 
We let $Z_A$ be the set of $\ell$ such that $H_\ell$ is large, 
$$Z_B = [z] \sm Z_A, \ \ \  A = \cup_{\ell \in Z_A} H_\ell, 
\ \ \  B = \cup_{\ell \in Z_B} H_\ell, \ \ \ C = K \sm (A \cup B).$$ 
For $\ell \in Z_B$ we specify the steps of the process 
at which the edges between $ x_\ell $ and $ H_\ell$ appear.  We write 
$H_\ell = \{ v_{\ell j} : j \in [h_\ell] \}$, 
where $x_\ell v_{\ell j}$ is selected at step $i_{\ell j}$, 
and $i_{\ell j}$ is increasing in $j$. 
For $\ell \in Z_A$ we specify the entire neighbourhood of $x_\ell$ in $G(i_{max})$:\ we write $d_\ell = |N(x_\ell)|$ and 
$N(x_\ell) = \{ v_{\ell j} : j \in [d_\ell] \}$, where $x_\ell v_{\ell j}$ is selected at step $i_{\ell j}$, and $i_{\ell j}$ is increasing in $j$.  
We will estimate $\mb{P}(\mc{E})$, where $\mc{E}$ is the event that there is an independent set $K$ with some fixed choices 
of $z$; $x_\ell$ and $h_\ell$ for $\ell \in [z]$; and $d_\ell$ for $\ell \in Z_A$.  We will refer to these choices of hole sizes, vertices 
with large neighbourhoods in $K$ and vertex degrees 
as the {\em initial data} that defines $ \mc{E} $.  
Note that by Lemma~\ref{large-degs}(ii) we can assume
\begin{equation} \label{eq:ZAz}
|Z_A| < 8L^{16+\beta} \quad  \text{ and } \quad z < 4L^{15-2\alpha} k. 
\end{equation}
For $\ell \in Z_A$, $j \in [d_\ell]$ we claim that
\begin{equation}
\label{eq:recalldegrees}
i_{\ell j} = jn/2 \pm n^{3/2 - \eps/3} 
\quad \text{ and } \quad d_\ell = d \pm n^{1/2 - \eps/3},
\end{equation}
where we recall $d = 2t_{max}\sqrt{n} = 2m/n 
= \sqrt{(1/2 - \eps)n\log n}$.
To see \eqref{eq:recalldegrees}, note that if e.g.\
we had $i = i_{\ell j} < jn/2 - n^{3/2 - \eps/3} $
then we would have $Y_{x_\ell}(i) \ge j
> 2n^{-1}(i + n^{3/2 - \eps/3})
= y_1(t) + 2n^{1/2 - \eps/3}$,
which contradicts the degree bounds
$Y_u(i) = (1 \pm \dD_{Y_1}(t)) y_1(t) $ 
in the event $\mc{G}_i$ (see Definition~\ref{def:I0}).

Now, in addition to the initial data, 
we fix the independent set $K$,
the specific edges  $x_\ell v_{\ell j}$ 
and appearance times $i_{\ell j} $ 
for $ \ell \in Z_A, j \in [d_\ell] $, 
and likewise for $ \ell \in Z_B, j \in [h_\ell] $.  
We let $ \mc{E}_K$ be the event that $K$ is independent 
and all the specified edges appear at the specified steps of the process.
Thus $ \mc{E} $ is a union of events of the form $\mc{E}_K$.  

To estimate the probability of any given event $ \mc{E}_K $,
for each step $i$ we need to estimate the probability 
that the selected edge is compatible with $\mc{E}_K$, 
conditional on the history of the process. 
We say $i$ is a \emph{selection step} if $i$ is one of $i_{\ell j}$ 
for $\ell \in Z_A$, $j \in [d_\ell]$ or $\ell \in Z_B$, $j \in [h_\ell]$; then the selected edge is specified by $\mc{E}_K$, so the required 
probability is simply $2/Q = (1 \pm 2\dD_Q)2q^{-1}$. For other $i$, the required probability is $1 - N_i/Q$, where $N_i$ is the number 
of ordered open pairs that cannot be selected at step $i$ when $\mc{E}_K$ occurs. If $i=i_{\ell j}$ is a selection step write $N_i=0$. Then we estimate
\begin{equation}
\label{eq:EK}
\mb{P}(\mc{E}_K) \le \prod_{\ell \in Z_A} \prod_{j=1}^{d_\ell} (1 \pm 2\dD_Q)2q(t_{\ell j})^{-1}
\cdot \prod_{\ell \in Z_B} \prod_{j=1}^{h_\ell} (1 \pm 2\dD_Q)2q(t_{\ell j})^{-1}
\cdot \prod_{i=1}^m (1 - N_i/Q).
\end{equation}
To estimate $N_i$, we classify open pairs 
that cannot be selected at step $i$ as follows.
\begin{itemize}
\item Let $N_{iAi}$ be the number of ordered open pairs 
of the form $v_{\ell j} v_{\ell j'}$ 
for some $\ell \in Z_A$, $j,j' \in [d_\ell]$.
\item Let $N_{iAo}$ be the number of ordered open pairs 
of the form $ x_\ell y $ or $ y x_{\ell} $
where $ \ell \in Z_A$ and 
$ y \not\in N(x_\ell) \cup K \cup \{ x_1, \dots, x_z\}$.
\item Let $N_{iBi}$ be the number of ordered open pairs $ab$ such that $ B \cap ab \neq \emptyset $ and  selecting $e_i = ab$ would close an open pair of the form $x_\ell v_{\ell j}$ for $\ell \in Z_B$, $j \in [h_\ell]$.
\item Let $N_{iBo}$ be the number of ordered open pairs $ab$ such that $ B \cap ab = \emptyset $ and selecting $e_i = ab$ would close an open pair of the form $x_\ell v_{\ell j}$ for $\ell \in Z_B$, $j \in [h_\ell]$.
\item Let $N_{iK}$ be the number of ordered open pairs in $K$ that are not contained within any hole.
\end{itemize}
We refer to pairs counted by $N_{iAo}$ or $N_{iBo}$ as {\em outer}
and those counted by $N_{iAi}$ or $N_{iBi}$ as {\em inner}
(which is indicated by one of the $i's$ in the notation; the other
refers to the step $i$, which we hope will not cause confusion).
For $\ell \in Z_A$ we stress that by naming the $v_{\ell j}$'s 
we have specified all neighbours of $x_\ell$ (not only those in $K$), 
so we cannot select a pair $yx_\ell$ with $y \notin N(x_\ell)$;
we also exclude $y \in K \cup \{ x_1, \dots, x_z\}$ 
in the definition of $N_{iAo}$ to facilitate the estimate
for overcounting in Lemma \ref{lem:SO}.
For $N_{iK}$ we note that all open pairs within $K$ are forbidden
(as $K$ is independent) but again to avoid overcounting
we only include those not contained within any hole.
We write $$N_i \ge N_{iAi} + N_{iAo} + N_{iB} + N_{iK} - N_{iO},$$ 
where  $N_{iB} = N_{iBi} + N_{iBo}$ and 
$N_{iO}$ corrects for any open pairs that appear in more 
than one of the above collections. 
(We will see that the most significant source
of overcounting comes from pairs counted 
by both $ N_{iK}$ and $ N_{iBi}$.)  
We substitute 
\begin{equation}
\label{eq:splitter}
1 - N_i/Q \le \exp \left\{  - (1-2\dD_Q) q^{-1}(N_{iAi}+ N_{iAo}+ N_{iB}+N_{iK}-N_{iO}) \right\}
\end{equation}
in (\ref{eq:EK}), recalling that $\dD_Q = O(n^{-\eps/5})$, to obtain
\begin{align}
\label{eq:all1}
- \log \mb{P}(\mc{E}_K) 
& \ge S_{Ai} - T_A + S_B - T_B + S_{Ao} + S_K - S_O  \\
 & \ \ \ +  \log \frac{n^2}{2} \brac{\sum_{\ell \in Z_A} d_\ell + |B|} 
  - O( n^{1/2- \eps/5}), \text{ where } \nonumber
\end{align}
\begin{align*}
S_\mu &= \sum_{i=1}^m N_{i\mu}q^{-1}  
\ \ \ \ \ \text{ for } \mu \in \{Ai,Ao,B,K,O\}, \\
T_A &= \sum_{\ell \in Z_A} \sum_{j=1}^{d_\ell} 4t_{\ell j}^2 
\ \ \ \ \ \text { and }  \ \ \ \ \
T_B = \sum_{\ell \in Z_B} \sum_{j=1}^{h_\ell} 4t_{\ell j}^2.
\end{align*}
To estimate the terms in \eqref{eq:all1},
we start by showing in the next two lemmas
that $S_{Ai} - T_A$ and $S_B - T_B$ are negligible.
(The remaining terms will be used to balance 
the number of events in our union bound calculation.)

\begin{lemma}
\label{lem:SAi}
$ T_A - S_{Ai} < O( n^{1/2 - \eps/5}) $.
\end{lemma}

\begin{proof}
We start by giving a lower bound on $N_{iAi}$
for any $i$ that is not a selection step.
For $\ell \in Z_A$ let $j_\ell=j_\ell(i)$ 
be the value of $j \in [d_\ell]$ such that 
$i_{\ell (j-1)} \le i < i_{\ell j}$, 
where $i_{\ell 0 } := 0$,
i.e.\ $j_\ell - 1$ edges have been selected at $x_\ell$. 
Let $S_\ell = \{ v_{\ell j} \}_{j=j_\ell+1}^{d_\ell}$ 
and $s_\ell = |S_\ell| = d_\ell + 1 - j_\ell$;
thus $\{ x_\ell v : v \in S_\ell \}$ is the set of open
pairs at $x_\ell$ that will later be selected as edges.
As we consider the whole neighbourhood of $x_\ell$
(not just the neighbourhood in $K$),
the number of ordered open pairs $v_{\ell j} v_{\ell j'}$ 
with $j > j_\ell$, $j' \le j_\ell$ is 
$\sum_{v \in S_\ell} 2Y_{vx_\ell} = (1 \pm \dD_Y) 2ys_\ell$. 
 
We also note that any vertex has at most $L^4$ neighbours 
in $S_\ell$ by the codegree bound in $ G(i_{max})$,
which is valid as we assume $ I < i_{max} $.
Then by Lemma \ref{opens}(i) whp 
$Q_{S_\ell} = (1 \pm n^{-\eps/5}) \hat{q}s_\ell^2$ 
if $s_\ell > n^{1/4}$ and $\hat{q}s_\ell \ge n^{2\eps/5} L^{14}$. 
Since $\hat{q} \ge n^{-1/2+\eps}$ 
this holds for $s_\ell > n^{1/2 - \eps/2}$, 
so we can write
 $Q_{S_\ell} \ge (1 - n^{-\eps/5}) \hat{q} 
 s_\ell(s_\ell - n^{1/2 - \eps/2})$,
as this bound is trivial for  $s_\ell \le n^{1/2 - \eps/2}$.
  The bound on codegrees also implies that 
 the number of open pairs that can be counted by 
 more than one $\ell \in Z_A$ is at most 
 $(|Z_A|L^4)^2 = \wt{O}(1)$ by \eqref{eq:ZAz}, 
 which is negligible. Thus
\begin{align} \label{eq:NiAi}
 N_{iAi} & \ge (1 - n^{-\eps/5}) \sum_{\ell \in Z_A} 
\brac{ 2ys_\ell + \hat{q} s_\ell (s_\ell - n^{1/2 - \eps/2}) }
- O( \hat{q} n^{1 - \eps/5} ) \nonumber \\
& = \sum_{\ell \in Z_A} 
 \brac{ 2ys_\ell + \hat{q} s_\ell^2 } 
 - O( \hat{q} n^{1 - \eps/5} ). 
\end{align}
To estimate $S_{Ai} =  \sum_{i=1}^m N_{iAi}q^{-1}$,
it is convenient to use the bound \eqref{eq:NiAi} 
for all $i$, even selection steps (where $N_i=0$);
this is valid as the resulting correction
is $ \wt{O} (n^{-1/2})$, which is negligible.
We write $S_{Ai}=S_{Ai1}+S_{Ai2} 
+ \wt{O}( n^{1/2 - \eps/5} )$ 
according to the contributions 
of the first and second terms in  \eqref{eq:NiAi}. Then
\begin{align*}
S_{Ai1} & = \sum_{i=1}^m \sum_{\ell \in Z_A} 2ys_\ell q^{-1}
 = \sum_{\ell \in Z_A} \sum_{j=1}^{d_\ell} 
\sum_{i=i_{\ell (j-1)}}^{i_{\ell j}-1} 4tn^{-3/2} (d_\ell+1 - j) 
 = \sum_{\ell \in Z_A} \sum_{j=1}^{d_\ell} 
\sum_{i=1}^{i_{\ell j}} 4in^{-3} \\
& = \sum_{\ell \in Z_A} \sum_{j=1}^{d_\ell} 2t_{\ell j}^2 
- \sum_{\ell \in Z_A} \sum_{j=1}^{ d_\ell } 2 t_{\ell j} n^{-3/2} 
 = \frac{T_A}{2} - \wt{O}( n^{-1} ).
\end{align*}
Recalling (\ref{eq:recalldegrees}), we note that
\begin{align} \label{eq:TA/2}
\frac{T_A}{2} = \sum_{ \ell \in Z_A} \sum_{j=1}^{d_\ell} 2 t_{\ell j}^2 
& < |Z_A| \sum_{j=1}^{d + n^{1/2-\eps/3}} 
 2 \brac{ j n^{-1/2} /2 + n^{- \eps/3}}^2  \nonumber \\
& < |Z_A| \sum_{j=1}^d j^2 (2n)^{-1} 
+ \wt{O}( n^{1/2 - \eps/3}).
\end{align}
We also have
\[ S_{Ai2} = \sum_{i=1}^m  \sum_{\ell \in Z_A} \hat{q} s_\ell^2 q^{-1}
= \sum_{\ell \in Z_A} \sum_{j=1}^{d_\ell} 
\sum_{i=i_{\ell (j-1)}}^{i_{\ell j}} n^{-2}(d_\ell-j)^2, \]
which is minimized when each $ d_\ell$ is as small as possible, 
and then each $ i_{\ell j}$ occurs as early as possible,
so $S_{Ai2}  \ge |Z_A| \sum_{j=1}^d (2n)^{-1} j^2 
- \wt{O}( n^{1/2-\eps/3}) \ge T_A/2 - \wt{O}( n^{1/2-\eps/3})$
by \eqref{eq:TA/2}. The lemma follows.
\end{proof}

\begin{lemma} \label{lem:SB}
$T_B - S_B \le O( L^{-2} n^{1/2} )$. 
\end{lemma}

\begin{proof}
Similarly to the proof of Lemma \ref{lem:SAi},
we start by giving a lower bound on $N_{iB}$
for any $i$ that is not a selection step.
For $ \ell \in Z_B$ let $S_\ell = S_\ell(i) $ 
be the set of $v_{\ell j}$ with $j \in [h_\ell]$ 
such that $x_\ell v_{\ell j}$ is still open.
We write $s_\ell = |S_\ell|$. 
Each $v_{\ell j}$ in $S_\ell$ contributes 
$2Y_{ v_{\ell j} x_\ell} = (1 \pm \dD_Y) 2y$ to $N_{iBi}$ 
and $2Y_{ x_\ell v_{\ell j}} = (1 \pm \dD_Y) 2y$ to $N_{iBo}$; 
however, we need to account for open pairs 
that may be counted by more than one pair $x_\ell v_{\ell j}$. 

We claim that there is no overcounting for inner pairs.
To see this, note that if $v_{\ell j}v_{\ell' j'}$ is counted 
for $x_\ell v_{\ell j}$ and for $x_{\ell'} v_{\ell' j'}$ then 
$x_\ell v_{\ell' j'}$ and $x_{\ell'} v_{\ell j}$ are both edges,
but this cannot occur by the hole construction procedure.
Furthermore, there is no overcounting between $N_{iBi}$ and $N_{iBo}$,
as inner pairs intersect $K$ but outer pairs do not
(as $K$ is independent). 

Thus the claim holds,
and it remains to consider overcounting for outer pairs.
This may occur for $x_\ell v_{\ell j}$ 
and $x_\ell v_{\ell j'}$ with $\ell \in Z_B$ and $j,j' \in S_\ell$. 
The number of such overcounted pairs is at most  
$W_{x_\ell S_\ell}$, which we will estimate by Lemma~\ref{WxH}. 
To see that this lemma applies, we note that
$s_\ell \le h_\ell < L^{-\beta} \sqrt{n}$
as holes $H_\ell$ with $\ell \in Z_B$ are not large.
We also note that $(x_\ell,S_\ell)$ is neighbourly,
as $S_\ell \sub N(x_\ell)$ and 
all pairs $x_\ell y$ with $y \in S_\ell$ are open,
so $G(i)$ has no edges within $H_\ell \cup \{x_\ell\}$
and for any vertex $a \ne x_\ell$
at most $L^4$ edges $ab$ with $b \in H_\ell$ and 
at most $2x$ open pairs $ab$ with $b \in H_\ell$.
If $s_\ell \ge L^{\alpha}$ then  Lemma~\ref{WxH} gives
$W_{x_\ell S_\ell} < L^{-\alpha} s_\ell \hat{q}\sqrt{n}$.
Summing over $\ell \in Z_B$, 
using $|Z_B| \le z \le 4L^{15-2\alpha}k$ from \eqref{eq:ZAz}
and $\sum_{\ell \in Z_B} s_\ell \le k$ we obtain
\[N_{iBo} \ge (1-\dD_Y)2y \sum_{\ell \in Z_B} (s_\ell - L^{\alpha}) - \sum_{\ell \in Z_B} L^{-\alpha} s_\ell \hat{q}\sqrt{n}
\ge  2y \sum_{\ell \in Z_B} s_\ell - L^{17-\alpha} k \hat{q}\sqrt{n}.\]
Including $N_{iBi}$, we deduce
\begin{equation} \label{eq:NiB}
N_{iB} \ge (1 - \dD_Y) 4y \sum_{\ell \in Z_B} s_\ell  - L^{17-\alpha} k \hat{q}\sqrt{n}
= 4y \sum_{\ell \in Z_B} s_\ell - O( L^{-3} \hat{q} n ), 
\end{equation}
as $ \alpha $ is large.
As $S_B =  \sum_{i=1}^m N_{iB}q^{-1}$, we have
\begin{align*}
S_B + O( L^{-2} n^{1/2} )  &=
\sum_{i=1}^m \sum_{\ell \in Z_B} 4ys_\ell q^{-1}
 \ge \sum_{\ell \in Z_B} \sum_{j=1}^{h_\ell} 
 \sum_{i=i_{\ell (j-1)}}^{i_{\ell j}} 8tn^{-3/2} s_\ell \\
& = \sum_{\ell \in Z_B} \sum_{j=1}^{h_\ell} \sum_{i=1}^{i_{\ell j}} 8in^{-3} = T_B - \wt{O}( n^{-1}).  
\end{align*}
Similarly to Lemma \ref{lem:SAi}, there is a negligible
correction due to using the bound \eqref{eq:NiB} 
at selection steps. The lemma follows.
\end{proof}

Lemmas \ref{lem:SAi} and \ref{lem:SB} reduce (\ref{eq:all1}) to 
\begin{equation}
\label{eq:all1'}
- \log \mb{P}(\mc{E}_K) \ge S_{Ao} + S_K - S_O +  \log \frac{n^2}{2} \brac{\sum_{\ell \in Z_A} d_\ell + |B|} -  O( n^{1/2} L^{-2}),
\end{equation}

We continue to estimate the terms in \eqref{eq:all1'}
over the next three lemmas.

\begin{lemma} \label{lem:SAo}
$S_{Ao} \ge 2|Z_A|m/n - \wt{O}( n^{1/2 - \eps/5})$.
\end{lemma}
\begin{proof}
If $i$ is not a selection step then by control of open degrees
\[ N_{iAo} \ge 2 
\sum_{ \ell \in Z_A} ( X_{x_\ell} - d_\ell  - k -z ) 
\ge 2 |Z_A| \hat{q}n - \tilde{O} ( \hat{q} n^{1 -\eps/5}). \]
As $S_{Ao} =  \sum_{i=1}^m N_{iAo}q^{-1}$ the lemma follows.
\end{proof}

For $ N_{iK}$ we will require more precise 
estimates for the contribution
from open pairs with one vertex in the smaller holes, 
and so we need to account for this contribution 
further into the process. Accordingly, we define
the following thresholds for hole sizes. We write
\[h^* = h^*(i) =  \min\{ n^{2/5}, L^{-50}\hat{q}\sqrt{n}\},\]
and let $\ell^* = \ell^*(i) \in [z+1]$ be such that 
$h_\ell \ge h^*$ for $1 \le \ell < \ell^*$ 
and $h_\ell < h^*$ for $\ell^* \le \ell \le z$.  

We also let $z'$ be such that $h_\ell \ge n^{2/5}$ for $\ell \le z'$ 
and $h_\ell < n^{2/5}$ otherwise. Thus $\ell^* \ge z'$ 
and equality holds at the beginning of the process.  
By Lemma \ref{large-degs}(ii) we have
\begin{equation} \label{eq:z'}
z' < 4L^{15} k/n^{2/5} = \wt{O}(n^{1/10}).
\end{equation}

We let $J_1 = J_1(i) = \cup_{\ell \le \ell^*} H_\ell$ 
and $J_2 = J_2(i) = \cup_{\ell > \ell^*} H_\ell$;
thus $(J_1,J_2)$ is a partition of $A \cup B$.

We write $N_{iK} \ge 
\sum_{\ell =1}^{z'} N_{iKH_\ell} + N_{iKJ_2}+ N_{iKC}$, 
where each $N_{iKX}$ counts ordered open pairs
counted by $N_{iK}$ with first vertex in $X$.

\begin{lemma}
\label{lem:NiK}
If $i$ is not a selection step then $N_{iK} \ge 
\sum_{\ell =1}^{z'} N_{iKH_\ell} + N_{iKJ_2}+ N_{iKC}$, where
\begin{enumerate}[(i)]
\item $N_{iKX} \ge \hat{q}k|X|$ for $X \in \{J_2,C\}$, and
\item $N_{iKH_\ell} > (1-L^{-5}) \hat{q} h_\ell k/2$
if $\ell \le z'$ and $\hat{q} \ge n^{-1/6}$.
\end{enumerate}
\end{lemma}

\begin{proof}
We write $N_{iKJ_2} = Q'_{J_2} + Q_{J_1J_2} + Q_{J_2 C}$,
where $Q'_{J_2}$ counts ordered open pairs in $J_2$ 
that are not contained within any hole.
To estimate $Q_{J_2}$ we note that
any vertex has degree at most $h^*$ in $J_2$
by the hole construction procedure.
By Lemma \ref{opens}(i) whp $Q_{J_2} = (1 \pm L^{-5}) \hat{q}|J_2|^2$ 
if $\hat{q}|J_2| \ge L^{20} h^*$, so we can write 
$Q_{J_2} \ge (1 - L^{-5}) \hat{q} |J_2| (|J_2| - L^{-30}\sqrt{n})$. 
Then \[Q'_{J_2} \ge Q_{J_2} - h^*|J_2| 
\ge (1-L^{-5}) \hat{q} |J_2| (|J_2| - 2L^{-30}\sqrt{n}).\]
For the second term we consider $Q_{J_1 J_2} \ge Q_{J_1 J'_2}$
where $ J_2' = J_2 \sm N(T)$ and $T$ is the set of vertices 
with at least $L^{20} h^*$ neighbours in $J_1$. 
We can assume $|T| < 4L^{-5}|J_1|/h^* < 6L^{-4}\sqrt{n}/h^*$ 
by Lemma~\ref{large-degs}, so $|N(T) \cap J_2| < 6L^{-4}\sqrt{n}$. 
We apply Lemma \ref{opens-bip} with $R=J_1$ 
and $S = J_2' = J_2 \sm N(T)$, noting that 
if a vertex $x$ has a neighbour in $S$ then $x \notin T$, 
so $x$ has at most $L^{20} h^*$ neighbours in $J_1$.
If $\hat{q}\min\{|J_1|,|J_2'|\} \ge L^{40} h^*$ this gives
whp $Q_{J_1J_2'} = (1 \pm L^{-5}) \hat{q}|J_1||J_2'|$, 
so as $h^* \le L^{-50}\hat{q}\sqrt{n}$ we have
\[Q_{J_1J_2'} \ge (1 - L^{-5}) \hat{q} 
(|J_1| - L^{-4}\sqrt{n}) (|J_2| - 7L^{-4}\sqrt{n}).\]
We can apply the same argument to estimate
$Q_{J_2 C} \ge Q_{J_2 C'}$ where $C' = C \sm N(T')$
and $T'$ is the set of vertices with 
at least $L^{20+2\alpha}$ neighbours in $J_2$. We can assume 
$|T'| < 4L^{-5-2\alpha}|J_2| < 6L^{-4-2\alpha}\sqrt{n}$ 
by Lemma~\ref{large-degs}, so $|N(T') \cap C| < 6L^{-4}\sqrt{n}$
as any vertex has at most $L^{2\alpha}$ neighbours in $C$.
Applying Lemma \ref{opens-bip} with $R=J_2$ and $S = C' = C \sm N(T')$, 
whp $Q_{J_2 C'} = (1 \pm L^{-5}) \hat{q}|J_2||C'|$ if 
$\hat{q}\min\{|J_2|,|C'|\} \ge L^{40+2\alpha}$, so we can write
$Q_{J_2 C'} \ge (1 - L^{-5}) \hat{q}
 (|J_2| - L^{-4}\sqrt{n}) (|C| - 7L^{-4}\sqrt{n})$.
In total, as $|J_1|+|J_2|+|C|=k$
and $\hat{q} k L^{-4} \sqrt{n} = O( L^{-3} \hat{q}n )$
we obtain
\[N_{iKJ_2} \ge Q'_{J_2} + Q_{J_1J_2'} + Q_{J_2 C'} 
\ge \hat{q}k|J_2| - O( L^{-3} \hat{q}n ).\]

We now turn to $ N_{iKC} \ge Q_C + Q_{A \cup B,C}$.
As any vertex has at most $L^{2\alpha}$ neighbours in $C$,
by Lemma \ref{opens}(i) whp 
$Q_C \ge (1 - L^{-5}) \hat{q} |C| (|C| - L^{-4} \sqrt{n})$.
Next we estimate $Q_{A \cup B,C} \ge Q_{A \cup B,C''}$ 
where $ C'' = C \sm N(T'')$ and $T''$ is the set of vertices 
with at least $L^{20 + 2 \alpha}$ neighbours in $A \cup B$.  
As in the argument for $ Q_{ J_2 C'}$, we have
$Q_{A \cup B,  C''} = (1 \pm L^{-5}) \hat{q}|A \cup B||C''|$ 
if $\hat{q}\min\{|A \cup B|,|C''|\} \ge L^{40+2\alpha}$,  so 
\[N_{iKC} \ge Q_{A \cup B, C''} + Q_C \ge \hat{q}k|C| - O(L^{-3}\hat{q}n).\]

This completes the proof of (i). For (ii) we need to estimate
$N_{iKH_\ell}$ when $\hat{q} \ge n^{-1/6}$ 
and $\ell \le z'$ (i.e.\ $h_\ell \ge n^{2/5}$).
We write $X = \{ \ell' \ne \ell: h_{\ell'} \ge 2n^{1/4} \}$ and
$N_{iKH_\ell} = \sum_{\ell' \in X} Q_{H_\ell H_{\ell'} }
 + Q_{H_\ell K'}$, where $K' = K \sm \bigcup_{\ell' \in X} H_{\ell'}$.
We first apply Lemma~\ref{opens-bip} for each $\ell' \in X$ 
to $R = H_\ell \sm N(x_{\ell'})$ 
and $S = H_{\ell'} \sm N(x_\ell)$. 
This is valid by the codegree bound, 
which implies $|R|, |S| \ge n^{1/4}$ and also that
any vertex with a neighbour in one of $R$ or $S$ has at most 
$L^4 < L^{-20} \hat{q} (2n^{1/4})$ neighbours in the other,
as $\hat{q} \ge n^{-1/6}$.
Thus $Q_{H_\ell H_{\ell'}} = (1 \pm L^{-5}) \hat{q}h_\ell h_{\ell'}$.

Now we estimate $Q_{H_\ell K'} \ge Q_{RK'}$
where $R = H_\ell \sm N(U)$ and $U$ is the
set of $x \ne x_\ell$ 
with at least $n^{1/5}$ neighbours in $K$.
We have $|U| < 8L^{16}n^{3/10}$
by Lemma \ref{large-degs}(ii),
so $ | N(U) \cap H_\ell | < L^{21} n^{3/10}$
by the codegree bound.
Next we note  that if a vertex $x$ has a neighbour in $K'$
then $x \ne x_\ell$ by the hole construction procedure, 
so by the codegree bound $x$ has 
at most $L^4 < n^{1/5}$ neighbours in $R \sub H_\ell$.
On the other hand, if $x$ has a neighbour in $R$
then $x \notin U$, so $x$ has at most $n^{1/5}$
neighbours in $K' \sub K$. By Lemma~\ref{opens-bip}, 
as $\hat{q} \ge n^{-1/6}$ we have
$Q_{H_\ell K'} \ge  (1 - L^{-5}) \hat{q}
 ( h_\ell -  L^{21} n^{3/10}) (|K'| - n^{2/5})$. 
As $h_\ell \le d_\ell < (1-\eps)k/2$
we have $k-h_\ell - n^{2/5}>k/2$,
and (ii) follows.
\end{proof}

\begin{lemma}
\label{lem:SO}
The overcount at step $i$ is
$N_{iO} = O( L^{-3} \hat{q} n )$,
so $S_O = \sum_i N_{iO} q^{-1} = O(L^{-2} n^{-1/2})$.
\end{lemma}

\begin{proof}
Let us consider the possible pairwise overcounting
between $N_{iAo}$, $N_{iAi}$, $N_{iBo}$, $N_{iBi}$ and $N_{iK}$. 
Note that by excluding $y \in K \cup \{ x_1, \dots, x_z\}$
in the definition of $ N_{iAo}$ we ensured that it does not 
intersect any of the other collections.
There is no overcounting between $N_{iBo}$ and $N_{iBi}+N_{iK}$,
as pairs counted by the former do not intersect $K$ 
while pairs counted by the latter do intersect $K$.  
There is no overcounting between $ N_{iBi} $ and $ N_{iAi} $,
as the hole construction procedure ensures that no vertex 
in a hole $ H_\ell $ with $ \ell \in Z_B$ is also a neighbour 
of some vertex $x_{\ell'}$ such that $\ell' \in Z_A$. 
It remains to consider the following 
possible overcounting of pairs:

(i) $N_{iAi}$ with $N_{iK}$,
(ii) $ N_{iAi}$ with $ N_{iBo}$, 
(iii)  $ N_{iK} $ with $N_{iBi}$. 

For (i), we note that a pair counted by $N_{iAi}$ and $N_{iK}$
has the form $yy'$ where $y,y'$ are both neighbours
of some $x_\ell$ with $\ell \in Z_A$, and are both in $K$ 
but not in the same hole. By the hole construction 
procedure at least one is also adjacent
to some other $x_{\ell'}$, so by the codegree bound
there are $\wt{O}(k) = \wt{O}( n^{1/2} )$ such pairs. 
For (ii), the overcount between $N_{iAi}$ and $N_{iBo}$ 
is determined by  naming a vertex $b \in B$, 
a vertex $x_\ell$ such that $ \ell \in Z_A$, 
and a vertex $c$ that is in the (final) 
common neighbourhood of $x_\ell$ and $b$;
this overcount is at most
$k|Z_A|L^4 = \wt{O}( n^{1/2})$. 

To bound the most significant overcount (iii), 
namely that between $ N_{iK} $ and $N_{iBi}$,
we introduce the following definition. 
We say that a hole $H_\ell$ with $\ell \in Z_B$ is {\em black} 
if $x_\ell$ has more than $L^{30} h_\ell$ neighbours in $K$.
We let $XH$ be the set of such $x_\ell$ and 
$BH$ be the set of vertices that belong to black holes.
By Lemma~\ref{density}(ii) applied to $S = K \cup XH$ we have
$L^{15}|S| > \eta_S \ge \sum_{x_\ell \in XH} L^{30} h_\ell
= L^{30}|BH|$, so $|BH| \le L^{-14} k$. 
The contribution to $N_{iBi}$ of pairs 
that would close pairs $x_\ell v_{\ell j}$ with $v_{\ell j} \in BH$
is at most $3y|BH| \le 3L^{-14} yk \le 3L^{-13} \hat{q} k n^{1/2}$.
  
Now consider overcounted pairs that would close pairs 
that are not incident to black holes.  
Such a pair has the form $v_{\ell j} v_{\ell' j'}$ 
where $x_\ell v_{\ell'j'}$ is an edge, 
so $\ell'<\ell$ by the hole construction procedure.  
It suffices to show for any fixed $ x_\ell$ 
that at most $L^{-10}h_\ell \hat{q} \sqrt{n}$
such pairs are also counted by $N_{iBi}$.
Suppose first that $ h_\ell \ge n^{2/5}$, so that 
$\ell' < \ell \le z' = \wt{O}(n^{1/10})$ by \eqref{eq:z'}.
By the codegree bound there are at most $z'\cdot L^4 < n^{1/5}$ 
such edges $ x_\ell v_{\ell'j'}$, which are only counted 
in our estimate for $ N_{iK} $ in Lemma \ref{lem:NiK}
while $ \hat{q} > n^{-1/6} $, so the overcount for such a 
hole is at most $ h_\ell n^{1/5} < h_\ell \hat{q} n^{2/5} $. 
Now suppose $ h_\ell < n^{2/5} $. We recall that open pairs 
between $ H_\ell$ and $ H_{\ell'}$ are 
only counted in our estimate for $N_{iK}$ 
in Lemma \ref{lem:NiK} if $H_\ell \sub J_2$,
i.e.\ if $h_\ell < h^* \le L^{-50}\hat{q}\sqrt{n}$.  
Since $H_\ell$ is not black, 
the number of choices for $v_{\ell' j'}$ is at most 
$L^{30} h_\ell < L^{-10}\hat{q}\sqrt{n}$,
so such pairs contribute at most 
$L^{-10}h_\ell \hat{q} \sqrt{n}$.
Summing over all holes gives the desired bound.
\end{proof}

We are now ready for the union bound bound calculation 
that bounds $\mb{P} ( \mc{E})$.
Recall that we have fixed the initial data 
that defines the event $ \mc{E}$; that is, 
we have specified $z$, the vertices $ x_1, x_2, \dots, x_z $, 
the hole sizes $ h_1, \dots, h_z $ and
the degrees $d_\ell$ of vertices $ x_\ell $ for $ \ell \in Z_A $.
We then partition $\mc{E}$ into events $\mc{E}_K$ as analysed above,
defined by choices of neighbourhoods of $ x_\ell $ for $ \ell \in Z_A $,
vertices in $A \cup B$ (which are named by specifying
the vertices in holes), selection steps $i_{\ell j}$,
and vertices in $C$. The number of choices for the data 
that defines $\mc{E}_K$ is at most
\[ \left( \prod_{\ell \in Z_A} \binom{n}{ d_\ell} \binom{ d_\ell}{ h_\ell} m^{d_\ell} \right)
\left( \prod_{ \ell \in Z_B} \binom{n}{ h_\ell} m^{h_\ell} \right) \binom{n}{ |C|}. \]
To estimate $\mb{P}( \mc{E})$ we apply (\ref{eq:all1'}) 
to each such choice of $\mc{E}_K$, substituting
$S_O = O(L^{-2} n^{-1/2})$ from Lemma \ref{lem:SO} and 
$S_{Ao} \ge 2|Z_A|m/n - \wt{O}( n^{1/2 - \eps/5})$
from Lemma \ref{lem:SAo} (the latter acounts for
the $\exp(-2m/n)$ term in the calculation below).
Recalling $|B| = \sum_{\ell \in Z_B} h_\ell$
and $d_\ell = 2m/n \pm n^{1/2 - \eps/3} $, using 
$\binom{ d_\ell}{ h_\ell} < \exp\{ O( \log \log n ) h_\ell \} $ 
for $\ell \in Z_A$ and
$\log \tbinom{n}{|C|} < |C| \log n/2 + O( \log \log n)k$,
we have
\begin{align} \label{eq:PE}
\mb{P}( \mc{E})
& \le \prod_{ \ell \in Z_A} 
\left[ \Big( \tfrac{ ne}{ d_\ell} \cdot \tfrac{2 m}{ n^2}  \Big)^{d_\ell}
\exp \left\{ -2m/n + O( \log \log n ) h_\ell \right\} \right] 
\nonumber \\
& \hskip1cm \cdot \left( \prod_{ \ell \in Z_B}
\left( \tfrac{ ne}{ h_\ell} \cdot \tfrac{2 m}{ n^2}  \right)^{h_\ell}  \right) \binom{n}{|C|}
e^{ - S_K + O( L^{-2} n^{1/2}) }
\nonumber \\
& \le \exp \left\{ \sum_{\ell \in Z_B} h_\ell \log( \sqrt{n}/h_\ell) + |C| \log n/2 - S_K
+ O( \log \log n)k \right\}
\end{align}

It remains to show that $S_K$ is sufficiently large
to make the above probability expression small enough
for the union bound over the initial data defining $\mc{E}$.
We first note for $Z_A$ that the counting terms
$\Big( \tfrac{ ne}{ d_\ell} \cdot \tfrac{2 m}{ n^2}  
\Big)^{d_\ell} = (e \pm O(n^{-\eps/5}))^{d_\ell}$
are cancelled to highest order by the probability term
$\exp(-2m/n)$ from Lemma \ref{lem:SAo},
so we require $S_K$ to dominate the counting terms
from the choice of $B$ and $C$. For $B$ we consider the 
contributions from each hole as follows.
  
The contributions corresponding to the hole $ H_\ell $ 
depends on time when the hole moves out of the set $J_1$
defined before Lemma \ref{lem:NiK}.
If $ h_\ell \ge n^{2/5}$ (i.e.\ $\ell \le z'$) 
we obtain a term $\hat{q}kh_\ell/2$ in the bound from Lemma~\ref{lem:NiK}  while $\hat{q} > n^{-1/6}$, 
i.e.\ up to time $\tfrac{1}{2} \sqrt{\tfrac{1}{6}\log n}$. 
If $h_\ell < n^{2/5}$ we obtain a term $\hat{q}kh_\ell$ 
from Lemma~\ref{lem:NiK} while $\hat{q} > L^{50} h_\ell/\sqrt{n}$, 
i.e.\ up to time 
$ t_\ell =  \tfrac{1}{2} \sqrt{\log \tfrac{\sqrt{n}}{L^{50} h_\ell}}$ 
if this time is less than $ i_{max}$ and up to time $i_{max}$ otherwise. 
Let $ z''$ be the smallest index $\ell$ such that 
$ t_\ell < t_{max} $ (this corresponds to a threshold for 
hole sizes that is about $L^{-50} n^\eps$).
As $S_K = \sum_i N_{iK} q^{-1}$, we have
\begin{align} \label{eq:SK}
S_K & \ge  |C|  \frac{m k }{n^2}
+  \left( \sum_{ \ell =1}^{z'} \tfrac{  h_\ell}{2}  \right) 
 n^{3/2} \cdot \tfrac{1}{2} \sqrt{\tfrac{1}{6}\log n} 
 \cdot \frac{k}{n^2}
+ \left(  \sum_{ \ell = z'+1}^{z''} h_\ell  
\cdot  n^{3/2} \cdot \tfrac{1}{2}
\sqrt{\log \tfrac{\sqrt{n}}{ L^{50} h_\ell}} \right)  \cdot \frac{k}{n^2}
\nonumber \\
& \hskip1cm + \left( \sum_{\ell = z''}^z h_\ell \right) \frac{ m k}{ n^2}
-O( L^{-2} n^{1/2}). 
\end{align}

Finally we substitute \eqref{eq:SK} in \eqref{eq:PE},
grouping terms according to the contribution of each $h_\ell$,
organised into the same summation ranges as in \eqref{eq:SK}.
For each hole $H_\ell$ with $\ell \in Z_B$
included in one of these ranges we have a counting term
$\log \left( \tfrac{ ne}{ h_\ell} \cdot \tfrac{2 m}{ n^2}  \right)^{h_\ell}
= h_\ell (\log \tfrac{\sqrt{n}}{h_\ell} + O(\log\log n))$
from \eqref{eq:PE} which we pair with a probability term from \eqref{eq:SK}.
In the calculations below we also use
(i) $\log \tfrac{\sqrt{n}}{h_\ell} \le \tfrac{1}{10}\log n$
for $\ell \le z'$,
(ii)  $\sqrt{ ( \tfrac{1}{2} \log n ) \cdot \log( \tfrac{\sqrt{n}}{ L^{50} h_\ell} }) > \log \tfrac{\sqrt{n}}{h_\ell}$ for $ z' < \ell \le z''$,
and (iii) $mk/n^2 > (1 + \eps) \tfrac{1}{2} \log n $,
which holds (for small $\eps$)
as $k = (1+3\eps) \sqrt{2n\log n}$ 
and $ m = \sqrt{ (1/2 - \eps) \log n} \cdot n^{3/2}/2$. 
We have
\begin{equation*}
\begin{split}
\log \mb{P}( \mc{E})
& \le
- \sum_{ \ell=1}^{z'}  h_\ell  
\left( \tfrac{1}{4 \sqrt{3}} - \tfrac{1}{10} \right) \log n
- \sum_{ \ell = z' +1}^{z''} 3 \eps h_\ell  \log \tfrac{\sqrt{n}}{h_\ell}  \\
&  \hskip1cm - \sum_{ \ell = z''+	1 }^{z}  \eps  h_\ell \tfrac{1}{2} \log n 
- \eps |C| \tfrac{1}{2}\log n + O( \log \log n)k \\
& \le - \tfrac{\eps}{4} k \log n  + O( \log \log n) k.
\end{split}
\end{equation*}
As the number of choices of the initial data 
that defines $\mc{E}$ is $ O( n^{2z})$
and $ z \le 4k L^{15 - 2\alpha} $, 
where $\alpha$ is large, 
the probability that any such event $\mc{E}$
holds is $o(1)$, which completes the proof.
\qed

\subsection{Proof of the upper bound in Theorem~\ref{edges}.}
\label{sec:upper}

This proof is very similar to that of Theorem~\ref{indep}, but much simpler.
 The lower bound on degrees in $G$ follows from Theorem \ref{good}, 
 so it remains to show the upper bound. 
 We take a union bound over every vertex $x$, 
 potential neighbourhood $A$, and set $C$ such that
$$ |C| = 5\eps \sqrt{ n \log n}$$
of the event that
\begin{enumerate}
\item $A$ is the neighbourhood of $x$ in $ G(i_{max})$,
\item $ A \cup C$ spans no edge in $ G(i_{max}) $, and
\item $ vx$ is open in $ G(i_{max})$ for all $ v \in C$.
\end{enumerate}
We view $C$ as vertices that might be added to the neighbourhood of $v$
 between time $t_{max}$ and the end of the process.  
 We show that whp there is no triple $(x,A,C)$ with these properties.

We fix $x,A,C$, write $ A = \{ v_1, \dots, v_{d'} \}$ for some $d'$
and specify the appearance time $i_j$ for every edge $ xv_j$,
where  $ j < j'$ implies $ i_j < i_{j'}$. 
As in \eqref{eq:recalldegrees}, $I < i_{max}$ implies
\[ i_j = jn/2 \pm n^{3/2 - \eps/3} 
\quad \text{ and } \quad d' = d \pm n^{1/2 - \eps/3},\]
where we recall $d = 2t_{max}\sqrt{n} = 2m/n 
= \sqrt{(1/2 - \eps)n\log n}$.

Let $ \mc{F}$ be the event that
$A \cup C$ is an independent set in $ G(i_{max}) $, 
all pairs joining $x$ and $C$ are open in $ G(i_{max})$,
and all the specified edges appear at the specified steps of the process.  
To estimate the probability of the event $ \mc{F} $, 
for each step $i$ we need to estimate the probability 
that the selected edge is compatible with this event, 
conditional on the history of the process.  
We say $i$ is a {\em selection step} 
if $i$ is one of $i_{j}$ for $j \in [d']$; 
then the selected edge is specified by $\mc{F}$, 
so the required probability 
is simply $2/Q = (1 \pm 2\dD_Q)2q^{-1}$. 
For other $i$, the required probability is $1 - N_i/Q$, 
where $N_i$ is the number of ordered open pairs 
that cannot be selected at step $i$ when $\mc{F}$ occurs. 
If $i=i_{j}$ is a selection step write $N_i=0$.  Then we estimate
\[ \mb{P}(\mc{F}) \le \prod_{j=1}^{d'} (1 \pm 2\dD_Q)2q(t_{ j})^{-1} 
\cdot \prod_{i=1}^m (1 - N_i/Q), \]
where $ t_j = i_j/n^{3/2} $.
We write $ N_i = N_{iA} + N_{iC} $, 
where $ N_{iA} $ counts the ordered open pairs within $A$
and $ N_{iC} $ counts those in $ A \cup C$ with at least one vertex in $C$.
We have
\begin{equation}
\label{eq:Fbound}
- \log \mb{P}(\mc{F}) \ge S_{A} - T_A + S_C 
  +  d' \log \frac{n^2}{2}  - O( n^{1/2}) ,
 \end{equation}
where $S_\mu = \sum_{i=1}^m N_{i\mu}q^{-1}$ for $\mu \in \{A,C\}$ and
$T_A = \sum_{j=1}^{d'} 4t_{j}^2$.

Following the argument in the previous section 
for estimating $ S_{Ai}-T_A$, we have
the following estimate on $S_A- T_A$.  
We include a proof here in the interest of
presenting a complete and self-contained proof 
of the upper bound in Theorem~\ref{edges}.

\begin{lemma} \label{dejavu}
$ S_{A} - T_A = \wt{O}(n^{1/2 - \eps/3})$. 
\end{lemma}

\begin{proof}
We first estimate $N_{iA}$ when $i$ is not a selection step. 
Let $S  = S(i) = \{ v_j \in A : i_j > i \} $ and $s = |S|$;
thus $S(i)$ is the set of vertices $y$ in $A$
such that $yx$ is open and is yet to be joined to $x$. 
The number of ordered open pairs 
$v_{ j} v_{ j'}$ with $j > i $, $j' \le i$ 
is $\sum_{v \in s} 2Y_{vx} = (1 \pm \dD_Y) 2ys$. 
Next note that any vertex has at most $L^4$ neighbours in $S$,
by the bound on codegrees in $ G(i_{max})$,
which applies as $I > i_{max}$.
Then by Lemma~\ref{opens}(i) whp $Q_{s} = (1 \pm n^{-\eps/5}) \hat{q}s^2$ 
if $s > n^{1/4}$ and $\hat{q}s \ge n^{2\eps/5} L^{14}$. Since $\hat{q} \ge n^{-1/2+\eps}$ this holds for $s > n^{1/2 - \eps/2}$, so we can write
 $Q_{s} \ge (1 - n^{-\eps/5}) \hat{q} s(s - n^{1/2 - \eps/2})$. Thus
\[ N_{iA} \ge (1 - n^{-\eps/5}) 
\brac{ 2ys + \hat{q} s (s - n^{1/2 - \eps/2}) }
= 2ys + \hat{q} s^2  - \wt{O}( \hat{q} n^{1 - \eps/5} ).\]

Now we estimate $S_A =  \sum_{i=1}^m N_{iA}q^{-1}$,
which we write as $S_A = S_{A1}+S_{A2} 
+ \wt{O}( n^{1/2 - \eps/5} )$ 
according to the contributions 
of the first and second terms 
in the estimate for $N_{iA}$,
and as before we incur a negligible error by 
using this bound even at selection steps. Thus
\begin{align*}
S_{A1} & = \sum_{i=1}^m  2ys q^{-1}
 = \sum_{j=1}^{d'} 
\sum_{i=i_{j-1}}^{i_j-1} 4tn^{-3/2} (d'+1 - j) 
 = \sum_{j=1}^{d'} \sum_{i=1}^{i_j} 4in^{-3} \\
& = \sum_{j=1}^{d'} 2t_j^2 
- \sum_{j=1}^{ d'} 2 t_j n^{-3/2} 
 = \frac{T_A}{2} - \wt{O}( n^{-1} ), \text{ and } \\
S_{A2} & =  \sum_{i=1}^m   \hat{q} s^2 q^{-1}
= \sum_{j=1}^{d'} 
\sum_{i=i_{j-1}}^{i_j-1} n^{-2}(d'+1-j)^2 \\
& \ge  \sum_{j=1}^d (2n)^{-1} j^2 
- \wt{O}( n^{1/2-\eps/3}) 
\ge T_A/2 - \wt{O}( n^{1/2-\eps/3}).
\end{align*}
The lemma follows.
\end{proof}

To estimate $S_C$ we require the crucial claim that
\begin{equation} \label{claim:SC}
|N(u) \cap C| < L^2 n^\eps
\end{equation}
for any vertex $u$. Indeed, if this failed for some $u$
then at time $t_{max}$ we have $Y_{xu} > 2y$.
However, this would contradict our estimate on $Y$-variables.  
(We can assume $xu$ is a non-edge as $x$ is open to $C$,
and we recall that we track $Y_{xu}$ 
whether $xu$ is open or closed.)
Thus the claim holds.
 
While $\hat{q} |C| > L^{15} n^\eps $, 
which as $|C| = \wt{\TT}(\sqrt{n})$
holds up to time $(1+o(1))t_{max}$, 
we can apply Lemmas \ref{opens}(i) and \ref{opens-bip} to obtain 
$Q_C \ge (1-L^{-1}) \hat{q} |C|^2 $ 
and $Q_{AC} \ge (1-L^{-1}) \hat{q}|A||C|$. 
When $i$ is not a selection step this gives
$N_{iC} = 2Q_{AC} + Q_C \ge (1-L^{-1}) \hat{q} (2|A||C|+|C|^2)$, so
\[ S_C  = \sum_{i=1}^m  N_{iC} q^{-1}
> (1 - o(1)) (2 |A| + |C|) |C| m/n^2
= (1-o(1)) \left(1 - 2\eps + 5\eps\sqrt{ \frac{1}{2} - \eps} \right) |C| \tfrac{1}{2} \log n. \]
Now we substitute Lemma \ref{dejavu} in (\ref{eq:Fbound}),
and take the union over all possible choices of the data 
that specifies an event $\mc{F}$, 
namely the choices of $x$, $d'$, $A$, $C$ 
and the collection of times at which the edges joining $x$ to $A$ appear.
Thus we bound the probability $p_0$ that any triple $(x,A,C)$ as above
exists by \[ p_0 < n \sum_{d'}
\binom{n}{d'} \binom{n}{|C| } m^{d'} 
\left( \tfrac{ 2}{ n^2} \right)^{d'}
\exp \left\{ - (1-o(1))
 \left(1 - 2\eps + 5\eps \sqrt{ \frac{1}{2} - \eps}\right) |C| \tfrac{1}{2} \log n + O(n^{1/2}) \right\}.\]
Here we note that the counting term 
$\tbinom{n}{d'} m^{d'} \left( \tfrac{ 2}{ n^2} \right)^{d'}
= \exp [(1+o(1))d]$ is of  lower order
than the main counting term 
$\tbinom{n}{|C|} = \exp [(1+o(1))|C|\tfrac{1}{2} \log n]$,
and this is more than compensated for by the probability term:
assuming $ \eps< 1/4$, we obtain
\[ p_0 < n \sum_{d'} \exp \left\{ - \eps |C| \tfrac{1}{5} \log n  \right\}. \]
Thus the required bound on degrees holds with high probability.
\qed

\section{Concluding remarks} \label{sec:conclude}

We have determined $R(3,t)$ to within a factor of $4+o(1)$, 
so we should perhaps hazard a guess for its asymptotics:\ 
we are tempted to believe the construction rather than the bound, 
i.e.\ that $R(3,t) \sim t^2/4\log t$.  
We only proved an upper bound 
on the independence number of the graph $G$ 
produced by the triangle-free process,
so in principle it might give a better lower bound on $ R(3,t)$.
However, we believe that this is not the case:\
we conjecture that the bound on the independence number 
in Theorem \ref{indep} is asymptotically best possible. 

Another natural direction for future research is to provide an asymptotically optimal analysis in greater generality for the $H$-free process. No doubt the technical challenges will be formidable, given the difficulties that arise in the case of triangles. But on an optimistic note, it is encouraging that one can build on two different proofs of this case.


\end{document}